\documentclass[reqno,twoside,11pt]{amsart}

\usepackage{amsmath,amsfonts,calrsfs,fullpage,amssymb,xcolor,verbatim,eucal,yfonts,mathrsfs}

\usepackage{mathtools}
\mathtoolsset{showonlyrefs}

 \usepackage{xcolor}
\newtheorem{Theorem}{Theorem}[section]

\newtheorem{Lemma}[Theorem]{Lemma}

\newtheorem{Remark}[Theorem]{Remark}

\newtheorem{Hypothesis}{Hypothesis}
\makeatletter
\@addtoreset{equation}{section}

\makeatother

\def\a{\alpha}
\def\le{\left}
\def\r{\right}

\def\R{\mathbb R}

\def\E{\mathbb E}
\def\P{\mathbb P}
\def\F{\mathcal F}

\def\L{\mathcal L}
\def\eps{\varepsilon}
\def\ds{\displaystyle}

\def\e{\epsilon}
\def\vs{\vspace{.1mm}\\}

 \newcommand{\norm}[1]{\left\| #1\right\|}

 \newcommand{\Inner}[1]{\big\langle #1\big\rangle}
 
 \newcommand{\abs}[1]{\lvert #1\rvert}

\title{ Parabolic scaling of a stochastic wave map with co-normal noise: limit and fluctuations}\date{}

\author[S. Cerrai]{Sandra Cerrai}
\address{Department of Mathematics\\
University of Maryland\\
}
\email{cerrai@umd.edu}
\thanks{S. Cerrai was partiallly supported by NSF Grant DMS-2348096 (2024-2027), {\em 
 Multiscale Analysis of Infinite-Dimensional Stochastic Systems}}

\author[M. Xie]{Mengzi Xie}
\address{Institut f\"{u}r Mathematik\\
Technische Universit\"{a}t Berlin\\
}
\email{xie@tu-berlin.de}

\subjclass[2010]{}


\begin{document}

 \begin{abstract}  
This paper investigates the parabolic rescaling limit of a damped stochastic wave map from the real line into the two-dimensional sphere, perturbed by multiplicative Gaussian noise of co-normal type. We prove that  under this rescaling, the solutions converge to those of the deterministic heat flow for harmonic maps, revealing a transition from stochastic hyperbolic to deterministic parabolic dynamics. We further analyze the fluctuations around this limit, proving a  central limit theorem and identifying the limiting process as the solution to a linear stochastic partial differential equation. The study combines tools from geometric analysis, stochastic calculus, and functional analysis, offering insights into the interplay between geometry, noise, and scaling in nonlinear stochastic systems.
 \end{abstract}

 \maketitle
 
 \tableofcontents

\section{Introduction}

Wave maps are critical objects in mathematical physics and geometric analysis, describing energy-minimizing maps from a spacetime manifold into a target Riemannian manifold. In the deterministic setting, they generalize harmonic maps and obey a second-order hyperbolic PDE. Stochastic wave maps introduce randomness into this framework, modeling effects such as thermal noise or turbulent forcing. These equations, who occupy a particularly interesting intersection of deterministic PDE theory, stochastic analysis, and differential geometry, are notoriously difficult due to their combination of nonlinearity, geometry, and stochasticity.

In this context, the present paper investigates a damped stochastic wave map defined on the whole real line with values in the unit sphere $\mathbb{S}^2$, under the influence of spatially homogeneous Gaussian noise that is white in time and regularized in space. We focus on the parabolic rescaling limit, exploring both law of large numbers and central limit behaviors, thereby capturing the deterministic and stochastic asymptotics of the model under small-parameter limits.

We start from the following stochastic damped wave equation
\begin{equation}\label{SPDE-intro}
	\le\{\begin{array}{l}
	\ds{\partial_{t}^{2}u(t,x)+\abs{\partial_{t}u(t,x)}^{2}u(t,x)=\partial_x^2 u(t,x)+\abs{\partial_x u(t,x)}^{2}u(t,x)-\gamma\partial_{t}u(t,x)}\\[10pt]
	\ds{\quad \quad \quad \quad \quad \quad \quad \quad \quad \quad \quad \quad \quad \quad  +\big(u(t,x)\times\partial_{t}u(t,x)\big)\circ\partial_{t}w(t,x), }\\
	[10pt]
	\ds{u(0,x)=u_0(x),\ \ \ \ \partial_{t}u(0,x)=v_0(x) },
	\end{array}\r.
\end{equation}
in  dimension $1+1$, whose solution takes value in the unitary sphere $\mathbb{S}^2$. Here $\gamma$ is a positive constant friction coefficient and  the initial condition $(u_0,v_0)$ is taken in $\mathcal{M}$, where 
\begin{equation}
	\mathcal{M} := \big\{ (u,v): \R\mapsto T\mathbb{S}^{2} \big\},\ \ \ \ \ T\mathbb{S}^{2}:=\left\{(h,k)\in\,\mathbb{S}^2\times \mathbb{R}^3\,:\, h\cdot k=0\right\}.
\end{equation}
The noise $w(t),\ t\geq0,$ is a spatially homogeneous Wiener process defined on a complete filtered probability space $\big(\Omega, \F, (\F_{t})_{t\geq 0}, \P \big)$, with finite spectral measure $\mu$ having finite second moment, and the stochastic differential is understood in Stratonovich sense. The well-posedness of equations of this type have been already studied in the existing literature, and to this purpose, we refer to \cite{brz2007} and \cite{brzgoldys2022}.

The key novelty of this paper lies in the exploration of  equation \eqref{SPDE-intro} under a parabolic scaling, which transforms the system into a family of equations parametrized by a small parameter $\e>0$, in which  time is dilated and space is rescaled. Namely, we define
\[u_\e(t,x):=u(t/\e,x/\sqrt{\e}),\ \ \ \ \ \ \ \ (t,x) \in\,[0,+\infty)\times \R,\]
and investigate the asymptotic behavior of $u_\e$, as $\e\downarrow 0$, 
 with a particular interest in the transition from the stochastic hyperbolic regime to a deterministic parabolic limit. It turns out that, after converting the Stratonovich's differential into the It\^o's one,  $u_\e$ satisfies the equation \begin{equation}\label{SPDE3-intro}
	\le\{\begin{array}{l}
		\ds{\eps\partial_{t}^{2}u_{\eps}(t,x)+\eps\abs{\partial_{t}u_{\eps}(t,x)}^{2}u_{\eps}(t,x)=\partial_x^{\,2} u_{\eps}(t,x)+\abs{\partial_x u_{\eps}(t,x)}^{2}u_{\eps}(t,x)-\gamma_0\,\partial_{t}u_{\eps}(t,x)}\\
		[10pt]
\ds{\quad\quad\quad\quad\quad\quad\quad\quad\quad\quad\quad\quad\quad\quad \quad\quad\quad\quad +\sqrt{\eps}\big(u_{\eps}(t)\times\partial_{t}u_{\eps}(t)\big)\partial_{t}w^{\eps}(t,x), }\\
		[10pt]
		\ds{u_{\eps}(0,x)=u^\e_0(x),\ \ \ \ \partial_{t}u_{\eps}(0,x)=v^\e_0(x) },
	\end{array}\r.
\end{equation}
for some initial conditions $u^\e_0$ and $v^\e_0$ depending on $\e$, where the friction $\gamma$ is now enhanced by the new one
\begin{equation}
\label{gamma0-intro}
\gamma_0:=\gamma+\frac 12\, \mu(\R),	
\end{equation}
and the rescaled noise $w^\epsilon(t,x)$  is white in time and inherits a spatial covariance structure adapted to the parabolic scaling. Concerning our choice to work with Stratonovich noise, we emphasize that, as a general principle,  considering Stratonovich formulation in S(P)DEs is essential for preserving the geometric compatibility of the solution space. In the specific case studied here, where the random perturbation involves the co-normal term  $u\times \partial_t u$, it would be possible to introduce an It\^o type noise instead, and still preserve the geometric constraint of remaining on $\mathbb{S}^2$.  However, this alternative would  require imposing a sufficiently large friction coefficient $\gamma$, in order to control the dynamics and study the limiting behavior of $u_\e$, instead of any arbitrary $\gamma>0$ as in the present paper.

\smallskip

Before presenting the main results of this paper, we would like to highlight that similar asymptotic problems for stochastic damped wave equations with constraints have been previously studied in \cite{brzcer} and \cite{cerxie}.  However, both of these works considered equations defined on a bounded domain  $\mathcal{O}\subset \mathbb{R}^d$, where the solutions were constrained to lie on the Hilbert manifold of functions in $L^2(\mathcal{O})$ with norm equal $1$. The analysis of both papers was motivated by the study of the small-mass limit, also known as Smoluchowski-Kramers approximation (see the original paper \cite{smolu}, and \cite{CF1}, where the infinite dimensional case was studied for the first time), and was only related to the study of the limit of $u_\e$.  
We emphasize that the methods developed in \cite{brzcer} and \cite{cerxie} cannot be reduced to a parabolic rescaling and relied on substantially different techniques, due to the fundamentally distinct nature of the imposed constraints.

\smallskip

The first major contribution of the present paper is a rigorous derivation of the deterministic parabolic limit. We prove that, as $\e\downarrow 0$, the solution $u_\e$
 of the rescaled wave equation \eqref{SPDE3-intro} converges in probability to the unique solution of the heat flow harmonic map equation on the sphere
\begin{equation}\label{limiting_equation-intro}
	\le\{\begin{array}{l}
		\ds{ \gamma_0\,\partial_{t}u(t,x) = \partial_{x}^{2} u(t,x)+\abs{\partial_{x} u(t,x)}^{2}u(t,x)  , \ \ \ \ \  \ (t,x)\in \mathbb{R}^{+}\times \mathbb{R},}\\[10pt]
		\ds{u(0,x)=u_{0}(x),\ \ \ \ \ \ x\in \mathbb{R} },
	\end{array}\r.
\end{equation}
where $u_0\in M:=\{u:\mathbb{R}\to \mathbb{S}^{2}\}$. 
 This limiting equation captures the  geometric flow associated with harmonic maps and is a well-known object of study in geometric analysis. The convergence result is a law of large numbers: for suitable initial conditions and under appropriate uniform estimates, the stochastic system \eqref{SPDE3-intro} loses its randomness in the limit and collapses onto the deterministic trajectory governed by the heat flow  harmonic map equation. It is important to stress that the limiting deterministic dynamics keeps memory of the stochasticity of the wave map equation through the modified friction $\gamma_0$ defined in \eqref{gamma0-intro} which involves the total mass of the spectral measure $\mu$ associated with the noise. 
 
 Our analysis extends also to the limiting behavior of the velocity $\partial_t u_\e$. In this case, we show that while $\partial_t u_\e$ converges to $\partial_t u$ in probability with respect to the weak convergence in $L^2(0,T;L^2(\R))$\footnote{\,Here and in what follows, for the sake of simplicity, we will denote $L^2(\mathbb{R};\mathbb{R}^3)$, $\dot{H}^k(\mathbb{R};\mathbb{R}^3)$ and $H^k(\mathbb{R};\mathbb{R}^3)$, $k\geq 1$, by $L^2(\mathbb{R})$, $\dot{H}^k(\mathbb{R})$ and $H^k(\mathbb{R})$, respectively.
}, because of the presence of the noise in equation \eqref{SPDE3-intro} $\partial_t u_\e$ never converges to $\partial_t u$ in probability with respect to the strong  convergence in $L^2(0,T;L^2(\R))$. As we will discuss later, this fact will have fundamental consequences when it comes to the study of the fluctuations and the proof of the validity of the central limit theorem.

\smallskip

The second major contribution addresses the nature of the  stochastic fluctuations around the deterministic limit, in the case the noise results from  the convolution of a fractional noise 
$w_H$ of Hurst index 
$H \in\,(1/2,1)$  with a smoothing kernel $\eta$, ensuring that the resulting noise satisfies the spectral assumptions needed for the analysis described above. Here, we establish a central limit-type theorem for the rescaled difference 
\[y_\epsilon(t,x):=\e^{H/2-1}(u_\e(t,x)-u(t,x)) ,\]
showing the convergence to the solution $\varrho$ of the following linear stochastic PDE, driven by the spatially  fractional noise 
$w_H(t,x)$ and having  coefficients depending on the deterministic heat flow   harmonic map equation
\begin{equation}\label{clt-intro}
		\le\{\begin{array}{l}
			\ds{\gamma_0\,\partial_{t}\varrho(t) = \partial_x^{\,2} \varrho(t)+\abs{\partial_x u(t)}^{2}\varrho(t)+2(\partial_x u(t)\cdot \partial_x \varrho(t))u(t) +\big(u(t)\times\partial_{t}u(t)\big)\partial_{t}w^{H}(t), }\\
			[10pt]
			\ds{\varrho(0)=0.}
		\end{array}\r.
	\end{equation}

This fluctuation analysis is technically demanding due to the geometry of the target manifold and the non-trivial structure of the noise which involves not only the position $u$, but also the velocity $\partial_tu$. The key fundamental point in our analysis is showing that proving the convergence of $y_\e$ to $y$ in $L^2(0,T;L^2(\R))$ - whether in distribution,  in probability,    in mean-square, or  in the weak topology of $L^2(\Omega;L^2(0,T;L^2(\R)))$ -  can be reduced to proving the  analogous convergence of  $z_\e$ to  $z$, where $z_\e$ solves
\[\gamma_0\,\partial_tz_\e(t)=\partial_x^2 z_\e(t)+(u_{\epsilon}(t)\times \partial_t u_\epsilon(t))Q_\epsilon\partial_t w_H(t),\ \ \ \ \ \ z_\e(0)=0,\]
and $z$ solves
\[\gamma_0\,\partial_tz(t)=\partial_x^2 z(t)+(u(t)\times \partial_t u(t))\partial_t w_H(t),\ \ \ \ \ \ z(0)=0.\]
In particular, since we can show that $z_\e$ converges to $z$ in $L^2(\Omega;L^2(0,T;L^2(\mathbb{R})))$,  we obtain  the validity of a strong  central limit theorem.

\smallskip

In addition to the two main  results we have just described, we would like to mention a few other facts that emerge from our analysis and that are related to the limiting behavior of $u_\e$ and $\partial_t u_\e$ and the properties of the solution $u$ of the heat flow  harmonic map equation \eqref{limiting_equation-intro}.

A first one   concerns the convergence of $u_\e$ to $u$. Actually, as a byproduct of the uniform bounds we prove for   the process $y_\e$, we can show that when dealing with the fractional noise $w_H$, the following estimate holds
\begin{equation}\label{conv-rate-intro}
		\E\sup_{t\in[0,T]}\abs{u_{\epsilon}(t)-u(t)}_{L^2(\mathbb{R})}^{2}+\E\int_{0}^{T}\abs{u_{\epsilon}(t)-u(t)}_{H^{1}}^{2}dt\lesssim_{\,\alpha,T}\epsilon^{3/2-H-\alpha}+ \vert u^\e_0-u_0\vert^2_{L^2(\mathbb{R})},
	\end{equation}
for every $\alpha>0$. In particular the convergence in probability of $u_\e$ to $u$ proved in the first part of the paper is improved to mean-square convergence and a rate of convergence is provided. Notice that in this case, we can take $H \in\,[1/2,1)$, so that the case of space-time white noise can be covered.

A second fact concerns the convergence of $u_\epsilon$ and $\partial_t u_\e$ in the deterministic setting. By assuming that $(u_0,v_0) \in\,\big(\dot{H}^1(\R)\times L^2(\R)\big)\,\cap\, \mathcal{M}$ and $(u^\e_0,v^\e_0) \in\,\big(\dot{H}^3(\R)\times H^2(\R)\big)\,\cap\, \mathcal{M}$, with
\begin{equation}\label{initial_condition_rate_additional-intro}
\abs{u_{0}^{\epsilon}-u_{0}}_{L^2(\R)}=o(\epsilon^{1-H/2}),\ \ \ \ \ \ \ \ \ 		\big\lvert (u_{0}^{\epsilon}-u_0,\sqrt{\epsilon}v_{0}^{\epsilon} ) \big\rvert_{H^{1}(\R)\times L^2(\mathbb{R})} =O(\epsilon^{\beta}), \ \ \ \ 0<\epsilon\ll1,
	\end{equation}
	for some $\beta>0$, we show that 
\begin{equation}
		\E\sup_{t\in[0,T]}\Big(\abs{u_{\epsilon}(t)-u(t)}_{H^{1}}^{2}+\epsilon\abs{\partial_{t}u_{\epsilon}(t)-\partial_{t}u(t)}_{L^2(\mathbb{R})}^{2}\Big) +  \E\int_{0}^{T}\abs{\partial_{t}u_{\epsilon}(t)-\partial_{t}u(t)}_{L^2(\mathbb{R})}^{2}dt \lesssim_{\,T} \mu(\R) +\epsilon^{1\wedge 2\beta}.
	\end{equation}
In particular, in the deterministic case, that is when $\mu(\R)=0$, we get the convergence of $(u_\e-u, \partial_t u_\e-\partial_t u)$ to $0$ in $L^\infty(0,T;H^1(\R))\times L^2(0,T;L^2(\R))$. This problem has been already addressed in  \cite{zarnescu}, in the  more delicate situation of space dimension $d>2$. In that case, in order to get the same convergence, which is clearly only local in time, the authors expanded $u_\e(t,x)$ as
\[u_\e(t,x)=u(t,x)+u_0^I(t/\e,x)+\sqrt{\e}\,u^\e_R(t,x),\] where $u_0^I$ is a suitable boundary layer which is given explicitly and $u_R^\epsilon$ is the solution of the following damped wave equation
\[\epsilon\,\partial^2_t u^\e_R(t,x)=\partial_x^2 u^\e_R(t,x)-\partial_t u^\e_R(t,x)+S(u_R^\e)(t,x)+R(u_R^\e)(t,x),\]
for some singular term $S(u_R^\e)$ and regular term $R(u_R^\e)$. However,  the analysis in \cite{zarnescu} is restricted  to the case $d>2$ and initial conditions $(u_0,v_0) \in\,\dot{H}^7(\R)\times H^6(\R)$.

One last fact that we would like to mention concerns with the heat-flow equation \eqref{limiting_equation-intro}.
As a consequence of our law of large numbers, we obtain that for every $u_0 \in\,\dot{H}^1(\R) \cap M$ there exists a solution for  equation \eqref{limiting_equation-intro}
in $L^{\infty}(0,T; \dot{H}^{1}(\R))\cap L^{2}(0,T; \dot{H}^{2}(\R))$. To the best of our knowledge, this result is new. All existing results we found in the literature concern either well-posedness for equations defined on compact manifolds or equations posed on the entire real line, but none provide square integrability of 
$u$ and its derivatives. Actually, starting from its well-posedness  in $L^{\infty}(0,T; \dot{H}^{1}(\R))\cap L^{2}(0,T; \dot{H}^{2}(\R))$, in Appendix \eqref{AppB}, we will study the well-posedness of equation \eqref{limiting_equation-intro}
in $L^{\infty}(0,T; \dot{H}^{k}(\R))\cap L^{2}(0,T; \dot{H}^{k+1}(\R))$, for every $k\geq 1$.

\section{Notations, assumptions and preliminary results}
\label{notations}

Throughout the present paper, we deal with the equation
\begin{equation}\label{SPDE}
	\le\{\begin{array}{l}
	\ds{\partial_{t}^{2}u(t,x)+\abs{\partial_{t}u(t,x)}^{2}u(t,x)=\partial_x^2 u(t,x)+\abs{\partial_x u(t,x)}^{2}u(t,x)-\gamma\partial_{t}u(t,x)}\\[10pt]
	\ds{\quad \quad \quad \quad \quad \quad \quad \quad \quad \quad \quad \quad \quad \quad  +\big(u(t,x)\times\partial_{t}u(t,x)\big)\circ\partial_{t}w(t,x), }\\
	[10pt]
	\ds{u(0,x)=u_0(x),\ \ \ \ \partial_{t}u(0,x)=v_0(x) },
	\end{array}\r.
\end{equation}
for $(t,x)\in \R^{+}\times \R$, taking values in the . Here $\gamma$ is some positive constant and the initial conditions $(u_0,v_0)$ are taken in $\mathcal{M}$, where 
\begin{equation}
	\mathcal{M} := \big\{ (u,v): \R\mapsto T\mathbb{S}^{2} \big\},\ \ \ \ \ T\mathbb{S}^{2}:=\left\{(h,k)\in\,\mathbb{R}^3\times \mathbb{R}^3\,:\,|h|=1,\ h\cdot k=0\right\}.
\end{equation}
Moreover, $w(t),\ t\geq0,$ is a spatially homogeneous Wiener process defined on a complete filtered probability space $\big(\Omega, \F, (\F_{t})_{t\geq 0}, \P \big)$. This means that if $\mathcal{S}(\mathbb{R})$ denotes the Schwartz space of rapidly decreasing functions on $\mathbb{R}$ and   $\mathcal{S}'(\mathbb{R})$ is the dual space of tempered distributions, with the duality denoted by $(\cdot,\cdot)$, then  $w(t)$, $t\geq 0$,  is an $\mathcal{S}'(\mathbb{R})$-valued Wiener process satisfying the following conditions.
\begin{enumerate}
	\item[i.] For each $\psi\in \mathcal{S}(\mathbb{R})$, $(w(t),\psi)$, $t\geq 0$, is a real-valued $\{\F_{t}\}$-adapted Wiener process.
	
	\item[ii.] There exists  a positive-definite $\Gamma \in \mathcal{S}'(\mathbb{R})$ such that for all $\psi$, $\varphi\in \mathcal{S}(\mathbb{R})$ and $s, t\geq 0$ 
	\begin{equation}
		\E\,(w(s),\varphi)\,(w(t),\psi)=(s\wedge t)\,(\Gamma,\varphi\ast\psi_{(s)} ),
	\end{equation}
	where $\psi_{(s)}(x)=\psi(-x)$.
\end{enumerate} 
Since $\Gamma$ is a positive-define distribution, there exists a positive and symmetric tempered measure $\mu$ on $\mathbb{R}$ such that $\Gamma = \mathcal{F}({\mu})$, where $\F$ denotes the Fourier transform. As  shown e.g. in  \cite{PZ},  the reproducing kernel  of $w(t)$, denoted by $K$, is characterized as 
\begin{equation}
	K = \big\{ \mathcal{F}({\psi \mu}) \in\,\mathcal{S}^\prime(\mathbb{R})\ :\  \psi\in L^{2}_{(s)}(\mathbb{R},\mu)   \big\},\end{equation}
	where 
\begin{equation*}
	L^{2}_{(s)}(\mathbb{R},\mu):=\big\{ \psi\in L^{2}(\mathbb{R},d\mu;\mathbb{C})\ : \ \overline{\psi(x)}=\psi(-x),\ \ x\in\mathbb{R}  \big\}.
\end{equation*}
Note that $K$ is a Hilbert space, endowed with the scalar product 
\begin{equation}
	\Inner{\mathcal{F}({\psi_{1}\mu}),\mathcal{F}({\psi_{2}\mu})  }_{K}=\Inner{\psi_{1},\psi_{2}}_{L^{2}(\mathbb{R},d\mu;\mathbb{C})},
\end{equation}
(for a proof, see e.g. \cite{PZ}). 
In particular, if $\{e_{k}\}_{k\in\mathbb{N}}\subset \mathcal{S}_{(s)}(\mathbb{R})$ is an orthonormal basis of $L^{2}_{(s)}(\mathbb{R},\mu)$, and 
\begin{equation}
\label{wm1}	
\xi_{k}:=\mathcal{F}({e_{k}\mu}),\ \ \ \ \ \ k \in\,\mathbb{N},
\end{equation}
then $\{\xi_k\}_{k \in\,\mathbb{N}}$ is an orthonormal basis of $K$.

 Throughout this paper, we shall assume that $\mu(\mathbb{R})<\infty$, so that $\Gamma$ is a bounded uniformly continuous function, and $K$ is contained in the space of real-valued  bounded and uniformly continuous functions on $\mathbb{R}$. In this case, the process $w(t)$ can be represented as a centered $\{\F_{t}\}_{t\geq0}$-adapted Gaussian random field $\{w(t,x)\ :\ t\geq 0, x\in\mathbb{R}\}$ with\begin{equation*}
	(w(t),\psi)= \int_{\mathbb{R}}w(t,x)\psi(x)dx,\ \ \ \ \psi\in\mathcal{S}(\mathbb{R}),
\end{equation*} 
and 
\begin{equation}
	\E\,(w(t,x)w(s,y))=(t\wedge s)\,\Gamma(x-y),\ \ \ \ t,s\geq0,\ \ \ \ x,y\in\mathbb{R}.
\end{equation}
In what follows, we shall assume that $\mu$ satisfies these further assumptions.
\begin{Hypothesis}
\label{H1}
The measure $\mu$ is absolutely continuous with respect to the Lebesgue measure with  density $m=d\mu/dx$. Moreover
\begin{equation}\label{fourth_moment}
	\int_{\mathbb{R}}(1+\abs{x})^{2}d\mu(x)=\int_{\mathbb{R}}(1+\abs{x})^{2}m(x)dx<\infty.
\end{equation}	
\end{Hypothesis}

\subsection{Parabolic rescaling}

If for every $\e>0$ we define 
\begin{equation*}
	u_{\epsilon}(t,x):=u(t/\epsilon,x/\sqrt{\epsilon}),\ \ \ \ \ \ \ (t,x) \in\,\mathbb{R}^+\times \mathbb{R},
\end{equation*}
then it is immediate to check that $u_{\eps}$ satisfies the following equation
\begin{equation}\label{SPDE1}
	\le\{\begin{array}{l}
		\ds{\epsilon\,\partial_{t}^{2}u_{\epsilon}(t,x)+\epsilon\abs{\partial_{t}u_{\epsilon}(t,x)}^{2}u_{\epsilon}(t,x)=\partial_x^{\,2} u_{\epsilon}(t,x)+\abs{\partial_x u_{\epsilon}(t,x)}^{2}u_{\epsilon}(t,x) -\gamma\,\partial_{t}u_{\eps}(t,x)  }\\[10pt]
		\ds{\quad\quad\quad\quad\quad\quad\quad\quad\quad\quad\quad\quad +\sqrt{\epsilon}\big(u_{\epsilon}(t)\times\partial_{t}u_{\epsilon}(t)\big)\circ\partial_{t}w^{\epsilon}(t,x), }\\
		[10pt]
		\ds{u_{\epsilon}(0,x)=u_{0}^{\epsilon}(x)\in \mathbb{S}^{2},\ \ \ \ \ \ \ \partial_{t}u_{\epsilon}(0,x)=v_{0}^{\epsilon}(x) },
	\end{array}\r.
\end{equation}
where for this rescaled system  the initial conditions $(u^\e_0,v^\e_0)$ depend on $\eps$. 
Throughout the paper, we assume that $(u^\e_0,v^\e_0)$ satisfy the following assumption.
\begin{Hypothesis}
	\label{H1-bis} For every $\e \in\,(0,1)$, we have $(u_0^{\epsilon},v_0^{\epsilon})\in \big(\dot{H}^{2}(\mathbb{R})\times H^{1}(\mathbb{R})   \big)\cap \mathcal{M}$, and the following conditions hold
\begin{equation}
	\label{initial_condition}
	\Lambda_{1}:=\sup_{\epsilon\in(0,1)} \big\lvert (u_{0}^{\epsilon},\sqrt{\epsilon}v_{0}^{\epsilon})\big\rvert_{\dot{H}^{1}(\R)\times L^{2}(\R)}<\infty,\ \ \ \ \Lambda_{2}:=\sup_{\epsilon\in(0,1)}\sqrt{\epsilon}\big\lvert (u_{0}^{\epsilon},\sqrt{\epsilon}v_{0}^{\epsilon})\big\rvert_{\dot{H}^{2}(\R)\times H^{1}(\R)}<\infty.
\end{equation}

\end{Hypothesis}
Note that, as we mentioned already in the Introduction,  throughout the whole paper we will denote $L^2(\mathbb{R};\mathbb{R}^3)$, $\dot{H}^k(\mathbb{R};\mathbb{R}^3)$ and $H^k(\mathbb{R};\mathbb{R}^3)$, $k\geq 1$, by $L^2(\mathbb{R})$, $\dot{H}^k(\mathbb{R})$ and $H^k(\mathbb{R})$, respectively.

As for the noise, for every $\epsilon >0$ we have that $w^{\epsilon}(t)$ is a spatially homogeneous Wiener process with  covariance 
\begin{equation}\label{sm150}
	\E\big[w^{\epsilon}(t,x)w^{\epsilon}(s,y)\big] =  (t\wedge s)\Gamma\big((x-y)/\sqrt{\epsilon}\big)=:(t\wedge s)\Gamma^{\epsilon}(x-y),\ \ \ \ t,s\geq0,\ \ \ \ x,y\in\mathbb{R}.
\end{equation}
If we define 
\[m^{\epsilon}(x):=\sqrt{\epsilon}m(\sqrt{\epsilon}x),\ \ \ \ \ \  d\mu^{\epsilon}(x):=m^{\epsilon}(x)dx,\ \ \ \ x\in\mathbb{R},\]  it can be immediately verified that $\Gamma^{\epsilon}=\mathcal{F}({\mu^{\epsilon}})$, which means that $\mu^{\epsilon}$ is the spectral measure of $w^{\epsilon}(t)$. The reproducing kernel Hilbert space of $w^{\epsilon}(t)$, denoted by $K_{\epsilon}$, can be represented as 
\begin{equation}
	K_{\epsilon} = \big\{ \mathcal{F}({\psi \mu^{\eps}}): \psi\in L^{2}_{(s)}(\mathbb{R},\mu^{\epsilon})   \big\}.
\end{equation}

\begin{Lemma}
	For every $\eps>0$, the family of real-valued $C^{1}$ functions 
	\begin{equation*}
		\xi_{k}^{\eps}(x):=\xi_{k}\big(x/\sqrt{\eps}\big),\ \ \ \ x\in\mathbb{R},\ \ \ \ k\in\mathbb{N},
	\end{equation*}
	forms a complete orthonormal system for $K_{\eps}$. Moreover, for every $\eps>0$ and $x \in\,\mathbb{R}$ we have 
	\begin{equation}
		\sum_{k=1}^{\infty}\abs{\xi_{k}^{\eps}(x) }^{2} = \mu(\mathbb{R})=:c_{0},\ \ \ \ \ \ \ \ \sum_{k=1}^{\infty}\big\lvert (\xi^{\eps}_{k})^\prime(x) \big\rvert^{2} =\frac 1{\epsilon}\,\int_{\mathbb{R}}\abs{z}^{2}d\mu(z)=: \frac{c_{1}}{\eps}.
	\end{equation}
	
\end{Lemma}

\begin{proof}
Since the family  $e_{k}^{\epsilon}(x):=e_{k}(\sqrt{\epsilon}x)$, $x\in\R$, $k\in\mathbb{N}$, forms a complete orthonormal system of $L^{2}_{(s)}(\mathbb{R},\mu^{\epsilon})$, we have $\mathcal{F}({e_{k}^{\epsilon}\mu^{\epsilon} })$, $k \in\,\mathbb{N}$, forms a complete orthonormal basis for $K_{\epsilon}$, and
\begin{equation}
	\mathcal{F}({e_{k}^{\epsilon}\mu^{\epsilon} })(x)=\int_{\mathbb{R}}e^{- i xz}e_{k}(\sqrt{\epsilon}z)\sqrt{\epsilon}m(\sqrt{\epsilon}z)dz = \mathcal{F}({e_{k}\mu})\big(x/\sqrt{\eps}\big)=\xi_{k}\big(x/\sqrt{\eps}\big),\ \ \ \ x\in\mathbb{R},\ \ \ \ k\in\mathbb{N}.
\end{equation}
For any $\eps>0$, we have
\begin{equation}\label{identity1}
	\begin{array}{ll}
		\ds{ \sum_{k=1}^{\infty}\abs{\xi_{k}^{\epsilon}(x)}^{2} =\sum_{k=1}^{\infty}\abs{\xi_{k}(x/\sqrt{\epsilon})}^{2} = \sum_{k=1}^{\infty}\Big\lvert \int_{\mathbb{R}}e^{-i z\frac{x}{\sqrt{\epsilon}}}e_{k}(z)m(z)dz \Big\rvert^{2}  =\int_{\mathbb{R}}m(z)dz=\mu(\mathbb{R})  }.
	\end{array}
\end{equation}
Moreover,  each $\xi_{k}^{\eps}$ is continuously differentiable, and for any $\epsilon>0$
\begin{equation}\label{identity2}
	\begin{array}{ll}
		&\ds{ \sum_{k=1}^{\infty}\big\lvert(\xi_{k}^{\epsilon})'(x)\big\rvert^{2} = \sum_{k=1}^{\infty}\Big\lvert -\frac{i}{\sqrt{\epsilon}} \int_{\mathbb{R}}z e^{ -i z\frac{x}{\sqrt{\epsilon}}}e_{k}(z)m(z)dz \Big\rvert^{2}  }\\
		\vs
		&\ds{\quad\quad\quad\quad\quad=\frac{1}{\epsilon}\sum_{k=1}^{\infty}\Big\lvert\Inner{ i\cdot e^{ i  \frac{x}{\sqrt{\epsilon}}\cdot}, e_{k} }_{L^{2}(\mathbb{R},d\mu;\mathbb{C})} \Big\rvert^{2} =\frac{1}{\epsilon}\int_{\mathbb{R}}\abs{z}^{2}m(z)dz  }.
	\end{array}
\end{equation}

\end{proof}

In what follows, we shall denote by $H^{s}_{\lambda}(\R)$, $s\geq0$,  the completion of $\mathcal{S}(\R)$ with respect to the norm
\begin{equation}
	\vert u\vert_{H^{s}_{\lambda}(\R)}:=\Big(\int_{\R}\big(1+\abs{x}^{2}\big)^{s}\big\lvert\F(\lambda \,u)(x)\big\rvert^{2}dx\Big)^{\frac{1}{2}},
\end{equation}
where  $\lambda(x)=e^{-x^{2}/2}$, $x\in\R$.

\begin{Lemma}
	For every $s\in[0,1]$, there exists some constant $c=c(s)>0$ such that 
	\begin{equation}\label{hilbert_schmidt}
		\sup_{\eps\in (0,1)}\eps^s\,\sum_{k=1}^{\infty}\vert \xi_{k}^{\epsilon}\vert _{H^{s}_{\lambda}(\mathbb{R})}^{2}<\infty.
			\end{equation}
\end{Lemma}

\begin{proof}
For every $s\in[0,1]$ and $\eps>0$, we have
\begin{equation*}
	\begin{array}{ll}
		&\ds{\sum_{k=1}^{\infty}\vert \xi_{k}^{\epsilon}\vert _{H^{s}_{\lambda}(\mathbb{R})}^{2} = \sum_{k=1}^{\infty}\int_{\mathbb{R}}\big(1+\abs{x}^{2}\big)^{s}\big\lvert \F(\lambda\,\xi_{k}^{\eps})(x)\big\rvert^{2}dx }\\
		\vs
		&\ds{\quad =\sum_{k=1}^{\infty}\int_{\mathbb{R}}\big(1+\abs{x}^{2}\big)^{s}\Big\lvert \int_{\mathbb{R}}e^{- i xz}\Big(\lambda(z)\int_{\mathbb{R}}e^{- iy\frac{z}{\sqrt{\eps}}}e_{k}(y)\mu(dy)\Big) dz \Big\rvert^{2} dx    }\\
		\vs
		&\ds{\quad \quad = \int_{\mathbb{R}}\big(1+\abs{x}^{2}\big)^{s}\sum_{k=1}^{\infty}\Big\lvert \int_{\mathbb{R}}\Big(\int_{\mathbb{R}}e^{- iz\big(x+\frac{y}{\sqrt{\eps}}\big)}\lambda(z)dz\Big)e_{k}(y)\mu(dy)  \Big\rvert^{2}dx  }\\
		\vs
		&\ds{\quad \quad \quad =\int_{\mathbb{R}}\big(1+\abs{x}^{2}\big)^{s}\int_{\mathbb{R}}\Big\lvert\int_{\mathbb{R}}e^{ iz\big(x+\frac{y}{\sqrt{\eps}}\big)}\lambda(z)dz\Big\rvert^{2}\mu(dy)dx }\\
		\vs
		&\ds{= \int_{\mathbb{R}^{2}}\big(1+\abs{x-y/\sqrt{\eps}}^{2}\big)^{s}\big\lvert\F(\lambda)(x)\big\rvert^{2}\mu(dy)dx\lesssim\frac{1}{\eps^{s}}\,\vert \lambda\vert _{H^{s}_{1}(\mathbb{R})}^{2}\int_{\mathbb{R}}\big(1+\abs{y}^{2}\big)^{s}\mu(dy)  }.
	\end{array}
\end{equation*}
\end{proof}

\begin{Remark}
{\em From \eqref{hilbert_schmidt}, we have that for any $\eps>0$, the embedding $K^{\eps}\subset H^{1}_{\lambda}(\mathbb{R})$ is Hilbert-Schmidt, which implies that the paths of $w^{\eps}$ belong to $C^{\delta}([0,T];H^{1}_{\lambda}(\mathbb{R}))$, for any $\delta\in [0,1/2)$. Moreover, for every $s\in [0,1]$ and $\delta\in [0,1/2)$, there exists some constant $c_\delta(s)>0$ such that 
\begin{equation}\label{uniform_est_noise}
\sup_{\eps\in(0,1)} \eps^{s}\,\E\vert w^{\eps}\vert _{C^{\delta}([0,T];H^{s}_{\lambda}(\mathbb{R}))}\leq c_\delta(s).
\end{equation}}
\end{Remark}

As a consequence of \eqref{uniform_est_noise}, we have that for any $s\in [0,1]$, the family  $\{\L(\eps^{s}w^{\eps}) \}_{\eps\in(0,1)}$ is tight in $C([0,T];U)$, whenever the space $H^{s}_{\lambda}(\mathbb{R})$ is compactly embedded in $U$.
		In particular, for any $s\in[0,1]$, $\delta\in[0,1/2)$ and $s'>s$, the family  
		\begin{equation}\label{sm130}\{\L(\eps^{s'}w^{\eps}) \}_{\eps\in(0,1)} \ \ \text{is tight in }\ \ C^{\delta}([0,T];H^{s}_{\lambda}(\mathbb{R})).\end{equation}

\subsection{Well-posedness of the second-order system \eqref{SPDE1}}

For every $\eps>0$ and $(u,v)\in\mathcal{M}$, we have
\begin{equation}
	\text{tr}_{K_{\eps}}\big(u\times (u\times v)\big)= \big(-\abs{u}^{2}v+(u\cdot v)u \big)\sum_{k=1}^{\infty}(\xi_{k}^{\eps})^{2} = -c_0\,v,
\end{equation}
and system \eqref{SPDE1} can be rewritten as 
\begin{equation}\label{SPDE3}
	\le\{\begin{array}{l}
		\ds{\eps\partial_{t}^{2}u_{\eps}(t,x)+\eps\abs{\partial_{t}u_{\eps}(t,x)}^{2}u_{\eps}(t,x)=\partial_x^{\,2} u_{\eps}(t,x)+\abs{\partial_x u_{\eps}(t,x)}^{2}u_{\eps}(t,x)-\gamma_0\,\partial_{t}u_{\eps}(t,x)}\\
		[10pt]
\ds{\quad\quad\quad\quad\quad\quad\quad\quad\quad\quad\quad\quad\quad\quad \quad\quad\quad\quad +\sqrt{\eps}\big(u_{\eps}(t)\times\partial_{t}u_{\eps}(t)\big)\partial_{t}w^{\eps}(t,x), }\\
		[10pt]
		\ds{u_{\eps}(0,x)=u^\e_0(x),\ \ \ \ \partial_{t}u_{\eps}(0,x)=v^\e_0(x) },
	\end{array}\r.
\end{equation}
where we have defined
\begin{equation}
\label{gamma0}
\gamma_0:=\gamma+\frac{c_0}2.	
\end{equation}


According to \cite[Theorem 11.1]{brz2007} (see also \cite[Theorem B.1]{brzgoldys2022}),  the following result holds.
\begin{Theorem}\label{system_wellposedness_loc}
	Let $0<T<R$ be fixed. Then, for every $\e>0$ and every  
	\[(u^\e_0,v^\e_0)\in \left(H^{2}_{\text{loc}}(\mathbb{R})\times  H^{1}_{\text{loc}}(\mathbb{R})\right)\cap \mathcal{M},\] there exists a process $u_\e:[0,T)\times \mathbb{R}\times \Omega\to \mathbb{S}^{2}$ which satisfies the following properties.
	\begin{enumerate}
		\item[i.] The random variable $u_\epsilon(t,x,\cdot)$ is $\F_{t}$-measurable, for every $t<T$ and $x\in \mathbb{R}$.
		
		\item[ii.]  The mapping $[0,T)\ni t\mapsto u_\epsilon(t,\cdot,\omega)\in H^{2}((-R,R);\mathbb{R}^{3})$ is continuous, and the mapping $[0,T)\ni t\mapsto u_\epsilon(t,\cdot,\omega)\in H^{1}((-R,R);\mathbb{R}^{3})$ is continuously differentiable,  $\mathbb{P}$-almost surely.

		\item[iii.] The identities $u_\epsilon(0,x,\omega)=u^\e_0(x)$ and $\partial_{t}u_\epsilon(0,x,\omega)=v^\e_0(x)$ hold, for every $x\in\mathbb{R}$, $\P$-almost surely.
		
		\item[iv.] For every $t\geq 0$ and $R>0$, the process $u_\epsilon$ satisfies the equation 
		
		\begin{equation*}
			\begin{array}{ll}
				&\ds{ \partial_{t}u_\epsilon(t) = v^\e_{0}+\int_{0}^{t}\Big(\partial_x^{\,2} u_\epsilon(s)+\abs{\partial_x u_\epsilon(s)}^{2}u_\epsilon(s)-\abs{\partial_{t}u_\epsilon(s)}^{2}u(s)-\gamma_0\,\partial_{t}u_\epsilon(s) \Big)ds  }\\
				\vs
				&\ds{\quad\quad\quad\quad  \quad\quad\quad\quad \quad\quad+\int_{0}^{t}\big(u_\epsilon(s)\times \partial_{t}u_\epsilon(s)\big)\,dw(s)  },
			\end{array}
		\end{equation*}
		in $L^{2}((-R,R);\mathbb{R}^{3})$, $\P$-almost surely.
	\end{enumerate}
	Moreover, if there exists another process $\tilde{u}_\epsilon:[0,T)\times \mathbb{R}\times \Omega\to \mathbb{S}^{2}$  which satisfies the above properties, then $\tilde{u}_\epsilon(t,x,\omega)=u_\epsilon(t,x,\omega)$, for every $\abs{x}<R-t$ and $t\in[0,T)$, $\P$-almost surely.
\end{Theorem}

In fact, as we will prove at the end of Section \ref{energy2}, we can extend the result above  as stated below.

\begin{Theorem}\label{system_wellposedness}
	Let $T>0$ and $(u^\e_0,v^\e_0)\in \big(\dot{H}^{2}(\mathbb{R})\times H^{1}(\mathbb{R})   \big)\cap \mathcal{M}$ be fixed. Then, under Hypothesis \ref{H1},  there exists a unique global strong $\{\F_{t}\}$-adapted solution $u_\epsilon$ to \eqref{SPDE3} such that
	\begin{equation*}
		u_\epsilon\in L^2(\Omega;L^{\infty}(0,T;\dot{H}^{2}(\mathbb{R}))),\ \ \ \ \partial_{t}u_\epsilon\in L^2(\Omega;L^{\infty}(0,T;H^{1}(\mathbb{R}))).
	\end{equation*}
\end{Theorem}

\subsection{Spatial fractional noise}
In the second part of this paper, when dealing with the study of the normal fluctuations of $u_\e$ around its deterministic  limit, we will assume a special structure for the noise. Below we give all the  definitions and preliminaries.  

We shall denote by $w^H(t)$   a Gaussian noise, defined on a complete filtered probability space $\big(\Omega,\mathcal{F},(\mathcal{F}_{t})_{t\geq0},\P\big)$, 
which is white in time and has a spatially homogeneous correlation of fractional type, with Hurst index $H\in [1/2,1)$. More precisely, if $C^{\infty}_{0}([0,\infty)\times \R)$ is the space of  infinitely differentiable functions with compact support, then $\{w^{H}(\varphi)\ ;\ \varphi\in\,C^{\infty}_{0}([0,\infty)\times \R)\}$ 
 is a family of centered Gaussian random variables, with covariance 
\begin{equation}
	\E\big(w^{H}(\varphi)w^{H}(\psi)\big) = \int_{0}^{\infty}\int_{\R}\mathcal{F}\varphi(t,\cdot)(y)\overline{\mathcal{F}\psi}(t,\cdot)(y)\mu_{H}(dy)dt,
\end{equation}
where the spectral measure $\mu_{H}$ of $w^{H}(t)$ is given by
\begin{equation}
	\mu_{H}(dx) = a_H\abs{x}^{1-2H}dx,
\end{equation}
for some constant $a_H>0$. In particular, the covariance function of the noise $w^{H}(t)$ can be written as 
\begin{equation}
	\E\big({w}^{H}(t,x)\,w^{H}(s,y)\big) = (t\wedge s)\,\Gamma_{H}(x-y),
\end{equation}
where, as we already mentioned, 
\begin{equation}
	\Gamma_{H}(x)=\int_{\R}e^{-ixy}\mu_{H}(dy) = \abs{x}^{2H-2}.
\end{equation}
Moreover, if $\{e_{k}^{H}\}_{k\in\mathbb{N}}\subset \mathcal{S}_{(s)}(\R)$ is a complete orthonormal basis of $L^{2}_{(s)}(\R,\mu_{H})$, then 
\[\xi_{k}^{H}:=\mathcal{F}(e_{k}^{H}\mu_{H}),\ \ \ \ \ \ k \in\,\mathbb{N},\]  is a complete orthonormal basis for $K^{H}$, the reproducing kernel Hilbert space of $w^{H}(t)$.

Next, we introduce the process 
\begin{equation}\label{sm50}
w(t,x):=(\eta\ast w^{H}(t,\cdot))(x) = \int_{\R}\eta(x-y)w^{H}(t,y)dy,
\end{equation}
and we assume the following conditions. 
\begin{Hypothesis}\label{H2}
The function  $\eta:\R\to \R$ is  smooth, symmetric, non-negative and belongs to  $L^{1}(\R)$, with
\begin{equation}
	\int_{\R}\eta(x)dx=1.
\end{equation}
Moreover,  there exist some positive constants $a\geq H-1/2$ and $b<H-2$ such that
\begin{equation}\label{mollifier1}
	1-\mathcal{F}\eta(x)\lesssim \abs{x}^{a},\ \ \ \ x\in(-1,1),
\end{equation}
and 
\begin{equation}\label{mollifier2}
	\abs{\mathcal{F}\eta(x)}\lesssim \abs{x}^{b},\ \ \ \ x\geq1.
\end{equation}
	\end{Hypothesis}
It is easy to check  that under these conditions $w(t)$ is a spatially homogeneous Wiener process, with  covariance 
\begin{equation}
	\E\big({w}(t,x){w}(s,y)\big) = (t\wedge s)\,\Gamma(x-y),
\end{equation}
where 
\begin{equation}
	\Gamma(x) = \int_{\R}\abs{\mathcal{F}\eta(y)}^{2}e^{-ixy}\mu_H(dy),
\end{equation}
and its spectral measure $\mu$  is
\begin{equation}
	\mu(dx) = \abs{\mathcal{F}\eta(x)}^{2}\,\mu_H(dx),
\end{equation}
(see \cite[Lemma 4.1]{chen2021}). Moreover, since $|\mathcal{F}(\eta)|\leq 1$, due to \eqref{mollifier2}, we have
\begin{equation}
	\begin{array}{ll}
		&\ds{c_0=\mu(\R)= a_H\int_{\R}\abs{\mathcal{F}\eta(y)}^{2}\abs{y}^{1-2H}dy  \lesssim \int_{\abs{y}<1}\abs{y}^{1-2H}dy+\int_{\abs{y}\geq1}\abs{y}^{2b+1-2H}dy <\infty  },
	\end{array}
\end{equation}
and
\begin{equation}
	\begin{array}{ll}
		&\ds{c_1 =\int_{\R}\abs{y}^{2}\mu(dy)= a_H\int_{\R}\abs{\mathcal{F}\eta(y)}^{2}\abs{y}^{3-2H}dy\lesssim \int_{\abs{y}<1}\abs{y}^{3-2H}dy+\int_{\abs{y}\geq1}\abs{y}^{2b+3-2H}dy <\infty  },
	\end{array}
\end{equation}
In particular,  $w(t)$, $t\geq 0$,  satisfies Hypothesis \ref{H1}. 
Furthermore, if $\{e_{k}\}_{k\in\mathbb{N}}\subset \mathcal{S}_{(s)}(\R)$ is a complete orthonormal basis of $L^{2}_{(s)}(\R,\mu)$, then $\xi_{k}:=\mathcal{F}(e_{k}\mu)$, $k \in\,\mathbb{N}$,  is a complete orthonormal basis of $K$, the reproducing kernel Hilbert space of $w(t)$.

\subsection{An alternative formulation}

Now,  we define $\eta_{\epsilon}(x)=\epsilon^{-1/2}\eta(\epsilon^{-1/2}x)$, and 
\begin{equation}
	Q^{\epsilon}h\, (x):= (\eta_{\epsilon}\ast h)(x)=\frac{1}{\sqrt{\epsilon}}\int_{\R}\eta\Big(\frac{x-y}{\sqrt{\epsilon}}\Big)h(y)dy,\ \ \ \ \ \ h \in\,L^2(\R).
\end{equation}
Thus, if we denote \[Q^{\epsilon}w^{H}(t,x):=(\eta_{\epsilon}\ast w^{H}(t,\cdot))(x),\]  we have 
\begin{equation}
	\begin{array}{l}
		\ds{ \E\big(Q^{\epsilon}{w}^{H}(t,x)\,Q^{\epsilon}{w}^{H}(s,y)\big)   = (t\wedge s)\int_{\R}\abs{\mathcal{F}\eta_{\epsilon}(z)}^{2}e^{-iz(x-y)}\mu_H(dz) }\\
		\vs 
		\ds{\quad\quad\quad\quad =(t\wedge s)\,\frac{1}{\epsilon}\int_{\R}\Big\lvert\int_{\R}e^{-izw}\eta\big(\frac{w}{\sqrt{\epsilon}}\big)dw\Big\rvert^{2}e^{-iz(x-y)}\mu_H(dz)   }\\
		\vs 
		\ds{=(t\wedge s)\frac{1}{\epsilon^{1-H}}\int_{\R}\abs{\mathcal{F}\eta(z)}^{2}e^{-iz(x-y)/\sqrt{\epsilon}}\,\mu_H(dz) =\epsilon^{H-1}(t\wedge s)\Gamma((x-y)/\sqrt{\epsilon}).   }
	\end{array}
\end{equation}
Recalling \eqref{sm150}, this implies that
\begin{equation}
\E\big(Q^{\epsilon}{w}^{H}(t,x)\,Q^{\epsilon}{w}^{H}(s,y)\big)=\e^{H-1}\,\E\big({w}^{\epsilon}(t,x)\,{w}^{\epsilon}(s,y)\big),	
\end{equation}
so that
\begin{equation}
	{w}^{\epsilon}(t,x) \sim\epsilon^{(1-H)/2}Q^{\epsilon}{w}^{H}(t,x).
\end{equation}
In particular, we can alternatively write system \eqref{SPDE3} as 
\begin{equation}\label{alternative_form}
	\le\{\begin{array}{l}
		\ds{\epsilon\partial_{t}^{2}u_{\epsilon}(t,x)+\epsilon\abs{\partial_{t}u_{\epsilon}(t,x)}^{2}u_{\epsilon}(t,x)=\partial_x^{\,2} u_{\epsilon}(t,x)+\abs{\partial_x u_{\epsilon}(t,x)}^{2}u_{\epsilon}(t,x) -\,\gamma_0\,\partial_{t}u_{\eps}(t,x)  }\\[10pt]
		\ds{\quad\quad\quad\quad\quad\quad\quad\quad\quad\quad\quad\quad \quad\quad\quad\quad +\epsilon^{1-H/2}\big(u_{\epsilon}(t)\times\partial_{t}u_{\epsilon}(t)\big)Q^{\epsilon}\partial_{t}w^{H}(t,x), }\\
		[10pt]
		\ds{u_{\epsilon}(0,x)=u_{0}^{\epsilon}(x),\ \ \ \ \ \ \partial_{t}u_{\epsilon}(0,x)=v_{0}^{\epsilon}(x) }.
	\end{array}\r.
\end{equation}

\section{Main results}\label{sec-main}

In this section, we will describe the two main results of this paper. 
The first one deals with the limiting behavior of the solution $u_\e$ of equation \eqref{SPDE3}, the second one deals with the analysis of the fluctuations of $u_\e$ around its deterministic limit $u$. We will also some results about $u$

\subsection{Uniqueness for the heat flow harmonic map equation}

We start introducing the deterministic equation
\begin{equation}\label{limiting_equation}
	\le\{\begin{array}{l}
		\ds{ \gamma_0\,\partial_{t}u(t,x) = \partial_{x}^{2} u(t,x)+\abs{\partial_{x} u(t,x)}^{2}u(t,x)  , \ \ \ \ \  \ (t,x)\in \mathbb{R}^{+}\times \mathbb{R},}\\[10pt]
		\ds{u(0,x)=u_{0}(x),\ \ \ \ \ \ x\in \mathbb{R} },
	\end{array}\r.
\end{equation}
where $u_0\in M:=\{u:\mathbb{R}\to \mathbb{S}^{2}\}$. We recall that  $\gamma_0=\gamma+c_0/2$.

If $u\in L^{\infty}(0,T;\dot{H}^{1}(\R))\cap L^{2}(0,T;\dot{H}^{2}(\R))$ is a solution of equation \eqref{limiting_equation} with initial condition $u_0\in M$, then we have $u(t)\in M$, for $t\in [0,T]$. Indeed, if we define 
	\begin{equation*}
		h(t,x):=\abs{u(t,x)}^{2}-1,\ \ \ \ (t,x)\in [0,T]\times \R,
	\end{equation*}
	then we have $\partial_x^{\,2} h(t,x) = 2\partial_x^{\,2} u(t,x) \cdot u(t,x)+2\abs{\partial_x u(t,x)}^{2}$ and $\partial_t, h(t,x)=2\partial_t u(t,x)\cdot u(t,x)$. This implies 
	\begin{equation*}
		\partial_{t}h(t,x) = \partial_x^{\,2} h(t,x) +2 \abs{\partial_x u(t,x)}^{2}h(t,x),\ \ \ \ h(0,x) = 0,
	\end{equation*}
so that $h(t,x)=0$, for $(t,x)\in[0,T]\times \mathbb{R}$.

	In Appendix \ref{AppB}  we will prove the following uniqueness result for   equation \eqref{limiting_equation}.

\begin{Theorem}\label{uniqueness}
	For every $T>0$ and $u_0\in \dot{H}^{1}(\R)\cap M$, problem \eqref{limiting_equation} admits at most one solution $u$ in $L^{\infty}((0,T)\times \R)\cap L^{\infty}(0,T;\dot{H}^{1}(\R))\cap L^{2}(0,T;\dot{H}^{2}(\R))$.
\end{Theorem}

\subsection{A law of large numbers}
The first main result of this paper concerns the convergence of the solution of equation \eqref{SPDE3} to the unique solution of equation \eqref{limiting_equation}. 

\begin{Theorem}\label{small_mass_limit}
	Assume Hypotheses \ref{H1} and \ref{H1-bis} hold, and fix $u_{0}\in \dot{H}^{1}(\R)\cap M$   such that
	\begin{equation}\label{sm200}
		\lim_{\epsilon\to0}\abs{u_{0}^{\epsilon}-u_0}_{L^2_{\text{loc}}(\R)}=0.
	\end{equation} Then,	for every $T>0$, $\delta_{1}<1$ and $\delta_{2}<2$, and  every $\eta>0$, we have
	\begin{equation}
		\label{convergence_pr}
		\lim_{\eps\to0}\P\Big(\abs{u_{\eps}-u}_{C([0,T];H^{\delta_{1}}_{\text{loc}}(\R))}+\abs{u_{\eps}-u}_{L^{2}(0,T;H^{\delta_{2}}_{\text{loc}}(\R))}>\eta \Big) = 0,
	\end{equation}
where $u$ is the unique solution of the  deterministic equation \eqref{limiting_equation}.
	
\end{Theorem}

\begin{Remark}
{\em \begin{enumerate}
 \item[1.]	Due to Hypothesis \ref{H1-bis}, the sequence $\{u^\e_0-u_0\}_{\e \in\,(0,1)}$ is bounded in $H^1_{\text{loc}}(\R)$. Hence \eqref{sm200} implies that for every $\delta<1$
\[\lim_{\e\to 0}\vert u^\e_0-u_0\vert_{H^\delta_{\text{loc}}(\R)}=0.\]
\item[2.] In Lemma \ref{lemB1} we will prove that any  function in $\dot{H}^1(\R)\cap M$ can be approximated by  smooth functions which still belong to  $M$. In particular, this justifies condition \eqref{sm200}.
 \end{enumerate}

}
	
\end{Remark}

Next result concerns the limiting behavior of $\partial_t u_\epsilon$.
\begin{Theorem}\label{teo3.4}
	Under the same assumptions of Theorem \ref{small_mass_limit}, we have that  $\partial_t u_\epsilon$ converges to $\partial_t u$ in probability, as $\e\downarrow 0$, with respect to the weak convergence in $L^2(0,T;L^2(\mathbb{R}))$. However, if $c_0\neq 0$ and  
	\begin{equation}
		\label{sm105}
	\lim_{\epsilon\to 0}\sqrt{\epsilon}\,\vert v^\epsilon_0\vert_{L^2(\mathbb{R})}=0,	
	\end{equation} we have that  $\partial_t u_\epsilon$ does not converge to $\partial_t u$  in probability, with respect to the strong convergence in $L^2(0,T;L^2(\mathbb{R}))$. 
\end{Theorem}

\medskip

\subsection{A central limit theorem}\label{sub3.3}
Once proved the convergence of $u_\e$ to its deterministic limit $u$, our next step is studying the fluctuations of $u_\e$ around $u$. 
To this purpose, we assume that the noise has a special structure
\begin{equation}\label{sm50-bis}
	w(t,x):=(\eta\ast w^{H}(t,\cdot))(x) = \int_{\R}\eta(x-y)w^{H}(t,y)dy,
\end{equation}
for some kernel $\eta$ satisfying Hypothesis \ref{H2}. Moreover, for every $\epsilon>0$ we define
\begin{equation}
	y_{\epsilon}(t,x) :=\epsilon^{H/2-1}(u_{\epsilon}(t,x)-u(t,x)),\ \ \ \ (t,x)\in \mathbb{R}^{+}\times \mathbb{R}.
\end{equation}
Once fixed these notations, we can state the second main result of this paper.

\begin{Theorem}\label{CLT}
	Assume $w(t)$ is given by \eqref{sm50-bis}, for some $H\in (1/2,1)$,  and assume Hypotheses \ref{H1}, \ref{H1-bis}, and \ref{H2} hold.  Moreover, in addition to Hypothesis \ref{H1-bis}, assume that $u^\e_0 \in\,\dot{H}^3(\mathbb{R})$ with
	\begin{equation} \label{fine25-bis}\lim_{\e\to 0}\e^{\lambda}\,\vert u_0^\e\vert_{\dot{H}^3(\mathbb{R})}=0,\end{equation}
	for some $\lambda \in\,(0,2(1-H))$. Then, if $u_0\in \dot{H}^{1}(\R)\cap M$  is such that
	\begin{equation}\label{initial_condition_rate}
		 \abs{u_{0}^{\epsilon}-u_{0}}_{L^2(\R)}=o(\epsilon^{1-H/2}), \ \ \ \ 0<\epsilon\ll 1,
	\end{equation}
	and
\begin{equation}\label{initial_condition_rate-bis}	
\lim_{\e\to 0} e^{-\lambda }\,\vert u^\e_0-u_0\vert^2_{H^1(\mathbb{R})}=0,
\end{equation}
for every $T>0$ we have 
	\begin{equation}
\lim_{\epsilon\to 0}		\mathbb{E}\,\vert y_{\epsilon}-\varrho\vert_{L^{2}(0,T;{L^2(\R)})}=0,
	\end{equation}
	where $\varrho\in L^{2}(\Omega;L^{2}(0,T;H))$ is the unique solution of the  equation
	\begin{equation}
		\le\{\begin{array}{l}
			\ds{\gamma_0\,\partial_{t}\varrho(t) = \partial_x^{\,2} \varrho(t)+\abs{\partial_x u(t)}^{2}\varrho(t)+2(\partial_x u(t)\cdot \partial_x \varrho(t))u(t) +\big(u(t)\times\partial_{t}u(t)\big)\partial_{t}w^{H}(t), }\\
			[10pt]
			\ds{\varrho(0)=0 },
		\end{array}\r.
	\end{equation}
	and $u$ is the unique solution of equation \eqref{limiting_equation}, with initial condition $u_0$.

\end{Theorem}

\subsection{Regularity of solutions for the heat flow  harmonic map equation}
As a consequence Theorem \ref{uniqueness} and  Theorem \ref{small_mass_limit}, we obtain that for every $T>0$ and $u_0\in\, \dot{H}^{1}(\R)\cap M$, there exists a unique solution $u\in L^{\infty}(0,T; \dot{H}^{1}(\R))\cap L^{2}(0,T; \dot{H}^{2}(\R))$ for the heat flow harmonic map equation \eqref{limiting_equation}. This result seems to be new in the existing literature. In fact, in Appendix \ref{AppB} we will prove the following further regularity result.

\begin{Theorem}\label{regularity}
	Let $u_0\in \dot{H}^{k}(\R)\cap M$, for some $k\geq 1$. Then there exists a unique solution $u\in L^{\infty}(0,T;\dot{H}^{k}(\R))\cap L^{2}(0,T;\dot{H}^{k+1}(\R))$ of \eqref{heat_flow} for every $T>0$, with 
	\begin{equation}
		\sup_{t\in [0,T]}\abs{u(t)}_{\dot{H}^{k}(\R)}^{2} + \int_{0}^{T}\abs{u(t)}_{\dot{H}^{k+1}(\R)}^{2}dt+\int_{0}^{T}\abs{\partial_{t}u(t)}_{H^{k-1}(\R)}^{2}dt \lesssim_{\,T} c_k(\abs{u_0}_{\dot{H}^{k}(\R)} \big).
	\end{equation}
\end{Theorem}

\section{Uniform estimates}
\label{energy}

If we denote $v_{\eps}=\partial_{t}u_{\eps}$, then problem \eqref{SPDE3} can be rewritten as
\begin{equation}
	\le\{\begin{array}{l}
		\ds{du_{\epsilon}=v_{\epsilon}dt, }\\[8pt]
		\ds{dv_{\epsilon}=\frac{1}{\epsilon}\Big(\partial_x^{\,2} u_{\epsilon}+\abs{\partial_x u_{\epsilon}}^{2}u_{\epsilon}-\epsilon\abs{v_{\epsilon}}^{2}u_{\epsilon}-\,\gamma_0 v_{\epsilon}\Big)dt +\frac{1}{\sqrt{\epsilon}}(u_{\epsilon}\times v_{\epsilon})dw^{\epsilon}(t), }\\[10pt]
		\ds{u_{\epsilon}(0,x)=u_{0}^{\epsilon}(x),\ \ \ v_{\epsilon}(0,x)=v_{0}^{\epsilon}(x) }.
	\end{array}\r.
\end{equation}
For the sake of simplicity of notations, we   denote
\begin{equation*}
	L^{p}(I)=L^{p}(\mathcal{O};\mathbb{R}^{3}),\ \ \ \ W^{r,p}(I)=W^{r,p}(I;\mathbb{R}^{3}),\ \ \ \ \dot{W}^{r,p}(I)=\dot{W}^{r,p}(I;\mathbb{R}^{3}),
\end{equation*}
for any interval $I\subseteq \mathbb{R}$, and any $r\in\R$ and $p\geq 1$. 
Moreover, for any $R>0$, we  denote by $I(R)$ the interval $(-R,R)$, and for any $u:\mathbb{R}\to \mathbb{R}^{3}$, we  define the scalar product
\begin{equation}
	\Inner{u,v}_{L^{2}(\partial I(R))}:=u(R)\cdot v(R)+u(-R)\cdot v(-R),
\end{equation}
and the corresponding norm
\begin{equation}
	\abs{u}_{L^{2}(\partial I(R))}^{2}:= \abs{u(R)}^{2}+\abs{u(-R)}^{2}.
\end{equation}

\subsection{Estimates in $\dot{H}^1(\R)\times L^2(\R)$}
\label{energy1}

The first result is an important consequence of a generalized It\^o's formula.
\begin{Lemma}
If $u_{\epsilon}$ is a strong solution to \eqref{SPDE3}, with initial condition $(u_0,v_0)$ belonging to $\big(H^{2}_{\text{loc}}(\mathbb{R})\times H^{1}_{\text{loc}}(\mathbb{R})   \big)\cap \mathcal{M}$, then
	for every  $0<T<R$ and $k=0,1$	we have
	\begin{equation}\label{ito1}
		\begin{array}{ll}
			&\ds{ d\Big( \abs{\partial_x^{\,k+1}u_{\epsilon}(t)}_{L^{2}(I((R-t)/\sqrt{\epsilon}))}^{2}+\epsilon\abs{\partial_x^{\,k}v_{\epsilon}(t)}_{L^{2}(I((R-t)/\sqrt{\epsilon}))}^{2} \Big) }\\[10pt]
			\vs
			&\ds{ \leq -2\,\gamma_0\abs{\partial_x^{\,k}v_{\epsilon}(t)}_{L^{2}(I((R-t)/\sqrt{\epsilon}))}^{2}dt  + 2\Inner{\partial_x^{\,k}v_{\epsilon}(t),\partial_x^{\,k}\big(\abs{\partial_x u_{\epsilon}(t)}^{2}u_{\epsilon}(t)\big)  }_{L^{2}(I((R-t)/\sqrt{\epsilon}))}dt  }\\
			\vs
			&\ds{  \quad -2\epsilon\Inner{\partial_x^{\,k}v_{\epsilon}(t), \partial_x^{\,k}\big(\abs{v_{\epsilon}(t)}^{2}u_{\epsilon}(t)\big)}_{L^{2}( I((R-t)/\sqrt{\epsilon}) )}dt  +\sum_{i=1}^{\infty}\big\lvert \partial_x^{\,k}\big((u_{\epsilon}(t)\times v_{\epsilon}(t))\xi_{i}^{\epsilon}\big)\big\rvert_{L^{2}(I((R-t)/\sqrt{\epsilon}))}^{2}dt}\\\vs
			&\ds{ \quad \quad \quad \quad \quad \quad+2\sqrt{\epsilon}\,\sigma_{k,\epsilon,I((R-t)/\sqrt{\epsilon})}\big(u_{\epsilon}(t),v_{\epsilon}(t)\big)dw^{\epsilon}(t),\ \ \ \ \ \ \ \ \ t\in[0,T], }
		\end{array}
	\end{equation}
	where,  we denoted 
	\begin{equation*}
		\sigma_{k,\epsilon,I}(u,v)\xi:= \Inner{\partial_x^{\,k}v, \partial_x^{\,k}((u\times v)\xi)}_{L^{2}(I)},\ \ \ \ (u,v)\in H^{k+1}(I)\times H^{k}(I),\ \ \ \xi\in K_{\epsilon}.
	\end{equation*}

\end{Lemma}

\begin{proof}
	The proof of \eqref{ito1} relies on the It\^{o} formula applied to 
	\begin{equation*}
		(t,v_{\epsilon})\mapsto \abs{\partial_x^{\,k}v_{\epsilon}}_{L^{2}(I((R-t)/\sqrt{\epsilon}))}^{2}.
	\end{equation*}
	For the rigorous proof, we refer to \cite[proof of Lemma 6.1]{brz2007}. Here, we  present some informal arguments.
	
	If $R>T$, then for every $t\in[0,T]$,
	\begin{equation}\label{ito1_step1}
		\begin{array}{ll}
			&\ds{ d\abs{\partial_x^{\,k}v(t)}_{L^{2}(I((R-t)/\sqrt{\epsilon}))}^{2}  =  -\frac{1}{\sqrt{\epsilon}}\abs{\partial_x^{\,k}v_{\epsilon}(t)}_{L^{2}(\partial I((R-t)/\sqrt{\epsilon}) )}^{2}dt  }\\
			\vs
			&\ds{\quad\quad + \frac{2}{\epsilon}\Big( \Inner{\partial_x^{\,k}v_{\epsilon}(t),\partial_x^{\,k}(\partial_x^{\,2} u_{\epsilon}(t)) }_{L^{2}(I((R-t)/\sqrt{\epsilon}) )}  +\Inner{\partial_x^{\,k}v_{\epsilon}(t),\partial_x^{\,k}\big(\abs{\partial_x u_{\epsilon}(t)}^{2}u_{\epsilon}(t)\big) }_{L^{2}(I((R-t)/\sqrt{\epsilon}) )} }\\
			\vs
			&\ds{\quad\quad - \epsilon\Inner{\partial_x^{\,k}v_{\epsilon}(t),\partial_x^{\,k}\big(\abs{v_{\epsilon}(t)}^{2}u_{\epsilon}(t)\big) }_{L^{2}(I((R-t)/\sqrt{\epsilon}) )}  -\gamma_{0}\abs{\partial_x^{\,k}v_{\epsilon}(t)}_{L^{2}(I((R-t)/\sqrt{\epsilon}) )}^{2} \Big)dt }\\
			\vs
			&\ds{\quad \quad  + \frac{1}{\epsilon}\sum_{i=1}^{\infty}\big\lvert \partial_x^{\,k}\big((u_{\epsilon}(t)\times v_{\epsilon}(t))\xi_{i}^{\epsilon}\big)\big\rvert_{L^{2}(I((R-t)/\sqrt{\epsilon}) )}^{2}dt   +\frac{2}{\sqrt{\epsilon}}\,\sigma_{k,\epsilon, I((R-t)/\sqrt{\epsilon})}(u_{\epsilon}(t),v_{\epsilon}(t))dw^{\epsilon}(t) }.
		\end{array}
	\end{equation}
	Note that for any $r>0$, 
	\begin{equation*}
		\begin{array}{ll}
			&\ds{ \Inner{\partial_x^{\,k}v, \partial_x^{\,k}(\partial_x^{\,2} u)}_{L^{2}(I(r))} = -\int_{-r}^{r}\partial_x^{\,k+1}v(x)\cdot \partial_x^{\,k+1}u(x) dx+(\partial_x^{\,k+1}u\cdot \partial_x^{\,k}v)\big|_{-r}^{r}  }\\
			\vs
			&\ds{\quad\quad\quad\quad \quad\quad \leq  -\Inner{\partial_x^{\,k+1}u,\partial_x^{\,k+1}v}_{L^{2}(I(r))}+\frac{1}{2\sqrt{\epsilon}}\abs{\partial_x^{\,k+1}u}_{L^{2}(\partial I(r))}^{2} +\frac{\sqrt{\epsilon}}{2}\abs{\partial_x^{\,k}v}_{L^{2}(\partial I(r))}^{2}  },
		\end{array}
	\end{equation*}
	so that we have 
	\begin{equation}\label{ito1_step2}
		2\Inner{\partial_x^{\,k+1}u,\partial_x^{\,k+1}v}_{L^{2}(I(r))}+2\Inner{\partial_x^{\,k}v,\partial_x^{\,k}(\partial_x^{\,2} u)}_{L^{2}(I(r))}\leq \frac{1}{\sqrt{\epsilon}}\abs{\partial_x^{\,k+1}u}_{L^{2}(\partial I(r))}^{2} + \sqrt{\epsilon}\abs{\partial_x^{\,k}v}_{L^{2}(\partial I(r))}^{2}.
	\end{equation}
Moreover,	
	\begin{equation}\label{ito1_step3}
		\frac{d}{dt}\abs{\partial_x^{\,k+1}u(t)}_{L^{2}(I(R-t) )}^{2}= 2\Inner{\partial_x^{\,k+1}u(t),\partial_x^{\,k+1}v(t)}_{L^{2}(I(R-t) )} - \frac{1}{\sqrt{\epsilon}}\abs{\partial_x^{\,k+1}u(t)}_{ L^{2}(\partial I(R-t) )}^{2},
	\end{equation}
	and combining this with \eqref{ito1_step2} we get
	\begin{equation}\label{fine40}2\Inner{\partial_x^{\,k}v,\partial_x^{\,k}(\partial_x^{\,2} u)}_{L^{2}(I(R-t))}\leq -\frac d{dt}\vert \partial_x^{k+1} u\vert_{L^2(I(R-t))}^2+\sqrt{\e} \,\abs{\partial_x^{\,k}v}_{L^{2}(\partial I(R-t))}^{2}.\end{equation}
Finally, if we plug \eqref{fine40} into \eqref{ito1_step1}	
	we obtain  \eqref{ito1}.

\end{proof}

\begin{Lemma}\label{lemmal1} For every $0<T<R$ and $t \in\,[0,T]$, we have	\begin{equation}\label{uniform_bound1_loc}
		\begin{array}{ll}
			&\ds{\abs{\partial_x u_{\epsilon}(t)}_{L^{2}(I((R-t)/\sqrt{\epsilon}))}^{2}+ \epsilon\abs{\partial_t u(t)}_{L^{2}(I((R-t)/\sqrt{\epsilon}))}^{2}+2\gamma \int_{0}^{t}\abs{\partial_t u(s)}_{L^{2}(I((R-s)/\sqrt{\epsilon}))}^{2}ds }\\
			\vs
			&\ds{\quad\quad \quad\quad \quad\quad \leq \abs{D u_0^{\epsilon}}_{L^{2}(\mathbb{R})}^{2}+\epsilon\,\abs{v_0^{\epsilon}}_{L^{2}(\mathbb{R})}^{2},\ \ \ \  \ \ \mathbb{P}-\text{a.s.}  }
		\end{array}
\end{equation}\end{Lemma}

\begin{proof}  
If $(u,v)\in (H^{1}(I)\times L^{2}(I))\cap \mathcal{M}$, then $u\cdot v=0$, and due to \eqref{identity1} this gives, we have
	\begin{equation*}
		\Vert u\times v\Vert _{\L_{2}(K_{\epsilon},L^{2}(I))}^{2}  = \sum_{i=1}^{\infty}\big\lvert (u\times v)\xi_{i}^{\epsilon}\big\rvert_{L^{2}(I)}^{2} = c_{0}\abs{v}_{L^{2}(I)}^{2}.
	\end{equation*}
	Hence, thanks to \eqref{ito1} with $k=0$ 
	\begin{equation*}
		\begin{array}{ll}
			&\ds{ d\Big( \abs{\partial_x u_{\epsilon}(t)}_{L^{2}(I((R-t)/\sqrt{\epsilon}) )}^{2}+\epsilon\abs{v_{\epsilon}(t)}_{L^{2}(I((R-t)/\sqrt{\epsilon}) )}^{2}
				\Big) \leq -2\,\gamma_0\abs{v_{\epsilon}(t)}_{L^{2}(I((R-t)/\sqrt{\epsilon}) )}^{2}dt  }\\
			\vs
			&\ds{\quad \quad \quad  + \sum_{i=1}^{\infty}\big\lvert (u_{\epsilon}(t)\times v_{\epsilon}(t))\xi_{i}^{\epsilon}\big\rvert_{L^{2}(I((R-t)/\sqrt{\epsilon}) )}^{2}dt= -2\gamma\abs{v_{\epsilon}(t)}_{L^{2}(I((R-t)/\sqrt{\epsilon}) )}^{2}dt,}
		\end{array}
	\end{equation*}
	and \eqref{uniform_bound1_loc} follows once we integrate with respect to  time.\end{proof}

\subsection{Estimates in $\dot{H}^2(\R)\times H^1(\R)$ and proof of Theorem \ref{system_wellposedness}
}
\label{energy2}

\begin{Lemma}
For every  $T>0$ and $R>T+1$ and for  every $\epsilon\in (0,1)$ we have
	\begin{equation}\label{uniform_bound2_loc}
		\begin{array}{ll}
			&\ds{ \E\sup_{t\in[0,T]}\Big(\abs{\partial_x ^{2}u_{\epsilon}(t)}_{L^{2}(I((R-t)/\sqrt{\epsilon}))}^{2}+\epsilon\abs{\partial_x v_{\epsilon}(t)}_{L^{2}(I((R-t)/\sqrt{\epsilon}))}^{2}\Big)   }\\
			\vs
			&\ds{\quad\quad +\E\int_{0}^{T}\Big(\abs{\partial_x v_{\epsilon}(s)}_{L^{2}(I((R-s)/\sqrt{\epsilon}))}^{2}+\abs{\partial_x ^{2}u_{\epsilon}(s)}_{L^{2}(I((R-s)/\sqrt{\epsilon}))}^{2}\abs{v_{\epsilon}(s)}_{L^{2}(I((R-s)/\sqrt{\epsilon}))}^{2}\Big) ds }\\
			\vs 
			&\ds{\quad\quad  \quad\quad  \quad\quad  \quad\quad  \lesssim_{\,T,\Lambda_1}\abs{D^{2}u_{0}^{\epsilon}}_{L^{2}(\R)}^{2}+\epsilon\abs{D v_0^{\epsilon}}_{L^{2}(\R)}^{2}+\frac{1}\epsilon}.
		\end{array}
	\end{equation} 
	In particular, 
	\begin{equation}\label{uniform_bound2_loc_particular}
		\begin{array}{ll}
			&\ds{  \sup_{R>1}\ \E\Bigg(\sup_{t\in[0,T]}\Big(\abs{\partial_{x}^{2}u_{\epsilon}(t)}_{L^{2}(-R,R)}^{2}+\epsilon\abs{\partial_{x}v_{\epsilon}(t)}_{L^{2}(-R,R)}^{2}\Big)+\int_{0}^{T}\abs{\partial_{x}v_{\epsilon}(t)}_{L^{2}(-R,R)}^{2}dt\Bigg)   }\\
			\vs 
			&\ds{ \quad\quad\quad\quad  \lesssim_{\,T,\Lambda_1}\abs{D^{2}u_{0}^{\epsilon}}_{L^{2}(\R)}^{2}+\epsilon\abs{D v_0^{\epsilon}}_{L^{2}(\R)}^{2}+\frac{1}\epsilon,\ \ \ \ \epsilon\in(0,1] }.
		\end{array}
	\end{equation}
\end{Lemma}

\begin{proof} By applying \eqref{ito1} with $k=1$,  we have  
	\begin{equation}\label{sm6}
		\begin{array}{ll}
			&\ds{ d\Big( \abs{\partial_x^{2}u_{\epsilon}(t)}_{L^{2}(I((R-t))}^{2} +\abs{\partial_xv_{\epsilon}(t)}_{L^{2}(I(((R-t)/\sqrt{\epsilon}))}^{2} \Big)    }\\
			\vs
			&\ds{\leq-2\,\gamma_0\abs{\partial_xv_{\epsilon}(t)}_{L^{2}(I(((R-t)/\sqrt{\epsilon}))}^{2}dt +2\Inner{\partial_xv_{\epsilon}(t),\partial_x\big(\abs{\partial_x u_{\epsilon}(t)}^{2}u_{\epsilon}(t)\big) }_{L^{2}(I(((R-t)/\sqrt{\epsilon}))}dt  }\\
			\vs
			&\ds{\quad - 2\epsilon\Inner{\partial_x v(t),\partial_x \big(\abs{v(t)}^{2}u(t)\big) }_{L^{2}(I((R-t))}dt   + \sum_{i=1}^{\infty}\big\lvert \partial_x \big((u_{\epsilon}(t)\times v_{\epsilon}(t))\xi_{i}^{\epsilon}\big) \big\rvert_{L^{2}(I((R-t)/\sqrt{\epsilon}))}^{2}dt   }\\
			\vs
			&\ds{\quad\quad \quad\quad \quad\quad \quad\quad  +2\sqrt{\epsilon}\,\sigma_{1,\epsilon,I((R-t)/\sqrt{\epsilon})}(u_{\epsilon}(t),v_{\epsilon}(t))dw^{\epsilon}(t)  }.
		\end{array}
	\end{equation}
Thus, if for every $\epsilon\in(0,1)$ and $\lambda>0$ we define 
	\begin{equation}
		Y_{\epsilon,\lambda}(t):= \exp\Bigg(-\lambda\int_{0}^{t}\Big(1+\abs{v_{\epsilon}(s)}_{L^{2}(I((R-s)/\sqrt{\epsilon}))}^{2}\Big)ds\Bigg),\ \ \ \ t\in[0,T],
	\end{equation}
we get
	\begin{equation*}
		\begin{array}{ll}
			&\ds{ \frac{1}{2}d\Big(Y_{\epsilon,\lambda}(t)\Big( \abs{\partial_{x}^{2}u_{\epsilon}(t)}_{L^{2}(I((R-t)/\sqrt{\epsilon}))}^{2}+\epsilon\abs{\partial_{x}v_{\epsilon}(t)}_{L^{2}(I((R-t)/\sqrt{\epsilon}))}^{2}  \Big) \Big)   }\\
			\vs
			&\ds{ \leq Y_{\epsilon,\lambda}(t)\Big(-\frac{\lambda}{2}\Big(1+\abs{v_{\epsilon}(t)}_{L^{2}(I((R-t)/\sqrt{\epsilon}))}^{2}\Big)\Big(\abs{\partial_{x}^{2}u_{\epsilon}(t)}_{L^{2}(I((R-t)/\sqrt{\epsilon}))}^{2}+\epsilon\abs{\partial_{x}v_{\epsilon}(t)}_{L^{2}(I((R-t)/\sqrt{\epsilon}))}^{2}\Big) }\\
			\vs
			&\ds{  \quad\quad -\gamma_{0}\abs{\partial_{x}v_{\epsilon}(t)}_{L^{2}(I((R-t)/\sqrt{\epsilon}))}^{2} +\Inner{\partial_{x}v_{\epsilon}(t),\partial_{x}\big(\abs{\partial_{x} u_{\epsilon}(t)}^{2}u_{\epsilon}(t)\big) }_{L^{2}(I((R-t)/\sqrt{\epsilon}))}  }\\
			\vs
			&\ds{\quad\quad -\epsilon\Inner{\partial_{x}v_{\epsilon}(t),\partial_{x}\big(\abs{v_{\epsilon}(t)}^{2}u_{\epsilon}(t)\big)}_{L^{2}(I((R-t)/\sqrt{\epsilon}))} +\frac 12 \sum_{i=1}^{\infty}\big\lvert \partial_{x}((u_{\epsilon}(t)\times v_{\epsilon}(t))\xi_{i}^{\epsilon})\big\rvert_{L^{2}(I((R-t)/\sqrt{\epsilon}))}^{2}    \Big)dt}\\
			\vs &\ds{\quad\quad\quad\quad\quad\quad\quad +\sqrt{\epsilon}\,Y_{\epsilon,\lambda}(t)\sigma_{1,\epsilon,I((R-t)/\sqrt{\epsilon})}(u_{\epsilon}(t),v_{\epsilon}(t))dw^{\epsilon}(t) }.
		\end{array}
	\end{equation*}
	For any interval $(a,b)\subset \mathbb{R}$, the Gagliardo-Nirenberg inequality gives
	\begin{equation}\label{GN_inequality}
		\abs{u}_{L^{\infty}(a,b)}^{2}\leq 2\abs{u}_{L^{2}(a,b)}\abs{\partial_x u}_{L^{2}(a,b)}+(b-a)^{-1}\abs{u}_{L^{2}(a,b)}^{2},\ \ \ \ u\in H^{1}(a,b).
	\end{equation}
	Thus, if we take $R>T+1$ and define $I:=I((R-t)/\sqrt{\epsilon})$, for  and $t\in [0,T]$ and $\e \in\,(0,1)$,  we get
	\begin{equation}
		\abs{u}_{L^{\infty}(I)}^{2}\leq 2\abs{u}_{L^{2}(I)}\abs{\partial_x u}_{L^{2}(I)}+\abs{u}_{L^{2}(I)}^{2},\ \ \ \ u\in H^{1}(I).
	\end{equation}
	Moreover, for every $(u,v)\in (H^{2}(I)\times H^{1}(I))\cap \mathcal{M}$, we have 
	\begin{equation*}
		\partial_x u\cdot v+u\cdot \partial_x v=\partial_x (u\cdot v) =0.
	\end{equation*}
		Hence,
	\begin{equation}
		\begin{array}{ll}
			&\ds{ \big\lvert\Inner{\partial_x v, \partial_x \big( \abs{\partial_x u}^{2}u\big) }_{L^{2}(I)} \big\rvert  \leq \big\lvert \Inner{\partial_x v,\abs{\partial_x u}^{2}\partial_x u}_{L^{2}(I)}\big\rvert +2\big\lvert \Inner{\partial_x v,(\partial_x u\cdot \partial_x^{2}u)u}_{L^{2}(I)} \big\rvert   }\\
			\vs
			&\ds{\quad \quad =\big\lvert \Inner{\partial_x v,\abs{\partial_x u}^{2}\partial_x u}_{L^{2}(I)}\big\rvert +2\big\lvert \Inner{v,(\partial_x u\cdot \partial_x^{2}u)\partial_x u}_{L^{2}(I)} \big\rvert    }\\
			\vs
			&\ds{\quad \quad \quad \quad \leq \abs{\partial_x u}_{L^{\infty}(I)}^{2}\abs{\partial_x u}_{L^{2}(I)}\abs{\partial_x v}_{L^{2}(I)} + 2\abs{\partial_x u}_{L^{\infty}(I)}^{2}\abs{\partial_x ^{2}u}_{L^{2}(I)}\abs{v}_{L^{2}(I)}   }\\
			\vs
			&\ds{\quad \quad \leq \Big(\abs{\partial_x u}_{L^{2}(I)}^{2}+2\abs{\partial_x u}_{L^{2}(I)}\abs{\partial_x ^{2}u}_{L^{2}(I)}\Big)\Big(\abs{\partial_x u}_{L^{2}(I)}\abs{\partial_x v}_{L^{2}(I)} + 2\,\abs{\partial_x ^{2}u}_{L^{2}(I)}\abs{v}_{L^{2}(I)}\Big)   }\\
			\vs
			&\ds{\quad \quad  = \abs{\partial_{x}u}_{L^{2}(I)}^{3}\abs{\partial_{x}v}_{L^{2}(I)}+2\,\abs{\partial_{x}u}_{L^{2}(I)}^{2}\abs{\partial_{x}^{2}u}_{L^{2}(I)}\big(\abs{\partial_{x}v}_{L^{2}(I)}+\abs{v}_{L^{2}(I)}\big)}\\
			\vs
			&\ds{\quad \quad \quad \quad \quad \quad \quad \quad \quad \quad +4\,\abs{\partial_{x}u}_{L^{2}(I)}\abs{\partial_{x}^{2}u}_{L^{2}(I)}^{2}\abs{v}_{L^{2}(I)}.   }
		\end{array}
	\end{equation}
	This implies that for any $\delta>0$ 	\begin{equation}\label{sm2}
		\begin{array}{ll}
			&\ds{\big\lvert\Inner{\partial_{x}v, \partial_{x}\big( \abs{\partial_{x} u}^{2}u\big) }_{L^{2}(I)} \big\rvert \leq \delta\abs{\partial_{x}v}_{L^{2}(I)}^{2}+c(\delta)\Big(\abs{\partial_{x}u}_{L^{2}(I)}^{6}+\abs{\partial_{x}u}_{L^{2}(I)}^{4}\abs{\partial_{x}^{2}u}_{L^{2}(I)}^{2} \Big) }\\
			\vs
			&\ds{\quad\quad\quad\quad +2\,\abs{\partial_{x}u}_{L^{2}(I)}^{2}\abs{\partial_{x}^{2}u}_{L^{2}(I)}\abs{v}_{L^{2}(I)}+ 4\,\abs{\partial_{x}u}_{L^{2}(I)}\abs{\partial_{x}^{2}u}_{L^{2}(I)}^{2}\abs{v}_{L^{2}(I)} }
		\end{array}
	\end{equation}
	Similarly, we have
	\begin{equation}\label{sm1}
		\begin{array}{ll}
			&\ds{ \big\lvert\Inner{\partial_x v, \partial_x (\abs{v}^{2}u) }_{L^{2}(I)} \big\rvert \leq \big\lvert\Inner{\partial_x v,\abs{v}^{2}\partial_x u}_{L^{2}(I)} \big\rvert +2\big\lvert \Inner{\partial_x v, (v\cdot \partial_x v)u}_{L^{2}(I)}   \big\rvert    }\\
			\vs
			&\ds{\quad =\big\lvert\Inner{\partial_x v,\abs{v}^{2}\partial_x u}_{L^{2}(I)} \big\rvert +2\big\lvert \Inner{v, (v\cdot \partial_x v)\partial_x u}_{L^{2}(I)}\big\rvert  }\\
			\vs
			&\ds{\quad\quad \leq 3\abs{v}_{L^{\infty}(I)}^{2}\abs{\partial_x u}_{L^{2}(I)}\abs{\partial_x v}_{L^{2}(I)} \leq 3\Big(\abs{v}_{L^{2}(I)}^{2}+2\abs{v}_{L^{2}(I)}\abs{\partial_x v}_{L^{2}(I)}\Big)\abs{\partial_x u}_{L^{2}(I)}\abs{\partial_x v}_{L^{2}(I)}.}
				\end{array}
	\end{equation}
	Moreover, as a consequence of \eqref{identity1} and \eqref{identity2}, we have 
	\[\sum_{i=1}^{\infty}\xi_{i}^{\epsilon}(x)(\xi_{i}^{\epsilon})'(x)=0,\ \ \ \ \  \ x\in\R,\ \ \ \epsilon>0,\] and this implies that 
	\begin{equation*}
		\begin{array}{ll}
			&\ds{ \sum_{i=1}^{\infty}\big\lvert \partial_x ((u\times v)\xi_{i}^{\epsilon})\big\lvert_{L^{2}(I)}^{2} = \sum_{i=1}^{\infty}\big\lvert \partial_x (u\times v)\xi_{i}^{\epsilon}+(u\times v)(\xi_{i}^{\epsilon})' \big\rvert_{L^{2}(I)}^{2} = \sum_{i=1}^{\infty}\big\lvert \partial_x (u\times v)\xi_{i}^{\epsilon}\big\rvert_{L^{2}(I)}^{2}  }\\
			\vs
			&\ds{\quad \quad \quad +\sum_{i=1}^{\infty}\big\lvert (u\times v)(\xi_{i}^{\epsilon})'\big\rvert_{L^{2}(I)}^{2} = \sum_{i=1}^{\infty}\Big(\big\lvert (\partial_x u\times v)\xi_{i}^{\epsilon}\big\rvert_{L^{2}(I)}^{2}+\big\lvert (u\times \partial_x v)\xi_{i}^{\epsilon}\big\rvert_{L^{2}(I)}^{2} \Big)}\\
			\vs
			&\ds{+ \sum_{i=1}^{\infty}\Big(  2\Inner{(\partial_x u\times v)\xi_{i}, (u\times \partial_x v)\xi_{i}^{\epsilon} }_{L^{2}(I)}+ \big\lvert (u\times v)(\xi_{i}^{\epsilon})'\big\rvert_{L^{2}(I)}^{2}\Big)   }\\
			\vs
			&\ds{\leq c_0 \abs{\partial_x u}_{L^{\infty}(I)}^{2}\abs{v}_{L^{2}(I)}^{2} +c_0\abs{\partial_x v}_{L^{2}(I)}^{2} +\frac{c_{1}}{\epsilon}\abs{v}_{L^{2}(I)}^{2}+2c_{0}\abs{\partial_x u}_{L^{\infty}(I)}\abs{\partial_x v}_{L^{2}(I)}\abs{v}_{L^{2}(I)}. }
		\end{array}
	\end{equation*}
	Then, thanks again to \eqref{GN_inequality}, for any $\delta>0$  we obtain
	\begin{equation}\label{sm3}
	\begin{array}{l}\ds{	\sum_{i=1}^{\infty}\big\lvert \partial_x ((u\times v)\xi_{i})\big\lvert_{L^{2}(I)}^{2}\leq \frac{c_{1}}{\epsilon}\abs{v}_{L^{2}(I)}^{2}+(c_0+\delta)\abs{\partial_{x}v}_{L^{2}(I)}^{2}}\\
	\vs	
	\ds{\quad \quad \quad \quad \quad +c(\delta)\abs{v}_{L^{2}(I)}^{2}\Big(\abs{\partial_{x}u}_{L^{2}(I)}^{2}+2\abs{\partial_{x}u}_{L^{2}(I)}\abs{\partial_{x}^{2}u}_{L^{2}(I)}  \Big).}
	\end{array}
	\end{equation}
	Furthermore, \eqref{GN_inequality} gives 	
	\begin{equation*}
		\begin{array}{ll}
			&\ds{\Vert \,\sigma_{1,\epsilon,I}(u,v)\Vert _{\L_{2}(K_{\epsilon},\R)}^{2} = \sum_{i=1}^{\infty}\Big\lvert \Inner{\partial_x v,\partial_x ((u\times v)\xi_{i}^{\epsilon})}_{L^{2}(I)} \Big\rvert^{2}\leq 3\sum_{i=1}\Big\lvert\Inner{\partial_x v,(\partial_x u\times v)\xi_{i}^{\epsilon}}_{L^{2}(I)}\Big\rvert^{2}}\\
			\vs
			&\ds{\quad \quad +3\sum_{i=1}\Big\lvert\Inner{\partial_x v,(u\times \partial_x v)\xi_{i}^{\epsilon}}_{L^{2}(I)}\Big\rvert^{2}+3\sum_{i=1}\Big\lvert\Inner{\partial_x v,(u\times v)(\xi_{i}^{\epsilon})'}_{L^{2}(I)}\Big\rvert^{2} }\\
			\vs
			&\ds{\quad \quad \quad \quad \quad \quad \leq 3\abs{\partial_x v}_{L^{2}(I)}^{2}\Big(c_0\abs{\partial_x u}_{L^{\infty}(I)}^{2}\abs{v}_{L^{2}(I)}^{2}+\frac{c_{1}}{\epsilon}\abs{v}_{L^{2}(I)}^{2} \Big) }\\
			\vs
			&\ds{\leq 3\abs{\partial_x v}_{L^{2}(I)}^{2}\Big(c_0\abs{\partial_x u}_{L^{2}(I)}^{2}\abs{v}_{L^{2}(I)}^{2}+2c_0\abs{\partial_x u}_{L^{2}(I)}\abs{\partial_x ^{2}u}_{L^{2}(I)}\abs{v}_{L^{2}(I)}^{2}+\frac{c_{1}}{\epsilon}\abs{v}_{L^{2}(I)}^{2}\Big),  }
		\end{array}
	\end{equation*}
	and, in view of \eqref{uniform_bound1_loc}, for any $\delta>0$ the Burkholder-Davis-Gundy inequality gives 
	\begin{equation}\label{sm4}
		\begin{array}{ll}
			&\ds{\sqrt{\epsilon}\ \E\sup_{r\in[0,t]}\Big\lvert\int_{0}^{r}\, Y_{\epsilon,\lambda}(s)\sigma_{1,\e, I(R-s)}(u(s),v(s))dw(s) \Big\rvert }\\
			\vs
			&\ds{\quad \quad \lesssim \sqrt{\epsilon}\ \E\Big(\int_{0}^{t}Y_{\epsilon,\lambda}^{2}(s) \Vert \,\sigma_{1, \epsilon, I((R-s)/\sqrt{\epsilon}))}(u(s),v(s))\Vert _{\L_{2}(K_{\epsilon},\mathbb{R})}^{2}ds\Big)^{\frac{1}{2}}}\\
			\vs
			&\ds{ \leq \epsilon\delta \ \E\sup_{r\in[0,t]}\Big(Y_{\epsilon,\lambda}(r)\abs{\partial_x v(r)}_{L^{2}(I((R-s)/\sqrt{\epsilon}))}^{2}\Big) + c(\delta)\E\int_{0}^{t}Y_{\epsilon,\lambda}(s)\Big(\frac{1}{\epsilon}+\abs{\partial_{x}u_{\epsilon}(s)}_{L^{2}(I((R-s)/\sqrt{\epsilon}))}^{2} }\\
			\vs 
			&\ds{ \quad \quad \quad \quad \quad \quad +\abs{\partial_{x}u_{\epsilon}(s)}_{L^{2}(I((R-s)/\sqrt{\epsilon}))}^{2}\abs{\partial_{x}^{2}u_{\epsilon}(s)}_{L^{2}(I((R-s)/\sqrt{\epsilon}))}^{2}\Big)\abs{v_{\epsilon}(s)}_{L^{2}(I((R-s)/\sqrt{\epsilon})}^{2}ds.  }
		\end{array}
	\end{equation}
	Hence, if we collect all these estimates, thanks to \eqref{uniform_bound1_loc}, we can find some large enough $\bar{\lambda}>0$, some small enough $\delta>0$  and $c>0$, all depending only on $\Lambda_{1}$, such that for every $\epsilon\in (0,1)$
		\begin{equation*}
		\begin{array}{ll}
			&\ds{ \E\sup_{t\in[0,T]}Y_{\epsilon,\bar{\lambda}}(t)\Big(\abs{\partial_{x}^{2}u_{\epsilon}(t)}_{L^{2}(I((R-t)/\sqrt{\epsilon}))}^{2}+\epsilon\abs{\partial_{x}v_{\epsilon}(t)}_{L^{2}(I((R-t)/\sqrt{\epsilon}))}^{2}\Big)   }\\
			\vs
			&\ds{\quad\quad +\E\int_{0}^{T}Y_{\epsilon,\bar{\lambda}}(s)\Big(\abs{\partial_{x}v_{\epsilon}(s)}_{L^{2}(I((R-s)/\sqrt{\epsilon}))}^{2}+\abs{\partial_{x}^{2}u_{\epsilon}(s)}_{L^{2}(I((R-s)/\sqrt{\epsilon}))}^{2}\abs{v_{\epsilon}(s)}_{L^{2}(I((R-s)/\sqrt{\epsilon}))}^{2}\Big) ds    }\\
			\vs 
			&\ds{\quad\quad\quad\quad \quad\quad \quad\quad \quad\quad \quad\quad \lesssim \Big(\abs{D^{2}u_{0}^{\epsilon}}_{L^{2}(\R)}^{2}+\epsilon\abs{D v_0^{\epsilon}}_{L^{2}(\R)}^{2} + \frac{1}{\epsilon}\Big) }
		\end{array}
	\end{equation*}
		Finally, since  for every $t\in [0,T]$ and $\eps\in(0,1)$
	\begin{equation*}
		Y_{\epsilon,\bar{\lambda}}(t)\geq \exp\Big(-\bar{\lambda}\big(t+\abs{D u_0^{\epsilon}}_{L^{2}(\mathbb{R})}^{2}+\epsilon\abs{v_0^{\epsilon}}_{L^{2}(\mathbb{R})}^{2}\big)\Big)\geq \exp\big(-\bar{\lambda}(\Lambda_{1}^{2}+T)\big)  ,
	\end{equation*}
	we complete the proof.
	
\end{proof}

As a consequence of \eqref{uniform_bound1_loc} and \eqref{uniform_bound2_loc_particular}, by taking the limit as $R\to+\infty$, the Fatou Lemma allows us to conclude that, for any fixed $\epsilon>0$, 
\begin{equation*}
	u_{\epsilon}\in L^{\infty}(0,T;\dot{H}^{2}(\R;\R^{3})),\ \ \ \ \partial_{t}u_{\epsilon}\in L^{\infty}(0,T;H^{1}(\R;\R^{3})).
\end{equation*}
This completes the proof of Theorem \ref{system_wellposedness}.

\begin{Lemma}
	For every $0\leq t\leq T$ and $\epsilon\in(0,1)$ we have
	\begin{equation}\label{uniform_bound1}
		\abs{\partial_x u_{\eps}(t)}_{L^{2}(\R)}^{2}+\eps\abs{v_{\eps}(t)}_{L^{2}(\R)}^{2} +2\gamma\int_{0}^{t}\abs{v_{\eps}(s)}_{L^{2}(\R)}^{2}ds = \abs{D u_0^{\epsilon}}_{L^{2}(\R)}^{2}+\eps\abs{v_0^{\epsilon}}_{L^{2}(\R)}^{2},\ \ \ \  \ \ \P\text{-a.s.}
	\end{equation}
	and 
		\begin{equation}\label{uniform_bound2}
		\begin{array}{ll}
			&\ds{\quad \quad \E\sup_{t\in[0,T]}\Big(\abs{\partial_x ^{2}u_{\epsilon}(t)}_{L^{2}(\R)}^{2}+\epsilon\abs{\partial_x v_{\epsilon}(t)}_{L^{2}(\R)}^{2}\Big)   }\\
			\vs
			&\ds{ +\E\int_{0}^{T}\Big(\abs{\partial_x v_{\epsilon}(s)}_{L^{2}(\R)}^{2}+\abs{\partial_x ^{2}u_{\epsilon}(s)}_{L^{2}(\R)}^{2}\abs{v_{\epsilon}(s)}_{L^{2}(\R)}^{2}\Big) ds \lesssim_{\,T,\Lambda_1}\abs{D^{2}u_{0}^{\epsilon}}_{L^{2}(\R)}^{2}+\epsilon\abs{D v_0^{\epsilon}}_{L^{2}(\R)}^{2}+\frac{1}{\epsilon}}.
		\end{array}
	\end{equation} 
\end{Lemma}

\begin{proof}
	Inequality \eqref{uniform_bound2} is a direct consequence of \eqref{uniform_bound2_loc_particular}. In order to prove  identity \eqref{uniform_bound1}, recalling that $(u_{\eps}(t),v_{\eps}(t))\in \mathcal{M}$, from the It\^{o} formula we get
	\begin{equation}\label{sm153}
		\begin{array}{ll}
			&\ds{  d\Big( \abs{\partial_x u_{\eps}(t)}_{L^{2}(\R)}^{2} +\eps\abs{v_{\eps}(t)}_{L^{2}(\R)}^{2} \Big) =2\Inner{\partial_x u_{\eps}(t),\partial_x v_{\eps}(t)}_{L^{2}(\R)}dt       }\\
			\vs
			&\ds{\quad\quad + 2\Inner{v_{\eps}(t), \partial_x^{\,2} u_{\eps}(t)}_{L^{2}(\R)}dt -2\,\gamma_0\abs{v_\eps(t)}_{L^2(\R)}^{2}dt +\sum_{i=1}^{\infty}\abs{ (u_{\eps}(t)\times v_{\eps}(t))\xi_{i}^{\eps} }_{L^{2}(\R)}^{2}dt. }
		\end{array}
	\end{equation}
	Since $\partial_x u_{\eps}(t)\cdot v_{\eps}(t)\in L^{1}(\R)$ and $\partial_x (\partial_x u_{\eps}(t)\cdot v_{\eps}(t) )\in L^{1}(\R)$, integrating by parts we get	
	\begin{equation*}
		\begin{array}{l}
			\ds{	\int_{\R}\partial_x^{\,2} u_{\eps}(t)\cdot v_{\eps}(t)\ dx = \int_{\R}\partial_x \big(\partial_x u_{\eps}(t)\cdot v_{\eps}(t)\big)dx - \int_{\R}\partial_x u_{\eps}(t)\cdot \partial_x v_{\eps}(t)\ dx}\\
			\vs
			\ds{\quad \quad \quad \quad \quad \quad \quad \quad \quad \quad \quad \quad = - \int_{\R}\partial_x u_{\eps}(t)\cdot \partial_x v_{\eps}(t)\ dx.}
		\end{array}
	\end{equation*}
	Hence, recalling that 
	\begin{equation*}
		\sum_{i=1}^{\infty}\abs{ (u_{\eps}(t)\times v_{\eps}(t))\xi_{i}^{\eps} }_{L^{2}(\R)}^{2} = c_0\abs{v_{\eps}(t)}_{L^{2}(\R)}^{2},
	\end{equation*}
from \eqref{sm153} we get
	\begin{equation*}
		d\Big( \abs{\partial_x u_{\eps}(t)}_{L^{2}(\R)}^{2} +\eps\abs{v_{\eps}(t)}_{L^{2}(\R)}^{2} \Big)   = -2\gamma\abs{v_\eps(t)}_{L^{2}(\R)}^{2}dt,
	\end{equation*}
	and this gives \eqref{uniform_bound1}.
\end{proof}

\subsection{Uniform estimate for $u_\e$ in $L^2(0,T;\dot{H}^2(\R))$}
\label{energy3}

	\begin{Lemma}\label{ito_loc}
		We fix $T>0$ and $R>T+1$. Then, for every $\eps\in(0,1)$, the following estimate  holds $\P$-almost surely 
		\begin{equation}\label{ito2}
			\begin{array}{ll}
				&\ds{ \eps\Inner{\partial_x u_{\eps}(t),\partial_x v_{\eps}(t)}_{L^{2}(I((R-t)/\sqrt{\eps}))}+\frac{\gamma_0}2\,\abs{\partial_x u_{\eps}(t)}_{L^{2}(I((R-t)/\sqrt{\eps}))}^{2} +\int_{0}^{t}\abs{\partial_x ^{2}u_{\eps}(s)}_{L^{2}(I((R-s)/\sqrt{\eps}))}^{2}ds  }\\
				\vs
				&\ds{\quad \quad \quad \leq  c\Big(\abs{D u_0^{\epsilon}}_{L^{2}(\mathbb{R})}^{2}+\epsilon\abs{D^{2}u_0^{\epsilon}}_{L^{2}(\R)}^{2}+\eps^{2}\abs{D v_0^{\epsilon}}_{L^{2}(\mathbb{R})}^{2}\Big)}\\
				\vs 
				&\ds{\quad \quad \quad \quad \quad \quad \quad \quad + \int_{0}^{t}\Inner{\partial_x u_{\eps}(s), \partial_x \big(\abs{\partial_x u_{\eps}(s)}^{2}u_{\eps}(s)\big)  }_{L^{2}(I((R-s)/\sqrt{\eps}))}ds }\\
				\vs
				&\ds{\quad\quad  -\eps\int_{0}^{t}\Inner{\partial_x u_{\eps}(s), \partial_x \big(\abs{v_{\eps}(s)}^{2}u_{\eps}(s)\big) }_{L^{2}(I((R-s)/\sqrt{\eps}))}ds + \eps\int_{0}^{t}\abs{\partial_x v_{\eps}(s)}_{L^{2}(I(R-s)/\sqrt{\eps}))}^{2}ds  }\\
				\vs
				&\ds{\quad\quad \quad\quad \quad\quad \quad\quad \quad\quad \quad\quad  +\sqrt{\eps}\int_{0}^{t}\,\vartheta_{1, I((R-s)/\sqrt{\eps})}\big(u_{\eps}(s),v_{\eps}(s)\big)dw^{\eps}(s),\ \ \ \ t\in[0,T] },
			\end{array}
		\end{equation}
		where, for any bounded interval $I\subset \mathbb{R}$ and $k \in\,\mathbb{N}$, we denote 
		\begin{equation}\label{sm175}
			\vartheta_{k,I}(u,v)\xi:= \Inner{\partial_x^{\,k}u, \partial_x^{\,k}((u\times v)\xi)}_{L^{2}(I)},\ \ \ \ (u,v)\in H^{k+1}(I)\times H^{k}(I),\ \ \ \ \xi\in K_{\epsilon}.
		\end{equation}
		
	\end{Lemma}

	\begin{proof}
		As a consequence of the It\^{o} formula applied to  
		\begin{equation*}
			\eps\Inner{\partial_x u_{\eps}(t), \partial_x v_{\eps}(t)}_{L^{2}(I((R-t)/\sqrt{\epsilon}))},\ \ \ \ t\in[0,T],
		\end{equation*}
		we have 
		\begin{equation}\label{ito2_step1}
			\begin{array}{ll}
				&\ds{ \epsilon d\Inner{\partial_x u_{\epsilon}(t),\partial_x v_{\epsilon}(t)}_{L^{2}(I((R-t)/\sqrt{\epsilon}))} = -\sqrt{\epsilon}\Inner{\partial_x u_{\epsilon}(t),\partial_x v_{\epsilon}(t)}_{L^{2}(\partial I((R-t)/\sqrt{\epsilon}))}dt }\\
				\vs
				&\ds{+\epsilon \abs{\partial_x v_{\epsilon}(t)}_{L^{2}(I((R-t)/\sqrt{\epsilon}))}^{2}dt  +\Inner{\partial_x u_{\epsilon}(t), \partial_x (\partial_x^{\,2} u_{\epsilon}(t)) }_{L^{2}(I((R-t)/\sqrt{\epsilon}))}dt  }\\
				\vs
				&\ds{+\Inner{\partial_x u_{\epsilon}(t),\partial_x \big(\abs{\partial_x u_{\epsilon}(t)}^{2}u_{\epsilon}(t)\big) }_{L^{2}(I((R-t)/\sqrt{\epsilon}))}dt-\epsilon\Inner{\partial_x u_{\epsilon}(t), \partial_x \big(\abs{v_{\epsilon}(t)}^{2}u_{\epsilon}(t)\big) }_{L^{2}(I((R-t)/\sqrt{\epsilon}))}dt}\\
				\vs
				&\ds{\quad\quad-\,\gamma_0\Inner{\partial_x u_{\epsilon}(t),\partial_x v_{\epsilon}(t) }_{L^{2}(I((R-t)/\sqrt{\epsilon}))}dt    +\sqrt{\epsilon}\,\vartheta_{1, I((R-t)/\sqrt{\epsilon})}(u_{\epsilon}(t),v_{\epsilon}(t))dw^{\epsilon}(t)}.
			\end{array}
		\end{equation}
		First of all, we notice that
		\begin{equation}\label{ito2_step2}
			\frac{d}{dt}\abs{\partial_x u_{\epsilon}(t)}_{L^{2}(I((R-t)/\sqrt{\epsilon}))}^{2}=-\frac{1}{\sqrt{\epsilon}}\abs{\partial_x u_{\epsilon}(t)}_{L^{2}(\partial I((R-t)/\sqrt{\epsilon}))}^{2}+2\Inner{\partial_x u_{\epsilon}(t),\partial_x v_{\epsilon}(t)}_{L^{2}(I((R-t)/\sqrt{\epsilon}))}.
		\end{equation}
		Next, since
		\begin{equation*}
			\begin{array}{ll}
				&\ds{ \frac{\sqrt{\epsilon}}{2}\frac{d}{dt}\abs{\partial_x u_{\epsilon}(t)}_{L^{2}(\partial I((R-t)/\sqrt{\epsilon}))}^{2}}\\
				\vs 
				&\ds{\quad \quad \quad =\sqrt{\epsilon}\,\Inner{\partial_x u_{\epsilon}(t),\partial_x v_{\epsilon}(t)}_{L^{2}(\partial I((R-t)/\sqrt{\epsilon}))}-\big(\partial_x u_{\epsilon}(t)\cdot \partial_x ^{2}u_{\epsilon}(t)\big)\big|_{-(R-t)/\sqrt{\eps}}^{(R-t)/\sqrt{\eps}},   }
			\end{array}
		\end{equation*}
		and 
		\begin{equation*}
			\begin{array}{ll}
				&\ds{\Inner{\partial_x u_{\epsilon}(t),\partial_x (\partial_x^{\,2} u_{\epsilon}(t)) }_{L^{2}(I((R-t)/\sqrt{\epsilon})}}\\
				\vs\
				&\ds{\quad \quad \quad\leq -\abs{\partial_x ^{2}u_{\epsilon}(t)}_{L^{2}(I((R-t)/\sqrt{\epsilon}))}^{2}+\big(\partial_x u_{\epsilon}(t)\cdot \partial_x ^{2}u_{\epsilon}(t)\big)\big|_{-(R-t)/\sqrt{\eps}}^{(R-t)/\sqrt{\eps}}  },
			\end{array}
		\end{equation*}
		we obtain 
		\begin{equation}\label{ito2_step3}
			\begin{array}{ll}
				&\ds{-\sqrt{\epsilon}\Inner{\partial_x u_{\epsilon}(t),\partial_x v_{\epsilon}(t)}_{L^{2}(\partial I((R-t)/\sqrt{\epsilon}))}+\Inner{\partial_x u_{\epsilon}(t),\partial_x (\partial_x^{\,2} u_{\epsilon}(t)) }_{L^{2}(I((R-t)/\sqrt{\epsilon})} }\\
				\vs
				&\ds{\quad\quad\quad\quad \leq -\frac{\sqrt{\epsilon}}{2}\frac{d}{dt}\abs{\partial_x u_{\epsilon}(t)}_{L^{2}(\partial I((R-t)/\sqrt{\epsilon}))}^{2}-\abs{\partial_x ^{2}u_{\epsilon}(t)}_{L^{2}(I((R-t)/\sqrt{\epsilon}))}^{2}   }.
			\end{array}
		\end{equation}
		Thus, by combining \eqref{ito2_step2} and \eqref{ito2_step3} with \eqref{ito2_step1}, we get 
		\begin{equation*}
			\begin{array}{ll}
				&\ds{d \Bigg( \epsilon\Inner{\partial_x u_{\epsilon}(t),\partial_x v_{\epsilon}(t)}_{L^{2}(I((R-t)/\sqrt{\epsilon}))} +\frac{\sqrt{\epsilon}}{2}\abs{\partial_x u_{\epsilon}(t)}_{L^{2}(\partial I((R-t)/\sqrt{\epsilon}))}^{2}+\frac{\,\gamma_0}2\abs{\partial_x u_{\epsilon}(t)}_{L^{2}(I((R-t)/\sqrt{\epsilon}))}^{2}  \Bigg)   }\\
				\vs
				&\ds{ \leq-\frac{\gamma_0}{2\sqrt{\epsilon}}\,\abs{\partial_x u_{\epsilon}(t)}_{L^{2}(\partial I((R-t)/\sqrt{\epsilon}))}^{2}dt-\abs{\partial_x ^{2}u_{\epsilon}(t)}_{L^{2}(I((R-t)/\sqrt{\epsilon}))}^{2}dt +\epsilon \abs{\partial_x v_{\epsilon}(t)}_{L^{2}(I((R-t)/\sqrt{\epsilon}))}^{2}dt   }\\
				\vs
				&\ds{  +\Inner{\partial_x u_{\epsilon}(t),\partial_x \big(\abs{\partial_x u_{\epsilon}(t)}^{2}u_{\epsilon}(t)\big) }_{L^{2}(I((R-t)/\sqrt{\epsilon}))}dt -\epsilon\Inner{\partial_x u_{\epsilon}(t), \partial_x \big(\abs{v_{\epsilon}(t)}^{2}u_{\epsilon}(t)\big) }_{L^{2}(I((R-t)/\sqrt{\epsilon}))}dt }\\
				\vs
				&\ds{\quad\quad \quad\quad \quad\quad + \sqrt{\epsilon}\,\vartheta_{1, I((R-t)/\sqrt{\epsilon})}(u_{\epsilon}(t),v_{\epsilon}(t))dw^{\epsilon}(t) }.
			\end{array}
		\end{equation*}
		In particular, this implies that 
		\begin{equation*}
			\begin{array}{ll}
				&\ds{ \eps\Inner{\partial_x u_{\eps}(t),\partial_x v_{\eps}(t)}_{L^{2}(I((R-t)/\sqrt{\eps}))}+\frac{\gamma_0}2\abs{\partial_x u_{\eps}(t)}_{L^{2}(I((R-t)/\sqrt{\eps}))}^{2} +\int_{0}^{t}\abs{\partial_x ^{2}u_{\eps}(s)}_{L^{2}(I((R-s)/\sqrt{\eps}))}^{2}ds  }\\
				\vs
				&\ds{\leq  c\Big(\eps\Inner{Du_0^{\epsilon},D v_0^{\epsilon}}_{L^{2}(I(R/\sqrt{\eps}))}+\abs{Du_0^{\epsilon}}_{L^{2}(I(R/\sqrt{\eps}))}^{2}+\sqrt{\eps}\abs{D u_0^{\epsilon}}_{L^{2}(\partial I(R/\sqrt{\eps}))}^{2}\Big)}\\
				\vs 
				&\ds{+ \eps\int_{0}^{t}\abs{\partial_x v_{\eps}(s)}_{L^{2}(I((R-s)/\sqrt{\eps}))}^{2}ds +\int_{0}^{t}\Inner{\partial_x u_{\eps}(s), \partial_x \big(\abs{\partial_x u_{\eps}(s)}^{2}u_{\eps}(s)\big)  }_{L^{2}(I((R-s)/\sqrt{\eps}))}ds}\\
				\vs
				&\ds{-\eps\!\!\int_{0}^{t}\Inner{\partial_x u_{\eps}(s), \partial_x \big(\abs{v_{\eps}(s)}^{2}u_{\eps}(s)\big) }_{L^{2}(I((R-s)/\sqrt{\eps}))}ds    +\sqrt{\eps}\!\!\int_{0}^{t}\,\vartheta_{1, I((R-s)/\sqrt{\eps})}\big(u_{\eps}(s),v_{\eps}(s)\big)dw^{\eps}(s) }.
			\end{array}
		\end{equation*}
       Finally, thanks to \eqref{GN_inequality}, and due to the fact that $R>T+1$, we obtain
		\begin{equation*}
			\begin{array}{l}
				\ds{\abs{D u_0^{\epsilon}}_{L^{2}(\partial I(R/\sqrt{\eps}))}^{2}\leq 2\abs{D u_0^{\epsilon}}_{L^{\infty}(I(R/\sqrt{\epsilon}))}^{2}\lesssim \sqrt{\e}\,\abs{D u_0^{\epsilon}}_{L^{2}(I(R/\sqrt{\epsilon}))}^{2}}\\
				\vs
				\ds{\quad \quad +\abs{D u_0^{\epsilon}}_{L^{2}(I(R/\sqrt{\epsilon}))}\abs{D^{2}u_0^{\epsilon}}_{L^{2}(I(R/\sqrt{\epsilon}))}\lesssim \sqrt{\epsilon}\,\abs{D^2 u_0^{\epsilon}}_{L^{2}(I(R/\sqrt{\epsilon}))}^{2}+\frac 1{\sqrt{\e}}\,\abs{D u_0^{\epsilon}}_{L^{2}(I(R/\sqrt{\epsilon}))}^{2}.}
			\end{array}
		\end{equation*}
		Therefore, recalling Hypothesis \ref{H1-bis}, we conclude that \eqref{ito2} holds.
		
	\end{proof}

	\begin{Lemma}
		For every $T>0$, there exists some constant $c_{T}>0$, depending only on $\Lambda_{1},\Lambda_{2}$ and $T>0$, such that for every $\epsilon\in (0,1)$ and $R>T+1$ 
		\begin{equation}\label{uniform_bound3_loc}
			\E\int_{0}^{T}\abs{\partial_x ^{2}u_{\epsilon}(s)}_{L^{2}(I((R-s)/\sqrt{\epsilon}))}^{2}ds \leq c_{T}.
		\end{equation}
		In particular, 
			\begin{equation}\label{uniform_bound4_loc}
			\sup_{R>1}\,	\E\int_{0}^{T}\abs{\partial_x ^{2}u_{\epsilon}(s)}_{L^{2}(-R,R)}^{2}ds\leq c_{T},\ \ \ \ \epsilon\in(0,1).
		\end{equation}
	\end{Lemma}

	\begin{proof}
	By Lemma \ref{ito_loc}, for every $\eps\in(0,1)$ and $t\in[0,T]$
 we have 	\begin{equation}
		\begin{array}{ll}
			&\ds{ \eps\Inner{\partial_{x}u_{\eps}(t),\partial_x v_{\eps}(t)}_{L^{2}(I((R-t)/\sqrt{\eps}))}+\frac{\gamma_0}{2}\abs{\partial_x u_{\eps}(t)}_{L^{2}(I((T-t)/\sqrt{\eps}))}^{2} +\int_{0}^{t}\abs{\partial_{x}^{2}u_{\eps}(s)}_{L^{2}(I((T-s)/\sqrt{\eps}))}^{2}ds  }\\
			\vs
			&\ds{\quad \quad \leq  c\Big(\abs{Du_0^{\epsilon}}_{L^{2}(\mathbb{R})}^{2}+\epsilon\abs{D^{2}u_0^{\epsilon}}_{L^{2}(\R)}^{2}+\eps^{2}\abs{Dv_0^{\epsilon}}_{L^{2}(\mathbb{R})}^{2}\Big)}\\
			\vs  &\ds{\quad\quad  \quad\quad  \quad\quad  \quad\quad  + \int_{0}^{t}\Inner{\partial_{x}u_{\eps}(s), \partial_{x}\big(\abs{\partial_x  u_{\eps}(s)}^{2}u_{\eps}(s)\big)  }_{L^{2}(I((R-s)/\sqrt{\eps}))}ds }\\
			\vs
			&\ds{\quad\quad \quad  -\eps\int_{0}^{t}\Inner{\partial_{x}u_{\eps}(s), \partial_{x}\big(\abs{v_{\eps}(s)}^{2}u_{\eps}(s)\big) }_{L^{2}(I((R-s)/\sqrt{\eps}))}ds + \eps\int_{0}^{t}\abs{\partial_{x}v_{\eps}(s)}_{L^{2}(I((R-s)/\sqrt{\eps}))}^{2}ds  }\\
			\vs
			&\ds{\quad\quad  \quad\quad  \quad\quad  \quad\quad  \quad\quad +\sqrt{\eps}\int_{0}^{t}\vartheta_{1,I((R-s)/\sqrt{\eps})}\big(u_{\eps}(s),v_{\eps}(s)\big)dw^{\eps}(s),\ \ \ \ \P\text{-a.s.} }.
		\end{array}
	\end{equation}
	As we did in the proof for \eqref{uniform_bound2_loc}, we take $R>T+1$ and $t\in [0,T]$, and define $I:=I((R-t)/\sqrt{\epsilon})$. Note that, since $u\cdot \partial_x u=0$, we have
	\begin{equation*}
		\Inner{\partial_x u,\partial_x \big(\abs{\partial_x u}^{2}u\big)}_{L^{2}(I)} = \Inner{\partial_x u,\abs{\partial_x u}^{2}\partial_x u}_{L^{2}(I)}.
	\end{equation*}
	Hence, thanks again to \eqref{GN_inequality}, for any $\delta>0$ 
	\begin{equation*}
		\begin{array}{ll}
			&\ds{\Inner{\partial_{x}u,\partial_x \big(\abs{\partial_{x} u}^{2}u\big)}_{L^{2}(I)}  \leq \abs{\partial_{x}u}_{L^{\infty}(I)}^{2}\abs{\partial_{x}u}_{L^{2}(I)}^{2}   \leq \Big( \abs{\partial_{x}u}_{L^{2}(I)}^{2}+2\abs{\partial_{x}^{2}u}_{L^{2}(I)}\abs{\partial_{x}u}_{L^{2}(I)}  \Big)\abs{\partial_{x}u}_{L^{2}(I)}^{2}  }\\
			\vs
			&\ds{\quad\quad\quad\quad\quad\quad  \leq \delta\abs{\partial_{x}^{2}u}_{L^{2}(I)}^{2}+c(\delta)\Big(\abs{\partial_{x}u}_{L^{2}(I)}^{4}+\abs{\partial_{x}u}_{L^{2}(I)}^{6}\Big)  }.
		\end{array}
	\end{equation*}
	In the same way, 
	\begin{equation*}
		\begin{array}{ll}
			&\ds{ \epsilon\big\lvert \Inner{\partial_{x}u, \partial_{x}\big(\abs{v}^{2}u\big) }_{L^{2}(I)}  \big\rvert = \epsilon\big\lvert \Inner{\partial_{x}u, \abs{v}^{2}\partial_{x}u }_{L^{2}(I)} \big\rvert   \leq \epsilon\abs{\partial_{x}u}_{L^{\infty}(I)}^{2}\abs{v}_{L^{2}(I)}^{2}    }\\
			\vs
			&\ds{\quad\quad\quad\quad \leq \epsilon\Big( \abs{\partial_{x}u}_{L^{2}(I)}^{2}+2\abs{\partial_{x}^{2}u}_{L^{2}(I)}\abs{\partial_{x}u}_{L^{2}(I)}  \Big)\abs{v}_{L^{2}(I)}^{2} }.
		\end{array}
	\end{equation*}
	Moreover, since 
	\begin{equation*}
		\begin{array}{ll}
			&\ds{ \norm{\vartheta_{1,I}(u,v)}_{\L_{2}(K_{\epsilon},\mathbb{R})}^{2} = \sum_{i=1}^{\infty}\Big\lvert \Inner{ \partial_{x}u, \partial_{x}\big((u\times v)\xi_{i}^{\epsilon}\big)}_{L^{2}(I)}\Big\rvert^{2}   }\\
			\vs
			&\ds{\quad \quad \quad \quad  \leq 3\sum_{i=1}^{\infty}\Big\lvert\Inner{\partial_{x}u,(u\times \partial_{x}v)\xi_{i}^{\epsilon} }_{L^{2}(I)} \Big\rvert^{2} + 3\sum_{i=1}^{\infty}\Big\lvert\Inner{\partial_{x}u,(u\times v)(\xi_{i}^{\epsilon})' }_{L^{2}(I)} \Big\rvert^{2}  }\\
			\vs
			&\ds{\quad \quad \quad \quad \quad \quad \quad \quad \leq 3c_{0}\abs{\partial_{x}u}_{L^{2}(I)}^{2}\abs{\partial_{x}v}_{L^{2}(I)}^{2}+\frac{3c_{1}}{\epsilon}\abs{\partial_{x}u}_{L^{2}(I)}^{2}\abs{v}_{L^{2}(I)}^{2} },
		\end{array}
	\end{equation*}
	 from \eqref{uniform_bound1_loc} and \eqref{uniform_bound2_loc} we have
	\begin{equation*}
		\begin{array}{ll}
			&\ds{ \sqrt{\epsilon}\ \E\sup_{t\in[0,T]}\Big\lvert \int_{0}^{t}\vartheta_{1,I((R-s)/\sqrt{\epsilon})}(u_{\epsilon}(s),v_{\epsilon}(s))dw^{\epsilon}(s)\Big\rvert  }\\
			\vs 
			&\ds{\quad \quad \quad \quad \quad  \lesssim \sqrt{\epsilon} \ \E\Big(\int_{0}^{T}\norm{\vartheta_{1,I((R-s)/\sqrt{\epsilon})}(u_{\epsilon}(s),v_{\epsilon}(s))}_{\L_{2}(K_{\epsilon},\mathbb{R})}^{2}ds\Big)^{\frac{1}{2}}  }\\
			\vs
			&\ds{\lesssim \Bigg( \E\int_{0}^{T}\abs{\partial_{x}u_{\epsilon}(s)}_{L^{2}(I((R-s)/\sqrt{\epsilon}))}^{2}\Big(\epsilon \abs{\partial_{x}v_{\epsilon}(s)}_{L^{2}(I((R-s)/\sqrt{\epsilon}))}^{2}+\abs{v_{\epsilon}(s)}_{L^{2}(I((R-s)/\sqrt{\epsilon}))}^{2}\Big)ds  \Bigg)^{\frac{1}{2}}\leq c_{T} }.
		\end{array}
	\end{equation*}
	Therefore, in view \eqref{uniform_bound1_loc} and \eqref{uniform_bound2_loc}, if we pick $\delta>0$ small enough we have
	\begin{equation*}
		\begin{array}{ll}
			&\ds{ \E\sup_{t\in [0,T]}\abs{\partial_{x}u_{\epsilon}(t)}_{L^{2}(I((R-t)/\sqrt{\epsilon}))}^{2}    +\E\int_{0}^{T}\abs{\partial_{x}^{2}u_{\epsilon}(s)}_{L^{2}(I((R-s)/\sqrt{\epsilon}))}^{2}ds }\\
			\vs
			&\ds{\quad \quad \quad   \lesssim \Big(\abs{Du_0^{\epsilon}}_{L^{2}(\mathbb{R})}^{2}+\epsilon\abs{D^{2}u_0^{\epsilon}}_{L^{2}(\R)}^{2}+\eps^{2}\abs{Dv_0^{\epsilon}}_{L^{2}(\mathbb{R})}^{2}\Big) +\E\sup_{t\in[0,T]}\abs{\partial_{x}u_{\epsilon}(t)}_{L^{2}(I((R-t)/\sqrt{\epsilon}))}^{2} }\\
			\vs
			&\ds{\quad +\epsilon^{2}\,\E\sup_{t\in[0,T]}\abs{\partial_{x}v_{\epsilon}(t)}_{L^{2}(I((R-t)/\sqrt{\epsilon}))}^{2} + \epsilon\, \E\int_{0}^{T}\abs{\partial_{x}v_{\epsilon}(s)}_{L^{2}(I((R-s)/\sqrt{\epsilon}))}^{2}ds  }\\
			\vs
			&\ds{\quad +\E\int_{0}^{T}\Big(\abs{\partial_{x}u_{\epsilon}(s)}_{L^{2}(I((R-s)/\sqrt{\epsilon}))}^{4}+\abs{\partial_{x}u_{\epsilon}(s)}_{L^{2}(I((R-s)/\sqrt{\epsilon}))}^{6}\Big)\,ds}\\
			\vs 
			&\ds{\quad\quad\quad \quad\quad\quad \quad\quad\quad +\E\int_{0}^{T}\Big(\abs{\partial_{x}^{2}u_{\epsilon}(s)}_{L^{2}(I((R-s)/\sqrt{\epsilon}))}^{2}\abs{v_{\epsilon}(s)}_{L^{2}(I((R-s)/\sqrt{\epsilon}))}^{2}\Big)ds  }\\
			\vs
			&\ds{ \quad\quad\quad + \sqrt{\epsilon}\ \E\sup_{t\in[0,T]}\Big\lvert \int_{0}^{t}\vartheta_{1,I((R-s)/\sqrt{\epsilon})}(u_{\epsilon}(s),v_{\epsilon}(s))dw^{\epsilon}(s)\Big\rvert}\\
			\vs
			&\ds{\quad \quad\quad\quad \quad\quad\quad \quad\quad  \lesssim_{T} 1+\epsilon\abs{D^{2}u_{0}^{\epsilon}}_{L^{2}(\mathbb{R})}^{2}+\abs{Du_0^{\epsilon}}_{L^{2}(\mathbb{R})}^{2}+\epsilon^{2}\abs{Dv_0^{\epsilon}}_{L^{2}(\mathbb{R})}^{2}  },
		\end{array}
	\end{equation*}
	and \eqref{uniform_bound3_loc} follows.
	
	\end{proof}
	
	As a consequence of the previous Lemma, we obtain the following fundamental result.

\begin{Lemma}
Under Hypotheses \ref{H1} and  \ref{H1-bis}, for every $T>0$, there exists some constant $c_{T}>0$, depending on $\Lambda_{1},\Lambda_{2}$ and $T>0$ only, such that 
	\begin{equation}\label{uniform_bound3}
	\sup_{\e \in\,(0,1)}	\E\int_{0}^{T}\abs{\partial_x ^{2}u_{\epsilon}(s)}_{L^{2}(\R)}^{2}ds \leq c_{T}.
	\end{equation}
\end{Lemma}

\section{The convergence of $u_\e$ - Proof of Theorem \ref{small_mass_limit}}\label{main1}

We proceed in  two steps. First we prove the tightness of $\{\L(u_{\epsilon})\}_{\epsilon\in(0,1)}$ in a proper functional space, and then we prove the validity of limit \eqref{convergence_pr}. 
\subsection{Tightness}
We recall that for any $1\leq p\leq +\infty$ and $\delta\geq 0$ we say that $f_n$ converges to $f$ in $L^{p}(0,T; H^{\delta}_{\text{loc}}(\mathbb{R}) )$, as $n\to\infty$,
if and only if for every $R>0$ 
\begin{equation*}
	\lim_{n\to\infty}f_n= f\ \ \ \ \text{in}\ \ L^{p}(0,T;H^{\delta}(-R,R)).
\end{equation*}

\begin{Lemma}\label{tightness}
	Assume that Hypotheses \ref{H1} and \ref{H1-bis} hold. Then, for every $T>0$ and for every  $\delta_{1}<1$ and $\delta_{2}<2$, the family of probability measures $\{\L(u_{\epsilon})\}_{\epsilon\in(0,1)}$ is tight in the space $C([0,T];H^{\delta_{1}}_{\text{loc}}(\R))\cap L^{2}(0,T;H^{\delta_{2}}_{\text{loc}}(\mathbb{R}))$.
\end{Lemma}

\begin{proof}
	We will first show that the family $\{\L(u_{\epsilon})\}_{\epsilon\in (0,1)}$ is tight in $L^{2}(0,T;H^{\delta_{2}}_{\text{loc}}(\mathbb{R}))$, for any $\delta_{2}<2$.  By \eqref{uniform_bound1} and \eqref{uniform_bound3}, for every $T>0$  we have 
	\begin{equation}
		\sup_{\epsilon \in\,(0,1)}\E\int_{0}^{T}\Big(\abs{\partial_x ^{2}u_{\epsilon}(t)}_{L^{2}(\R)}^{2}+\abs{\partial_{t} u_{\epsilon}(t)}_{L^{2}(\R)}^{2}\Big)dt\leq c_{T}.
	\end{equation}
Thus, for every fixed $\eta>0$ and $n \in\,\mathbb{N}$,  there exists $L_{\eta,n}>0$ such that if we denote by $B_{\eta,n}$ the closed ball of radius $L_{\eta,n}$ in 
	$L^{2}(0,T;H^{2}(-n,n))\cap W^{1,2}(0,T;L^{2}(-n,n))$, then
	\begin{equation*}
		\inf_{\epsilon \in\,(0,1)}\P\Big((u_{\epsilon})_{|_{[0,T]\times [-n,n]}}\in B_{\eta,n} \Big)\geq 1-\frac{\eta}{2^n}.
	\end{equation*}
Now, we define 	
\[K_\eta:=\bigcap_{n \in\,\mathbb{N}}K_{\eta,n},\]
where
\[K_{\eta, n}:=\left\{f \in\,L^2(0,T;H^{\delta_2}_{\text{loc}}(\R))\ :\ f_{|_{[0,T]\times [-n,n]}} \in\,B_{\eta, n}\right\}.\]
Since the Aubin-Lions lemma implies the compactness of   $K_{\eta,n}$  in $L^{2}(0,T;H^{\delta_{2}}_{\text{loc}}(\R))$, for any $\delta_{2}<2$ and $n \in\,\mathbb{N}$, it is immediate to check  that $K_\eta$ is compact in $L^2_{\text{loc}}(0,T;H^{\delta_2}(\R))$, for any $\delta_{2}<2$.
Moreover
\[\begin{array}{l}
\ds{\mathbb{P}\left(u_\e \in\,K_\eta\right)=\mathbb{P}\Big(u_\e \in\,\bigcap_{n=1}^\infty K_{\eta,n}\Big)=1-\mathbb{P}\Big(u_\e \in\,\bigcup_{n=1}^\infty K^c_{\eta,n}\Big)}\\
\vs
\ds{\quad\quad \quad \quad \quad \quad \quad \quad \geq 1-\sum_{n=1}^\infty \mathbb{P}\left((u_{\epsilon})_{|_{[0,T]\times [-n,n]}} \in\,B_{\eta,n}^c\right)\geq 1-\sum_{n=1}^\infty \frac{\eta}{2^n}=1-\eta.}	
\end{array}\]
 All this implies that $\{\L(u_{\epsilon})\}_{\epsilon\in (0,1)}$ is tight in  $L^{2}(0,T;H^{\delta_{2}}_{\text{loc}}(\mathbb{R}))$. Furthermore, by using \eqref{uniform_bound1} again, we have for every $T>0$,
	\begin{equation*}
		\sup_{\epsilon\in(0,1)}\E\sup_{t\in[0,T]}\abs{\partial_x u_{\eps}(t)}_{L^{2}(\R)}^{2}<\infty,
	\end{equation*}
so that	the tightness of $\{\L(u_{\eps})\}_{\epsilon\in(0,1)}$ in $C([0,T];H^{\delta_{1}}_{\text{loc}}(\mathbb{R}))$, for every $\delta_1<1$, can be proved by using  arguments analogous to those used  above.
\end{proof}

\subsection{Proof of Theorem \ref{small_mass_limit}}
	According to \eqref{sm130}, \eqref{uniform_bound1} and Lemma \ref{tightness},  we have that the family $\big\{\L(u_{\epsilon},\epsilon\partial_{t}u_{\epsilon} ,\eps\, w^{\eps})\big\}_{\epsilon\in(0,1)}$ is tight in
	\begin{equation}
		\Gamma_{T,\delta_{1},\delta_{2},s}:=\Big(C([0,T];H^{\delta_{1}}_{\text{loc}}(\R))\cap L^{2}(0,T;H^{\delta_{2}}_{\text{loc}}(\mathbb{R}))\Big)\times C([0,T];L_{\text{loc}}^{2}(\mathbb{R}))\times C([0,T];H^{s}_{\lambda}),
	\end{equation}
	 for any $T>0$, $\delta_{1}<1$, $\delta_{2}<2$ and $s\in[0,1)$. Then, the Skorokhod Theorem implies that for any sequence $\{\epsilon_{n}\}_{n\in\mathbb{N}}\subset(0,1)$ convergent to zero, there exists a subsequence, still denoted by $\{\epsilon_{n}\}$, and a sequence of  $\Gamma_{T,\delta_{1},\delta_{2},s}$-valued random variables
	\begin{equation*}
		\mathcal{Y}_{n}:=\big(\rho_{n}, \eps_{n}\vartheta_{n}, \eps_{n}w_{n}\big),\ \ \ \ \mathcal{Y}:=(\rho, 0, 0),\ \ \ \ n\in\mathbb{N},
	\end{equation*}
	all defined on some probability space $\big(\hat{\Omega},\hat{\F},\{\hat{\F}_{t}\}_{t\in[0,T]},\hat{\P}\big)$, such that 
	\begin{equation}
		\L(\mathcal{Y}_{n})=\L\big(u_{\eps_{n}}, \eps_{n}\partial_{t}u_{\eps_{n}}, \eps_{n}w^{\eps_{n}}\big),\ \ \ \ n\in\mathbb{N},
	\end{equation}
	and
	\[\lim_{n\to\infty}\vert \mathcal{Y}_n-\mathcal{Y}\vert_{\Gamma_{T,\delta_{1},\delta_{2},s}}=0,\ \ \ \ \ \ \hat{\mathbb{P}}\text{-a.s.}\]
	In particular, this means that 
	\begin{equation}
		\lim_{n\to\infty}\Big( \abs{\rho_n-\rho}_{C([0,T];H^{\delta_{1}}_{\text{loc}}(\R))}+\abs{\rho_{n}-\rho}_{L^{2}(0,T;H^{\delta_{2}}_{\text{loc}}(\R))} +\eps_{n}\abs{\vartheta_{n}}_{C([0,T];L^{2}_{\text{loc}}(\R))} + \eps_{n}\abs{w_{n}}_{C([0,T];H^{s}_{\lambda})}  \Big) = 0,
	\end{equation}
$\hat{\mathbb{P}}$-almost surely, or, equivalently, for any $R>0$
	\begin{equation}
		\lim_{n\to\infty}\Big(\abs{\rho_n - \rho}_{C([0,T];H^{\delta_{1}}(-R,R))}+\abs{\rho_{n}-\rho}_{L^{2}(0,T;H^{\delta_{2}}(-R,R))}+\eps_{n}\abs{\vartheta_{n}}_{C([0,T];L^{2}(-R,R))}\Big)=0,\ \ \ \ \hat{\P}\text{-a.s.}
	\end{equation}
	Note that in view of \eqref{uniform_bound1}, there exists a constant $c>0$ such that for any $T>0$ 
	\begin{equation}
		\sup_{n\in\mathbb{N}}\Bigg(\sup_{t\in[0,T]}\abs{\partial_x \rho_{n}(t)}_{L^{2}(\R)}^{2} + \eps_{n}\sup_{t\in[0,T]}\abs{\vartheta_{n}(t)}_{L^{2}(\R)}^{2} + \int_{0}^{T}\abs{\vartheta_{n}(t)}_{L^{2}(\R)}^{2}dt \Bigg)\leq c,\ \ \ \ \hat{\P}\text{-a.s.},
	\end{equation}
	and by \eqref{uniform_bound3}, for any $T>0$ 	
	\begin{equation}
		\sup_{n \in\,\mathbb{N}}\hat{\E}\int_{0}^{T}\abs{\partial_x ^{2}\rho_{n}(t)}_{L^{2}(\R)}^{2}dt <\infty.
	\end{equation}
	Hence, by the Fatou lemma, we conclude that $\rho\in  L^{\infty}(0,T;\dot{H}^{1}(\mathbb{R}))\cap L^{2}(0,T;\dot{H}^{2}(\mathbb{R})) $, $\hat{\P}$-a.s., and $\rho$ is weakly differentiable in time, with $\partial_{t}\rho\in L^{2}(0,T;L^{2}(\R))$. Moreover,
	\begin{equation}
		\sup_{t\in[0,T]}\abs{\partial_x \rho(t)}_{L^{2}(\R)} +\int_{0}^{T}\abs{\partial_{t}\rho(t)}_{L^{2}(\R)}^{2}dt\leq c,\ \ \ \ \hat{\P}\text{-a.s.},
	\end{equation}
	and 
	\begin{equation*}
		\hat{\E} \int_{0}^{T}\abs{\partial_x ^{2}\rho(t)}_{L^{2}(\R)}^{2}dt \leq c_{T}.
	\end{equation*}

Now, we  prove that for every $t\in[0,T]$ and every $\psi\in C^{\infty}_{c}(\R)$, the following identity  holds, $\hat{\P}$-almost surely, 
\begin{equation}
	\label{test_limit}
	\,\gamma_0\Inner{\rho(t),\psi}_{L^{2}(\R)} = \,\gamma_0\Inner{u_0,\psi}_{L^{2}(\R)}+\int_{0}^{t}\Inner{\partial_x^{\,2}\rho(s),\psi}_{L^{2}(\R)}ds + \int_{0}^{t}\Inner{\abs{\partial_x\rho(s)}^{2}\rho(s),\psi }_{L^{2}(\R)}ds.
\end{equation}
Indeed, for every $n \in\,\mathbb{N}$, $t\in[0,T]$ and $\psi\in C^{1}_{c}(\mathbb{R})$, we have 
\begin{equation}
	\begin{array}{ll}
		&\ds{ \Inner{\,\gamma_0\,\rho_{n}(t) + \eps_{n}\vartheta_{n}(t), \psi}_{L^{2}(\R)} = \Inner{\,\gamma_0\,u_{0}^{\epsilon_n}+\eps_{n}v_0^{\epsilon_n},\psi}_{L^{2}(\R)}+ \int_{0}^{t}\Inner{\partial_x^{\,2}\rho_{n}(s),\psi}_{L^{2}(\R)}ds   }\\
		\vs
		&\ds{\quad\quad\quad + \int_{0}^{t}\Inner{\abs{\partial_x\rho_{n}(s)}^{2}\rho_{n}(s),\psi}_{L^{2}(\mathbb{R})}ds -\eps_{n}\int_{0}^{t}\Inner{\abs{\vartheta_{n}(s)}^{2}\rho_{n}(s),\psi}_{L^{2}(\R)}ds   }\\
		\vs
		&\ds{\quad\quad\quad\quad \quad\quad\quad  \quad\quad\quad  + \sqrt{\eps_{n}}\int_{0}^{t}\Inner{\psi,\big(\rho_{n}(s)\times \vartheta_{n}(s)\big)dw_{n}(s)}_{L^{2}(\R)},\ \ \ \ \hat{\P}\text{-a.s.} }
	\end{array}
\end{equation}
First, note that for every $R>0$
\begin{equation}
	\lim_{n\to\infty} \gamma_0\,\rho_{n}+\eps_{n}\vartheta_{n} = \,\gamma_0\,\rho,\ \ \text{in}\ \ C([0,T];L^{2}(-R,R)),\ \ \hat{\P}\text{-a.s.},
\end{equation}
which implies that 
\begin{equation}\label{convergence_step1}
	\lim_{n\to\infty}\sup_{t\in[0,T]}\Big\lvert \Inner{\,\gamma_0\,(\rho_{n}(t)-\rho(t))+\eps_{n}\vartheta_{n}(t),\psi}_{L^{2}(\R)} \Big\rvert = 0, \ \ \ \ \hat{\P}\text{-a.s.}
\end{equation}
Since 
\begin{equation*}
	\sup_{t\in[0,T]}\Big\lvert \int_{0}^{t}\Inner{ \partial_x^{\,2}(\rho_{n}(s)-\rho(s)),\psi}_{L^{2}(\R)}ds\Big\rvert\leq \abs{\psi}_{H^{1}(\mathbb{R})}\int_{0}^{T}\abs{\partial_x(\rho_{n}(s)-\rho(s))}_{L^{2}(I)}ds,
\end{equation*}
where $I\subset \R$ is any bounded interval such that $\text{supp}(\psi)\subseteq I$, we obtain that 
\begin{equation}\label{convergence_step2}
	\lim_{n\to\infty}\sup_{t\in[0,T]}\Big\lvert \int_{0}^{t}\Inner{\partial_x^{\,2}(\rho_{n}(s)-\rho(s)), \psi}_{L^{2}(\R)}ds\Big\rvert = 0,\ \ \ \ \hat{\P}\text{-a.s.}
\end{equation} 
Next, thanks to \eqref{GN_inequality}, for any $R>0$ we have 
\begin{equation*}
	\begin{array}{ll}
		&\ds{ \hat{\E}\sup_{t\in[0,T]}\Big\lvert\int_{0}^{t} \Big(\abs{\partial_x\rho_{n}(s)}^{2}\rho_{n}(s)- \abs{\partial_x\rho(s)}^{2}\rho(s) \Big)ds \Big\lvert_{L^{2}(-R,R)}   }\\
		\vs
		&\ds{\lesssim \hat{\E}\int_{0}^{T}\Big\lvert \big\lvert \partial_x(\rho_{n}(s)-\rho(s))\big\rvert\big(\abs{\partial_x \rho_{n}(s)}+\abs{\partial_x \rho(s)}\big)\Big\rvert_{L^{2}(-R,R)}ds }\\
		\vs
		&\ds{\quad\quad\quad\quad +\hat{\E}\int_{0}^{T}\Big\lvert \abs{\partial_x\rho(s)}^{2}(\rho_{n}(s)-\rho(s))\Big\rvert_{L^{2}(-R,R)}ds  }\\
		\vs
		&\ds{\lesssim_{\,R}\hat{\E}\int_{0}^{T}\abs{\partial_x\rho_{n}(s)-\partial_x\rho(s)}_{L^{2}(-R,R)}\big(\abs{\partial_x^{2}\rho_{n}(s)}_{L^{2}(-R,R)}+\abs{\partial_x^{2}\rho(s)}_{L^{2}(-R,R)}\big)ds }\\
		\vs
		&\ds{\quad\quad\quad\quad +\hat{\E}\int_{0}^{T}
		\abs{\rho_{n}(s)-\rho(s)}_{L^{2}(-R,R)}\abs{\partial_x^{2}\rho(s)}_{L^{2}(-R,R)}^{2}ds },
	\end{array}
\end{equation*}
and  the Dominated Convergence Theorem implies that for every $R>0$
\begin{equation*}
	\lim_{n\to\infty}	\hat{\E}\sup_{t\in[0,T]}\Big\lvert\int_{0}^{t} \Big(\abs{\partial_x\rho_{n}(s)}^{2}\rho_{n}(s)- \abs{\partial_x\rho(s)}^{2}\rho(s) \Big)ds \Big\lvert_{L^{2}(-R,R)} = 0.
\end{equation*}
In particular, we get
\begin{equation}\label{convergence_step3}
	\lim_{n\to\infty}\hat{\E}\sup_{t\in[0,T]}\Big\lvert \int_{0}^{t}\Inner{ \abs{\partial_x\rho_{n}(s)}^{2}\rho_{n}(s)- \abs{\partial_x\rho(s)}^{2}\rho(s), \psi}_{L^{2}(\R)}\Big\rvert = 0.
\end{equation}
Moreover, since
\begin{equation*}
	\sup_{t\in[0,T]}\Big\lvert \int_{0}^{t}\Inner{\abs{\vartheta_{n}(s)}^{2}\rho_{n}(s), \psi}_{L^{2}(\R)}ds\Big\rvert\leq \abs{\psi}_{L^{\infty}(\R)} \int_{0}^{T}\abs{\vartheta_{n}(s)}_{L^{2}(\mathcal{O})}^{2}ds,
\end{equation*}
we obtain 
\begin{equation}\label{convergence_step4}
	\lim_{n\to\infty} \eps_{n}\sup_{t\in[0,T]}\Big\lvert \int_{0}^{t}\Inner{\abs{\vartheta_{n}(s)}^{2}\rho_{n}(s), \psi}_{L^{2}(\R)}ds\Big\rvert = 0,\ \ \ \ \hat{\P}\text{-a.s.}
\end{equation}
Finally, for any $R>0$
\begin{equation}
	\begin{array}{ll}
		&\ds{\sqrt{\eps_{n}}\ \hat{\E}\sup_{t\in[0,T]}\Big\lvert \int_{0}^{t}(\rho_{n}(s)\times \vartheta_{n}(s))dw_{n}(s)\Big\rvert_{L^{2}(-R,R)} }\\
		\vs
		&\ds{ \lesssim \sqrt{\eps_{n}} \ \hat{\E}\Big(\int_{0}^{T}\norm{\rho_{n}(s)\times \vartheta_{n}(s)}_{\L_{2}(K_{\eps_n},L^{2}(-R,R))}^{2}ds\Big)^{\frac{1}{2}} \lesssim \sqrt{\eps_n}\  \hat{\E}\Big(\int_{0}^{T}\abs{\vartheta_{n}(s)}_{L^{2}(-R,R)}^{2}ds\Big)^{\frac{1}{2}},}	\end{array}
\end{equation}
and this implies 
\begin{equation}\label{convergence_step5}
	\lim_{n\to\infty}\sqrt{\eps_n}\  \hat{\E}\sup_{t\in[0,T]}\Big\lvert \int_{0}^{t}\Inner{\psi, \rho_{n}(s)\times \vartheta_{n}(s))dw_{n}(s)}_{L^{2}(\mathbb{R})} ds \Big\rvert = 0.
\end{equation}
Consequently, by combining \eqref{convergence_step1}, \eqref{convergence_step2}, \eqref{convergence_step3}, \eqref{convergence_step4}, and \eqref{convergence_step5},  \eqref{test_limit} follows and the proof of Theorem \ref{small_mass_limit} is completed.

\section{The convergence of $\partial_t u_\e$ - Proof of Theorem \ref{teo3.4}}

{\em Step 1 - Proof of the first part.} We show that $\partial_t u_\e$ converges in probability to $\partial_t u$, with respect to the weak convergence in $L^2(0,T;L^2(\mathbb{R}))$.

According to \eqref{uniform_bound1} and \eqref{case1_boundedness1-bis}, there exists some $M>0$ such that  
\[\{\partial_t u_\e\}_{\e \in\,(0,1)}\subset \mathcal{S}_M:=\left\{ \varphi \in\,L^2(0,T;L^2(\mathbb{R}))\ :\ \vert\varphi\vert_{L^2(0,T;L^2(\mathbb{R}))}\leq M\right\},\ \ \ \ \ \ \mathbb{P}-\text{a.s.}\]
and $\partial_t u \in\,\mathcal{S}_M$.
As known, $\mathcal{S}_M$ is metrizable for the  weak topology of $L^2(0,T;L^2(\mathbb{R}))$, and the distance can be defined as
\[d(f,g)=\sum_{k=1}^\infty \frac 1{2^k}\,|\langle f-g,\varphi_k\rangle_{L^2(0,T;L^2(\mathbb{R}))}|,\ \ \ \ \ \ f, g \in\,L^2(0,T;L^2(\mathbb{R})),\]
where $\{\varphi_k\}_{k \in\,\mathbb{N}}\subset C^1([0,T];L^2(\mathbb{R}))$ is a dense subset of $\mathcal{S}_M$.

For every $\eta>0$, we pick $N_\eta \in\,\mathbb{N}$ such that
\[\sum_{k>N_\eta}\frac 1{2^k}\leq \frac \eta {4M^2}.\]
Then, we have
\[\begin{array}{l}
\ds{\mathbb{P}\left(d(\partial_t u_\e, \partial_t u)>\eta\right)=\mathbb{P}\left(\,\sum_{k=1}^\infty\frac 1{2^k}| \langle\partial_t u_\e- \partial_t u,\varphi_k\rangle_{L^2(0,T;L^2(\mathbb{R})}|>\eta\right)}\\
\vs	
\ds{\leq \mathbb{P}\left(\,\sum_{k\leq N_\eta}| \langle\partial_t u_\e- \partial_t u,\varphi_k\rangle_{L^2(0,T;L^2(\mathbb{R})}|>\frac\eta 2\right)\leq \sum_{k\leq N_\eta}\mathbb{P}\Big(| \langle\partial_t u_\e- \partial_t u,\varphi_k\rangle_{L^2(0,T;L^2(\mathbb{R})}|>\frac\eta{2N_\eta}\Big).}
\end{array}\]

Now, we fix $\e_\eta>0$ such that
\[\left|\langle u^\e_0-u_0,\varphi_k(0)\rangle_{L^2(\mathbb{R})}\right|\leq \frac{\eta}{6N_\eta},\ \ \ \ \ \ \e\leq \epsilon_\eta.\]
Since
\[\begin{array}{l}
\ds{\langle\partial_t u_\e- \partial_t u,\varphi_k\rangle_{L^2(0,T;L^2(\mathbb{R})}}\\
\vs
\ds{\quad \quad =\langle u_\e(T)-u(T),\varphi_k(T)\rangle_{L^2(\mathbb{R})}-\langle u^\e_0-u_0,\varphi_k(0)\rangle_{L^2(\mathbb{R})}-	\langle u_\e- u,\partial_t \varphi_k\rangle_{L^2(0,T;L^2(\mathbb{R})},}
\end{array}\]
for all $\e\leq \e_\eta$ we get
\[\begin{array}{l}
\ds{\mathbb{P}\left(d(\partial_t u_\e, \partial_t u)>\eta\right)\leq \sum_{k\leq N_\eta}\mathbb{P}\Big(|\langle u_\e(T)-u(T),\varphi_k(T)\rangle_{L^2(\mathbb{R})}|>\frac\eta{6N_\eta}\Big)	}\\
\vs
\ds{\quad \quad \quad \quad \quad \quad \quad \quad \quad \quad \quad \quad +\sum_{k\leq N_\eta}\mathbb{P}\Big(|\langle u_\e- u,\partial_t \varphi_k\rangle_{L^2(0,T;L^2(\mathbb{R})}|>\frac\eta{6N_\eta}\Big).}
\end{array}\]
Therefore, thanks to \eqref{convergence_pr} we conclude that 
\[\lim_{\e\to 0}\mathbb{P}\left(d(\partial_t u_\e, \partial_t u)>\eta\right)=0.\]

\medskip

{\em Step 2 - Proof of the second part.} We prove that $\partial_t u_\e$ does not converge in probability to $\partial_t u$, with respect to the strong convergence in $L^2(0,T;L^2(\mathbb{R}))$, if $c_0\neq 0$.

Due to \eqref{case1_boundedness1-bis}, we have
\[\vert\partial_x u(t)\vert^2_{L^2(\mathbb{R})}+(2\,\gamma+c_0)\,\int_0^t\vert \partial_t u(s)\vert_{L^2(\mathbb{R})}^2\,ds=\vert D u_0\vert_{L^2(\mathbb{R})}^2.\]
Therefore, in view of \eqref{uniform_bound1} we have
\[\begin{array}{l}
\ds{\int_0^T\int_0^t	\vert\partial_t u_\epsilon(s)\vert^2_{L^2(\mathbb{R})}\,ds\,dt-\left(1+\frac{c_0}{2\gamma}\right)\int_0^T\int_0^t	\vert\partial_t u(s)\vert^2_{L^2(\mathbb{R})}\,ds\,dt}\\
\vs
\ds{=\frac{\epsilon\,T}{2\gamma} \,\vert v_0^\epsilon\vert_{L^2(\mathbb{R})}^2-\frac 1{2\gamma}\left(\int_0^T \vert\partial_x u_\epsilon(s)\vert^2_{L^2(\mathbb{R})}\,dt-\int_0^T \vert\partial_x u(s)\vert^2_{L^2(\mathbb{R})}\,dt \right)-\frac \e{2\gamma}\int_0^T \vert\partial_t u_\epsilon(s)\vert^2_{L^2(\mathbb{R})}\,dt. }
\end{array}\]
Thanks to \eqref{sm105} and \eqref{uniform_bound1}, for every $\eta>0$ there exists $\epsilon_\eta>0$ such that
\[\frac{\epsilon\,T}{2\gamma} \,\vert v_0^\epsilon\vert_{L^2(\mathbb{R})}^2-\frac \e{2\gamma}\int_0^T \vert\partial_t u_\epsilon(s)\vert^2_{L^2(\mathbb{R})}\,dt<\eta/2,\ \ \ \ \ \ \e\leq\e_\eta,\]
so that 
\begin{equation}
\begin{array}{l}
\ds{J_{\e, \eta}:=\mathbb{P}\left(\Big|\int_0^T\int_0^t	\vert\partial_t u_\epsilon(s)\vert^2_{L^2(\mathbb{R})}\,ds\,dt-\left(1+\frac{c_0}{2\gamma}\right)\int_0^T\int_0^t	\vert\partial_t u(s)\vert^2_{L^2(\mathbb{R})}\,ds\,dt\Big|>\eta\right)}\\
\vs
\ds{\quad \quad \quad \quad \quad \leq \mathbb{P}\left(\Big|\int_0^T \vert\partial_x u_\epsilon(s)\vert^2_{L^2(\mathbb{R})}\,dt-\int_0^T \vert\partial_x u(s)\vert^2_{L^2(\mathbb{R})}\,dt\Big|>\eta \gamma\right).}	
\end{array}
	\end{equation}
This implies that for every $\e\leq \e_\eta$
\begin{equation}
	\begin{array}{l}
\ds{J_{\e, \eta}\leq \mathbb{P}\left(\vert\partial_x u_\epsilon-\partial_x u\vert_{L^2(0,T;L^2(\mathbb{R}))}\left(\vert\partial_x u_\epsilon\vert_{L^2(0,T;L^2(\mathbb{R}))}+\vert\partial_x u\vert_{L^2(0,T;L^2(\mathbb{R}))}\right)>\eta \gamma\right)}\\
\vs	
\ds{\quad \quad \leq \mathbb{P}\left(\vert\partial_x u_\epsilon-\partial_x u\vert_{L^2(0,T;L^2(\mathbb{R}))}>\eta \gamma/(R+\vert\partial_x u\vert_{L^2(0,T;L^2(\mathbb{R}))})\right)+\mathbb{P}\left(\vert\partial_x u_\epsilon\vert_{L^2(0,T;L^2(\mathbb{R}))}>R\right).}	
	\end{array}
\end{equation}
Now, according to \eqref{uniform_bound1} for every $\delta>0$ we can take $R_\delta>0$ such that 
\begin{equation}\label{sm110}\mathbb{P}\left(\vert\partial_x u_\epsilon\vert_{L^2(0,T;L^2(\mathbb{R}))}>R_\delta\right)\leq \frac{1}{R_\delta}\,\mathbb{E}\,\vert \partial_x u_\e\vert_{L^2(0,T;L^2(\mathbb{R}))}<\frac\delta 2.\end{equation}
Moreover, thanks to \eqref{convergence_pr}, there exists  $\epsilon_\delta\leq \epsilon_\eta$ such that
\begin{equation}\label{sm111}\mathbb{P}\left(\vert\partial_x u_\epsilon-\partial_x u\vert_{L^2(0,T;L^2(\mathbb{R}))}>\eta \gamma/(R_\delta+\vert\partial_x u\vert_{L^2(0,T;L^2(\mathbb{R}))})\right)<\frac\delta 2, \ \ \ \ \ \ \ \epsilon\leq \epsilon_\delta.\end{equation}
Thus, if we combine \eqref{sm110} and \eqref{sm111}, due to the arbitrariness of $\delta>0$ we get that for every $\eta>0$ 
\[\lim_{\e\to 0} \mathbb{P}\left(\Big|\int_0^T\int_0^t	\vert\partial_x u_\epsilon(s)\vert^2_{L^2(\mathbb{R})}\,ds\,dt-\left(1+\frac{c_0}{2\gamma}\right)\int_0^T\int_0^t	\vert\partial_x u(s)\vert^2_{L^2(\mathbb{R})}\,ds\,dt\Big|>\eta\right)=0.\]
In particular, if $c_0\neq 0$  we conclude that
\begin{equation}
\label{sm112}
\int_0^T\int_0^t	\vert\partial_x u_\epsilon(s)\vert^2_{L^2(\mathbb{R})}\,ds\,dt	\not\to \int_0^T\int_0^t	\vert\partial_x u(s)\vert^2_{L^2(\mathbb{R})}\,ds\,dt,\ \ \ \ \text{in probability, as}\ \e\to 0.
\end{equation}
Notice that for every $h \in\,L^2(0,T;L^2(\mathbb{R}))$ it holds
\[\int_0^T\int_0^t	\vert h(s)\vert^2_{L^2(\mathbb{R})}\,ds\,dt=\int_0^T(T-t)\vert h(t)\vert^2_{L^2(\mathbb{R})}\,dt.\]
Therefore, if we denote
\[L^2_T(0,T;L^2(\mathbb{R}))=\left\{h:(0,T)\to L^2(\mathbb{R})\ :\ \vert h\vert^2_{L^2_T(0,T;L^2(\mathbb{R}))}:=\int_0^T(T-t)\vert h(t)\vert^2_{L^2(\mathbb{R})}\,dt<\infty\right\},\]
from \eqref{sm112} we conclude that, as $\e\to 0$,
\begin{equation}
\label{sm113}
\vert \partial_t u_\epsilon\vert_{L^2_T(0,T;L^2(\mathbb{R}))}\not\to \vert \partial_t u\vert_{L^2_T(0,T;L^2(\mathbb{R}))},\ \ \ \ \text{in probability}.\end{equation}

Now, we show that this implies that, as $\e\to 0$,
 \begin{equation}
\label{sm115}
\partial_t u_\epsilon\not\to \partial_t u,\ \ \ \ \text{in probability}\ \ \ \ \ \text{in}\ L^2(0,T;L^2(\mathbb{R})).\end{equation}
We have
\[\begin{array}{l}
\ds{\vert \partial_t u_{\e}-\partial_t u\vert^2_{L_T^2(0,T;L^2(\mathbb{R})}=\int_0^T(T-t)\vert \partial_t u_{\e}(t)\vert^2_{L^2(\mathbb{R})}\,dt}\\
\vs
\ds{\quad \quad \quad +\int_0^T(T-t)\vert \partial_t u(t)\vert^2_{L^2(\mathbb{R})}\,dt-2\int_0^T (T-t)\langle \partial_t u_{\e}(t),\partial_t u(t)\rangle_{L^2(\mathbb{R})}\,dt,}
\end{array}\]
so that, for every $\eta>0$
\[\begin{array}{l}
\ds{\mathbb{P}\left(	\vert \partial_t u_{\e}-\partial_t u\vert^2_{L_T^2(0,T;L^2(\mathbb{R})}>\eta\right)}\\
\vs
\ds{\quad  \geq \mathbb{P}\left(\Big |\int_0^T(T-t)\vert \partial_t u_{\e}(t)\vert^2_{L^2(\mathbb{R})}\,dt-\int_0^T(T-t)\vert \partial_t u(t)\vert^2_{L^2(\mathbb{R})}\,dt\Big|>2\eta\right)}\\
\vs
\ds{\quad \quad -\mathbb{P}\left(\Big |\int_0^T(T-t)\langle \partial_t u_{\e}(t),\partial_t u(t)\rangle_{L^2(\mathbb{R})}\,dt-\int_0^T(T-t)\vert \partial_t u(t)\vert^2_{L^2(\mathbb{R})}\,dt\Big|>\eta/2\right).}
\end{array}\]
According to \eqref{sm113} there exists $\delta>0$, a subsequence $\{\e_n\}\downarrow 0$ and $\bar{\eta}>0$ such that
\begin{equation}
\label{sm117}
\mathbb{P}\left(\Big |\int_0^T(T-t)\vert \partial_t u_{\e_n}(t)\vert^2_{L^2(\mathbb{R})}\,dt-\int_0^T(T-t)\vert \partial_t u(t)\vert^2_{L^2(\mathbb{R})}\,dt\Big|>2\bar{\eta}\right)\geq 2\delta.	
\end{equation}
Moreover,  since for every $\psi \in\,L^2_T(0,T;L^2(\mathbb{R}))$  the function $t\mapsto (T-t)\psi(t)$ is in $L^2(0,T;L^2(\mathbb{R}))$, we have that the convergence in probability of $\partial_t u_\e$ to $\partial_t u$, with respect to the weak convergence in $L^2(0,T);L^2(\mathbb{R}))$ proved above, implies the same result, with respect to the weak convergence in $L_T^2(0,T);L^2(\mathbb{R}))$.
Hence, we can fix $\bar{n} \in\,\mathbb{N}$ such that for every $n\geq \bar{n}$
\begin{equation}
\label{sm120}
\mathbb{P}\left(\Big |\int_0^T(T-t)\langle \partial_t u_{\e_n}(t),\partial_t u(t)\rangle_{L^2(\mathbb{R})}\,dt-\int_0^T(T-t)\vert \partial_t u(t)\vert^2_{L^2(\mathbb{R})}\,dt\Big|>\bar{\eta}/2\right)\leq \delta.	
\end{equation}
Therefore, from \eqref{sm117} and \eqref{sm120} we conclude that 
\[\mathbb{P}\left(	\vert \partial_t u_{\e_n}-\partial_t u\vert^2_{L^2(0,T;L^2(\mathbb{R})}>\bar{\eta}/T\right)\geq \mathbb{P}\left(	\vert \partial_t u_{\e_n}-\partial_t u\vert^2_{L_T^2(0,T;L^2(\mathbb{R})}>\bar{\eta}\right)\geq \delta,\ \ \ \ \ \ n\geq \bar{n},\]
and this  implies \eqref{sm115}.

\section{Further uniform estimates}
\label{sec7}

Throughout this section, the noise ${w}(t)$ in system \eqref{SPDE1} is given by the spacial convolution of a  fractional noise $w^{H}(t)$ of  Hurst index $H\in (1/2,1)$ with a smooth function $\eta$.
As we will do also in Section \ref{secCLT}, here we assume Hypotheses \ref{H1}, \ref{H1-bis} and \ref{H2}, and take $(u_{0}^{\epsilon},v_{0}^{\epsilon})\in \big(\dot{H}^{2}(\R)\times H^{1}(\R)\big)$ fulfilling condition \eqref{initial_condition_rate}, for some  $u_{0}\in \dot{H}^{1}(\R)\cap M$.

Since the spectral measure $\mu$ of $w(t)=\eta\ast w^{H}(t)$ satisfies Hypothesis \ref{H1},
from \eqref{uniform_bound1}, \eqref{uniform_bound2} and \eqref{uniform_bound3}, for every $t \in\,[0,T]$ we have 
\begin{equation}\label{est_system_H}
	\abs{\partial_x u_{\epsilon}(t)}_{L^2(\mathbb{R})}^{2}+\epsilon\abs{\partial_{t}u_{\epsilon}(t)}_{L^2(\mathbb{R})}^{2}+2\gamma\int_{0}^{t}\abs{\partial_{t}u_{\epsilon}(s)}_{L^2(\mathbb{R})}^{2}ds = \abs{D u_0^{\epsilon}}_{L^2(\mathbb{R})}^{2}+\epsilon\abs{v_0^{\epsilon}}_{L^2(\mathbb{R})}^{2},
\end{equation}
$ \P$-a.s., and 
\begin{equation}\label{est_system_H1}
	\sup_{\epsilon\in(0,1)}\E\Big( \epsilon \sup_{t\in[0,T]}\abs{\partial_x ^{2}u_{\epsilon}(t)}_{L^2(\mathbb{R})}^{2} +\epsilon\int_{0}^{T}\abs{\partial_x \partial_{t}u_{\epsilon}(t)}_{L^2(\mathbb{R})}^{2}dt +\int_{0}^{T}\abs{\partial_x ^{2}u_{\epsilon}(t)}_{L^2(\mathbb{R})}^{2}dt   \Big)\leq c_{T}.
\end{equation}

\smallskip
For every $\epsilon\in (0,1)$ and $(t,x)\in \mathbb{R}^{+}\times \R,$  we denote
\begin{equation}\label{dec}
	\begin{array}{ll}
	&\ds{z_{\epsilon}(t):=\epsilon^{H/2-1}e^{\frac 1{\gamma_0}tA}(u_0^{\epsilon}-u_0)+ \frac 1{\gamma_0}\int_{0}^{t}e^{\frac 1{\gamma_0}(t-s)A}(u_{\epsilon}(s)\times \partial_{t}u_{\epsilon}(s))Q^{\epsilon}dw^{H}(s) =: z_{\epsilon, 1}(t)+z_{\epsilon, 2}(t)  },
	\end{array}
\end{equation}
where
\begin{equation}
	e^{tA}h\,(x):=\int_{\R}G_{t}(x-y)h(y)dy,\ \ \ \ \ \ \ \ \  G_{t}(x):=(\sqrt{4\pi t})^{-1/2}\exp\big(-\abs{x}^{2}/(4t)\big).
\end{equation}
Clearly,  $z_{\epsilon}$ solves the equation
\begin{equation}\label{sm168}
	\gamma_0\,\partial_{t}z_{\epsilon}(t) = \partial_x^{\,2} z_{\epsilon}(t)+ (u_{\epsilon}(t)\times \partial_{t}u_{\epsilon}(t))Q^{\epsilon}dw^{H}(t),\ \ \ \ \ \ z_{\epsilon}(0)=\epsilon^{H/2-1}(u_0^{\epsilon}-u_0).  \end{equation}

\subsection{Uniform  bounds for $z_{\epsilon}$}

In what follows, without any loss of generality, we  take $\gamma_0=1$. 

First, for $z_{\e,1}$  we have 
\begin{equation}\label{est_Z0_H}
	 \int_{0}^{T}\abs{z_{\epsilon,1}(t)}_{L^2(\mathbb{R})}^{2}dt\lesssim_{\,T}\epsilon^{H-2}\abs{u_{0}^{\epsilon}-u_0}_{L^2(\R)}^{2},
\end{equation}
 and 
\begin{equation}\label{est_Z0_H1}
	\sup_{t\in[0,T]}\abs{z_{\epsilon,1}(t)}_{L^2(\mathbb{R})}^{2}+ \int_{0}^{T}\abs{z_{\epsilon,1}(t)}_{H^{1}(\R)}^{2}dt\lesssim_{\,T}\epsilon^{H-2}\abs{u_{0}^{\epsilon}-u_0}_{L^2(\mathbb{R})}^{2}.\end{equation}

Next, for $z_{\e,2}$ we have
\begin{equation}
	z_{\epsilon, 2}(t) = \sum_{k=1}^{\infty}\int_{0}^{t}e^{(t-s)A}\Big(\big(u_{\epsilon}(s)\times \partial_{t}u_{\epsilon}(s)\big)\big(\eta_{\epsilon}\ast \xi_{k}^{H}\big)\Big)d\beta_{k}(s),\ \ \ \ t \in\,[0,T],
\end{equation}
where, with the notations introduced in Section \ref{notations}, $\{\beta_{k}\}_{k\in\mathbb{N}}$ is a sequence of i.i.d. Brownian motions, and 
\begin{equation}
	(\eta_{\epsilon}\ast \xi_{k}^{H})(x) = \int_{\R}\eta_{\epsilon}(x-y)\int_{\R}e^{-iyz}e_{k}^{H}(z)\,\mu_H(dz)dy = \int_{\R}e^{-ixy}\mathcal{F}\eta_{\epsilon}(y)e_{k}^{H}(y)\,\mu_H(dy),
\end{equation}
form a complete orthonormal basis of the reproducing kernel Hilbert spaceof $Q^{\epsilon}w^{H}(t)$.

\begin{Lemma} \label{lemma6.2}
	Let $H\in[1/2,1)$. For every $T>0$ and $\alpha>0$ we have 
	\begin{equation}\label{est_Z1_H_1}
		\sup_{\e \in\,(0,1)}\,\epsilon^{1/2+\alpha}\,\E\sup_{t\in[0,T]}\abs{z_{\epsilon,2}(t)}_{L^2(\mathbb{R})}^{2}<\infty.
	\end{equation}
	Moreover 
	\begin{equation}\label{est_Z1_H_2}
		\sup_{\epsilon\in(0,1)}\E\int_{0}^{T}\abs{z_{\epsilon, 2}(t)}_{L^2(\mathbb{R})}^{2}dt<+\infty.
	\end{equation}
\end{Lemma}

\begin{proof}

For every $t\in [0,T]$, we have
\begin{equation}
	\begin{array}{l}
		\ds{ \E\sup_{r\in[0,t]}\abs{z_{\epsilon, 2}(r)}_{L^2(\mathbb{R})}^{2}    }\\
		\vs 
		\ds{\lesssim \E\int_{0}^{t}\sum_{k=1}^{\infty}\int_{\R}\Big\lvert \int_{\R^{2}}e^{-iyz}G_{t-s}(x-y)(u_{\epsilon}(s,y)\times \partial_{t}u_{\epsilon}(s,y))\mathcal{F}\eta_{\epsilon}(z)e_{k}^{H}(z)\mu_H(dz)dy\Big\rvert^{2}dx\,ds    }\\
		\vs 
		\ds{= \E\int_{0}^{t}\!\!\int_{\R^2}\Big\lvert \int_{\R}e^{-iyz}G_{t-s}(x-y)(u_{\epsilon}(s,y)\times \partial_{t}u_{\epsilon}(s,y))dy\Big\rvert^{2}\cdot\abs{\mathcal{F}\eta_{\epsilon}(z)}^{2}\mu_H(dz)\,dx\,ds }\\
		\vs 
		\ds{\lesssim \E\int_{0}^{t}\!\!\int_{\R}\int_{\abs{z}>1}\Big\lvert \int_{\R}e^{-iyz}G_{t-s}(x-y)(u_{\epsilon}(s,y)\times \partial_{t}u_{\epsilon}(s,y))dy\Big\rvert^{2} dz\,dx\,ds  }\\
		\vs 
		\ds{+\E\int_{0}^{t}\!\!\int_{\R}\int_{\abs{z}\leq1}\Big\lvert \int_{\R}e^{-iyz}G_{t-s}(x-y)(u_{\epsilon}(s,y)\times \partial_{t}u_{\epsilon}(s,y))dy\Big\rvert^{2} \mu_H(dz)\,dx\,ds =: I_{1}^{\epsilon}(t)+I_{2}^{\epsilon}(t) }.
	\end{array}
\end{equation}
Note that there exists some $\kappa>0$
\begin{equation}\label{sm80}\abs{G_{t}}_{L^{2}(\R)}=\kappa\,t^{-1/4},\ \ \ \ \ t>0.\end{equation}
Hence  
\begin{equation}\label{sm61}
	\begin{array}{ll}
		&\ds{I_{1}^{\epsilon}(t)\leq \E\int_{0}^{t}\!\!\int_{\R^2}\Big\lvert \int_{\R}e^{-iyz}G_{t-s}(x-y)(u_{\epsilon}(s,y)\times \partial_{t}u_{\epsilon}(s,y))dy\Big\rvert^{2} dz\,dx\,ds   }\\
		\vs 
		&\ds{ \quad = \E\int_{0}^{t}\!\!\int_{\R^2}\Big\lvert G_{t-s}(x-z)(u_{\epsilon}(s,z)\times \partial_{t}u_{\epsilon}(s,z)) \Big\rvert^{2}dz\,dx\,ds   }\\
		\vs 
		&\ds{\quad\quad =\E\int_{0}^{t}\!\!\int_{\R}\abs{\partial_tu_{\epsilon}(s,y)}^{2}\int_{\R}\abs{G_{t-s}(x-y)}^{2}dx\,dy\,ds =\kappa^2\,\E\int_{0}^{t}(t-s)^{-1/2}\abs{\partial_{t}u_{\epsilon}(s)}_{L^2(\mathbb{R})}^{2}ds }.
	\end{array}
\end{equation}
Moreover, since for any $z\in\mathbb{R}$,
\begin{equation}
	\begin{array}{l}
		\ds{  \int_{\R}\Big\lvert \int_{\R}e^{-iyz}G_{t-s}(x-y)(u_{\epsilon}(s,y)\times \partial_{t}u_{\epsilon}(s,y))dy\Big\rvert^{2}dx    }\\
		\vs 
		\ds{ = \Big\lvert e^{(t-s)A}\big(e^{-iz\cdot}(u_{\epsilon}(s)\times \partial_{t}u_{\epsilon}(s))\big)\Big\rvert_{L^2(\mathbb{R})}^{2} \lesssim \Big\lvert e^{-iz\cdot}(u_{\epsilon}(s)\times \partial_{t}u_{\epsilon}(s))\Big\rvert_{L^2(\mathbb{R})}^{2} = \abs{\partial_{t}u_{\epsilon}(s)}_{L^2(\mathbb{R})}^{2} },
	\end{array}
\end{equation}
we have 
\begin{equation}\label{sm60}
	I_{2}^{\epsilon}(t)\lesssim \E\int_{0}^{t}\abs{\partial_{t}u_{\epsilon}(s)}_{L^2(\mathbb{R})}^{2}ds\cdot \int_{\abs{z}\leq 1}\abs{z}^{1-2H}dz \lesssim\E\int_{0}^{t}\abs{\partial_{t}u_{\epsilon}(s)}_{L^2(\mathbb{R})}^{2}ds.
\end{equation}
Therefore, from \eqref{sm61} and \eqref{sm60} we get
\begin{equation}
	\E\sup_{r\in[0,t]}\abs{z_{\epsilon, 2}(r)}_{L^2(\mathbb{R})}^{2}\lesssim \E\int_{0}^{t}\big(1+(t-s)^{-1/2}\big)\abs{\partial_{t}u_{\epsilon}(s)}_{L^2(\mathbb{R})}^{2}ds.
\end{equation}
On one hand, due to \eqref{uniform_bound1},  for every $\alpha\in(0,1/2)$ we have
\begin{equation}
	\begin{array}{ll}
	&\ds{\int_{0}^{t}(t-s)^{-1/2}\abs{\partial_{t}u_{\epsilon}(s)}_{L^2(\mathbb{R})}^{2}ds \leq \sup_{r\in[0,t]}\abs{\partial_{t}u_{\epsilon}(r)}_{L^2(\mathbb{R})}^{1+2\alpha}\int_{0}^{t}(t-s)^{-1/2}\abs{\partial_{t}u_{\epsilon}(s)}_{L^2(\mathbb{R})}^{1-2\alpha}ds}\\
	\vs 
	&\ds{\quad\quad  \lesssim \epsilon^{-(1/2+\alpha)}\Big(\int_{0}^{t}(t-s)^{-\frac{1}{1+2\alpha}}ds\Big)^{\frac{1+2\alpha}2}\Big(\int_{0}^{t}\abs{\partial_{t}u_{\epsilon}(s)}_{L^{2}(\R)}^{2}ds\Big)^{\frac{1-2\alpha}2} \lesssim_{\,\alpha,T}\epsilon^{-(1/2+\alpha)}  },
	\end{array}
\end{equation}
which implies that 
\begin{equation}\label{sm160}
	\E\sup_{t\in[0,T]}\abs{z_{\epsilon, 2}(t)}_{L^2(\mathbb{R})}^{2}\lesssim_{\,\alpha,T}\epsilon^{-(1/2+\alpha)}.
\end{equation}
On the other hand, thanks again to \eqref{uniform_bound1}  the Young convolution inequality gives
\begin{equation}
	\E\int_{0}^{T}\abs{z_{\epsilon, 2}(t)}_{L^2(\mathbb{R})}^{2}dt \lesssim \int_{0}^{T}(1+s^{-1/2})ds\,\E\int_{0}^{T}\abs{\partial_{t}u_{\epsilon}(s)}_{L^2(\mathbb{R})}^{2}ds \lesssim_{\,T}\E\int_{0}^{T}\abs{\partial_{t}u_{\epsilon}(s)}_{L^2(\mathbb{R})}^{2}ds.
\end{equation}
Therefore, if we put this together with \eqref{sm160} we get our thesis.

\end{proof}

\begin{Lemma}
 Let $H\in[1/2,1)$. For every $T>0$ and $\alpha >0$, we have 
 \begin{equation}\label{est_Z1_H1}
 \sup_{\e \in\,(0,1)}\,\epsilon^{-(1/2+\alpha)}\,	\E\int_{0}^{T}\abs{z_{\epsilon, 2}(t)}_{H^{1}}^{2}dt<\infty.
 \end{equation}
\end{Lemma}

\begin{proof}
	For every $t\in[0,T]$, we have 
	\begin{equation}
		\begin{array}{ll}
			&\ds{  \E\abs{\partial_x z_{\epsilon, 2}(t)}_{H^{1}}^{2} \lesssim  \E\int_{0}^{t}\int_{\R}\int_{\abs{z}>1}\big\lvert \partial_{x}\Phi_{\epsilon,t}(s,x,z)\big\rvert^{2}\abs{\mathcal{F}\eta_{\epsilon}(z)}^{2}\mu_H(dz)dxds  }\\
			\vs 
			&\ds{\quad\quad\quad\quad\quad +\E\int_{0}^{t}\int_{\R}\int_{\abs{z}\leq 1}\big\lvert \partial_{x}\Phi_{\epsilon,t}(s,x,z)\big\rvert^{2}\abs{\mathcal{F}\eta_{\epsilon}(z)}^{2}\mu_H(dz)dxds =: J_{1}^{\epsilon}(t)+J_{2}^{\epsilon}(t)},
		\end{array}
	\end{equation}
	where 
	\begin{equation}
		\begin{array}{ll}
			&\ds{ \Phi_{\epsilon,t}(s,x,z):=\int_{\R}e^{-iyz}G_{t-s}(x-y)(u_{\epsilon}(s,y)\times \partial_{\epsilon}(s,y))dy }\\
			\vs
			&\ds{\quad\quad \quad\quad \quad\quad =e^{(t-s)A}\big(e^{-iz\cdot}(u_{\epsilon}(s)\times \partial_{t}u_{\epsilon}(s))\big)(x), \ \ \ \ s\in[0,t],\ \ (x,z)\in \R^{2}.}
		\end{array}
	\end{equation}
For every $\a \in\,(0,1)$ we have
	\begin{equation}
		\abs{e^{tA}h}_{H^{1}}\leq c\,(t\wedge 1)^{-\frac{1-\alpha}{2}}\abs{h}_{H^{\alpha}},\ \ \ \ t>0,\ \ \ \ h\in H^{\alpha}(\R),
	\end{equation}
so that 
	\begin{equation}
		\begin{array}{ll}
			&\ds{\int_{\R}\big\lvert \partial_{x}\Phi_{\epsilon,t}(s,x,z)\big\rvert^{2}dx \leq  \big\lvert e^{(t-s)A}\big(e^{-iz\cdot}(u_{\epsilon}(s)\times \partial_{t}u_{\epsilon}(s))\big)\big\rvert_{H^{1}}^{2} }\\
			\vs 
			&\ds{\quad\quad\lesssim (t-s)^{-1+\alpha}\big\lvert e^{-iz\cdot}(u_{\epsilon}(s)\times \partial_{t}u_{\epsilon}(s))\big\rvert_{H^{\alpha}}^{2}     }\\
			\vs 
			&\ds{\quad\quad\quad \lesssim (t-s)^{-1+\alpha}\big\lvert e^{-iz\cdot}(u_{\epsilon}(s)\times \partial_{t}u_{\epsilon}(s))\big\rvert_{H^{1}}^{2\alpha}\big\lvert e^{-iz\cdot}(u_{\epsilon}(s)\times \partial_{t}u_{\epsilon}(s))\big\rvert_{L^2(\mathbb{R})}^{2(1-\alpha)} }.
		\end{array}
	\end{equation}
	Moreover, due to \eqref{uniform_bound1}, we have 
	\begin{equation}
		\begin{array}{l}
			\ds{ \big\lvert e^{-iz\cdot}(u_{\epsilon}(s)\times \partial_{t}u_{\epsilon}(s))\big\rvert_{H^{1}}^{2\alpha}   }\\
			\vs 
			\ds{\quad \lesssim_{\,\alpha}\abs{z}^{2\alpha}\abs{u_{\epsilon}(s)\times \partial_{t}u_{\epsilon}(s)}_{L^2(\mathbb{R})}^{2\alpha}+\abs{\partial_x u_{\epsilon}(s)\times \partial_{t}u_{\epsilon}(s)}_{L^2(\mathbb{R})}^{2\alpha} +\abs{u_{\epsilon}(s)\times \partial_x \partial_{t}u_{\epsilon}(s)}_{L^2(\mathbb{R})}^{2\alpha}   }\\
			\vs 
			\ds{\quad \quad \quad \lesssim_{\,\alpha} \abs{z}^{2\alpha}\abs{\partial_{t}u_{\epsilon}(s)}_{L^2(\mathbb{R})}^{2\alpha}+\abs{\partial_{t}u_{\epsilon}(s)}_{H^{1}}^{2\alpha},\ \ \ \ \ \ \ \mathbb{P}-\text{a.s}.   }
		\end{array}
	\end{equation}
This allows to conclude that, $\mathbb{P}$ almost-surely
	\begin{equation}\label{sm70}
		\int_{\R}\big\lvert \partial_{x}\Phi_{\epsilon, t}(s,x,z)\big\rvert^{2}dx\lesssim_{\,\alpha} (t-s)^{-1+\alpha}\Big(\abs{z}^{2\alpha}\abs{\partial_{t}u_{\epsilon}(s)}_{L^2(\mathbb{R})}^{2}+\abs{\partial_{t}u_{\epsilon}(s)}_{H^{1}}^{2\alpha}\abs{\partial_{t}u_{\epsilon}(s)}_{L^2(\mathbb{R})}^{2(1-\alpha)}\Big).
	\end{equation}
In particular
	\begin{equation}
		\begin{array}{ll}
			&\ds{ J_{1}^{\epsilon}(t) \lesssim_{\,\alpha} \E\int_{0}^{t}(t-s)^{-1+\alpha}\abs{\partial_{t}u_{\epsilon}(s)}_{L^2(\mathbb{R})}^{2}ds \int_{\abs{z}>1}\abs{z}^{1+2\alpha-2H}\abs{\mathcal{F}\eta_{\epsilon}(z)}^{2}dz  }\\
			\vs 
			&\ds{\quad\quad\quad\quad +\E\int_{0}^{t}(t-s)^{-1+\alpha}\abs{\partial_{t}u_{\epsilon}(s)}_{H^{1}}^{2\alpha}\abs{\partial_{t}u_{\epsilon}(s)}_{L^2(\mathbb{R})}^{2(1-\alpha)}ds \int_{\abs{z}>1}\abs{z}^{1-2H}\abs{\mathcal{F}\eta_{\epsilon}(z)}^{2}dz    }.
		\end{array}
	\end{equation}
	Now,  for any $\theta\leq 0$ it holds
	\begin{equation}
		\begin{array}{ll}
			&\ds{		\int_{\abs{z}>1}\abs{z}^{\theta}\abs{\mathcal{F}\eta_{\epsilon}(z)}^{2}dz = \int_{\abs{z}>1}\abs{z}^{\theta}\abs{\mathcal{F}\eta(\sqrt{\epsilon }z)}^{2}dz = \epsilon^{-\frac{1+\theta}2}\int_{\abs{z}>\sqrt{\epsilon}}\abs{z}^{\theta}\abs{\mathcal{F}\eta(z)}^{2}dz }\\
			\vs 
			&\ds{\quad\quad\quad\quad\quad\quad  \leq \epsilon^{-1/2}\int_{\abs{z}>\sqrt{\epsilon}}\abs{\mathcal{F}\eta(z)}^{2}dz\lesssim \epsilon^{-1/2}},
		\end{array}
	\end{equation}
	and for any $\theta\in (0,3-2H)$,
	\begin{equation}
		\int_{\abs{z}>1}\abs{z}^{\theta}\abs{\mathcal{F}\eta_{\epsilon}(z)}^{2}dz \leq \epsilon^{-\frac{1+\theta}2}\int_{\R}\abs{z}^{\theta}\abs{\mathcal{F}\eta(z)}^{2}dz\lesssim \epsilon^{-\frac{1+\theta}2}.
	\end{equation}
	This implies that, for any $H\in[1/2,1)$ and $\alpha\in (0, 1)$, thanks to \eqref{est_system_H1} we have 
	\begin{equation}\label{sm73}
		\begin{array}{l}
			\ds{ \int_{0}^{T}J_{1}^{\epsilon}(t)dt \lesssim_{\,\alpha} \epsilon^{-(1/2+\alpha)}\int_{0}^{T}s^{-1+\alpha}ds\, \E\int_{0}^{T}\abs{\partial_{t}u_{\epsilon}(s)}_{L^2(\mathbb{R})}^{2}ds   }\\
			\vs 
			\ds{ +\epsilon^{-1/2}\,\int_{0}^{T}s^{-1+\alpha}ds\,\E\int_{0}^{T}\abs{\partial_{t}u_{\epsilon}(s)}_{H^{1}}^{2\alpha}\abs{\partial_{t}u_{\epsilon}(s)}_{L^2(\mathbb{R})}^{2(1-\alpha)}ds  }\\
			\vs 
			\ds{\quad\quad \quad \quad  \lesssim_{\,\alpha,T} \epsilon^{-(1/2+\alpha)}\, \E\int_{0}^{T}\abs{\partial_{t}u_{\epsilon}(s)}_{L^2(\mathbb{R})}^{2}ds   }\\
			\vs 
			\ds{\quad  +\epsilon^{- 1/2}\E\,\left(\Big(\int_{0}^{T}\abs{\partial_{t}u_{\epsilon}(s)}_{H^{1}}^{2}ds\Big)^{\alpha}\Big(\int_{0}^{T}\abs{\partial_{t}u_{\epsilon}(s)}_{L^2(\mathbb{R})}^{2}ds\Big)^{1-\alpha} \right)  \lesssim_{\,\alpha,T} \epsilon^{-(1/2+\alpha)}  }.
		\end{array}
	\end{equation}
	Similarly, due to \eqref{sm70}
	\begin{equation}\label{sm74}
		\begin{array}{ll}
			&\ds{ \int_0^T\,J_{2}^{\epsilon}(t)\,dt \lesssim_{\alpha} \int_0^T\,\E\int_{0}^{t}(t-s)^{-1+\alpha}\abs{\partial_{t}u_{\epsilon}(s)}_{L^2(\mathbb{R})}^{2}ds\cdot \int_{\abs{z}\leq1}\abs{z}^{1+2\alpha-2H}dz\,dt  }\\
			\vs 
			&\ds{\quad\quad +\int_0^T\,\E\int_{0}^{t}(t-s)^{-1+\alpha}\abs{\partial_{t}u_{\epsilon}(s)}_{H^{1}}^{2\alpha}\abs{\partial_{t}u_{\epsilon}(s)}_{L^2(\mathbb{R})}^{2(1-\alpha)}ds\cdot \int_{\abs{z}\leq1}\abs{z}^{1-2H}dz\,dt  \lesssim_{\,\alpha,T}\epsilon^{-\alpha},}
			\end{array}\end{equation}
and in view of \eqref{sm73} and \eqref{sm74} we conclude that 	\eqref{est_Z1_H1} holds.

\end{proof}

Finally, combining together \eqref{est_Z0_H1}, \eqref{est_Z1_H_1}, \eqref{est_Z1_H_2} and \eqref{est_Z1_H1}, we obtain the following result.
\begin{Lemma}
	Let $H\in[1/2,1)$. For every $T>0$, we have 
	\begin{equation}\label{est_Z_sum1}
		\sup_{\epsilon\in(0,1)}\E\int_{0}^{T}\abs{z_{\epsilon}(t)}_{L^2(\mathbb{R})}^{2}dt<+\infty,
	\end{equation}
	and for every $\alpha>0$,
	\begin{equation}\label{est_Z_sum2}
		\E\,\sup_{t\in[0,T]}\abs{z_{\epsilon}(t)}_{L^2(\mathbb{R})}^{2}+\mathbb{E}\int_{0}^{T}\abs{z_{\epsilon}(t)}_{H^{1}}^{2}dt\lesssim_{\,\alpha,T} \epsilon^{-(1/2+\alpha)}+\e^{H-2}\,\vert u^\epsilon_0-u_0\vert_{L^2(\R)}^2,\ \ \ \ 0<\epsilon\ll1.
	\end{equation}
\end{Lemma}

\subsection{Uniform bounds for $y_{\epsilon}$}

In Subsection \ref{sub3.3} we have defined $y_\e$ as $\e^{H/2-1} (u_\e-u)$. It is 
immediate to check that $y_{\epsilon}$ satisfies the following equation
\begin{equation}
	\le\{\begin{array}{l}
		\ds{\gamma_0\,\partial_{t}y_{\epsilon}(t,x) = \partial_x^{\,2} y_{\epsilon}(t,x)+\abs{\partial_x u(t,x)}^{2}y_{\epsilon}(t,x)+2(\partial_x y_{\epsilon}(t,x)\cdot \partial_x u(t,x))u_{\epsilon}(t,x) }\\[10pt]
		\ds{\quad \quad\quad\quad\quad\quad\quad +\epsilon^{1-H/2}\abs{\partial_x y_{\epsilon}(t,x)}^{2}u_{\epsilon}(t,x)  -\epsilon^{H/2}\partial_{t}^{2}u_{\epsilon}(t,x)-\epsilon^{H/2}\abs{\partial_{t}u_{\epsilon}(t,x)}^{2}u_{\epsilon}(t,x)}\\
		\vs
		\ds{\quad \quad\quad\quad\quad\quad\quad\quad \quad\quad\quad\quad\quad\quad+\big(u_{\epsilon}(t)\times\partial_{t}u_{\epsilon}(t)\big)Q^{\epsilon}\partial_{t}w^{H}(t,x), }\\
		[10pt]
		\ds{y_{\epsilon}(0,x)=\epsilon^{H/2-1}(u_0^{\epsilon}(x)-u_{0}(x)) }.
	\end{array}\r.
\end{equation}
Thus, if we define 
\[r_{\epsilon}(t):= y_{\epsilon}(t)-z_{\epsilon}(t),\ \ \ \ \ \ t \in\,[0,T],\] we have that  $r_{\epsilon}$ satisfies the equation
\begin{equation}
	\le\{\begin{array}{l}
		\ds{\gamma_0\,\partial_{t}r_{\epsilon}(t) = \partial_x^{\,2} r_{\epsilon}(t)+\abs{\partial_x u(t)}^{2}r_{\epsilon}(t)+(\partial_x(u_{\epsilon}(t)-u(t))\cdot \partial_x r_{\epsilon}(t))u_{\epsilon}(t) }\\
		[10pt]
		\ds{\quad\quad\quad\quad + (\partial_x(u_{\epsilon}(t)-u(t))\cdot \partial_x z_{\epsilon}(t))u_{\epsilon}(t)+2(\partial_x r_{\epsilon}(t)\cdot \partial_x u(t))u_{\epsilon}(t)+\abs{\partial_x u(t)}^{2}z_{\epsilon}(t)   }\\
		[10pt]
		\ds{\quad\quad\quad\quad\quad\quad\quad +2(\partial_x z_{\epsilon}(t)\cdot \partial_xu(t))u_{\epsilon}(t)-\epsilon^{H/2}\partial_{t}^{2}u_{\epsilon}(t)-\epsilon^{H/2}\abs{\partial_{t}u_{\epsilon}(t)}^{2}u_{\epsilon}(t),   }\\
		[10pt]
		\ds{r_{\epsilon}(0)=0. }
	\end{array}\r.
\end{equation}

\begin{Lemma}
	Let $H\in [1/2,1)$. Then for every $T>0$ and $\alpha>0$, we have
	\begin{equation}\label{est_R}
		\E\sup_{t\in[0,T]}\abs{r_{\epsilon}(t)}_{L^2(\mathbb{R})}^{2}+\E\int_{0}^{T}\abs{r_{\epsilon}(t)}_{H^{1}}^{2}dt\lesssim_{\,\alpha,T}\epsilon^{-(1/2+\alpha)}+\epsilon^{H-2}\,\vert u^\e_0-u_0\vert_{L^2(\mathbb{R})}^2,\ \ \ \ 0<\epsilon\ll1.
	\end{equation}
	\end{Lemma}

\begin{proof}
	Integrating by parts, we get 
	\begin{equation}
		\begin{array}{ll}
			&\ds{ \frac{\gamma_0}{2}\frac{d}{dt}\abs{r_{\epsilon}(t)}_{L^2(\mathbb{R})}^{2} = -\abs{\partial_xr_{\epsilon}(t)}_{L^2(\mathbb{R})}^{2}+\Inner{\abs{\partial_x u(t)}^{2}r_{\epsilon}(t),r_{\epsilon}(t) }_{L^2(\mathbb{R})}}\\
		\vs
		&\ds{ \quad\quad\quad\quad+\Inner{(\partial_x(u_{\epsilon}(t)-u(t))\cdot \partial_x r_{\epsilon}(t) )u_{\epsilon}(t),r_{\epsilon}(t) }_{L^2(\mathbb{R})}    }\\
			\vs 
			&\ds{+\Inner{(\partial_x(u_{\epsilon}(t)-u(t))\cdot \partial_x z_{\epsilon}(t) )u_{\epsilon}(t),r_{\epsilon}(t) }_{L^2(\mathbb{R})} + 2\Inner{ (\partial_x r_{\epsilon}(t)\cdot \partial_x u(t))u_{\epsilon}(t),r_{\epsilon}(t) }_{L^2(\mathbb{R})}   }\\
			\vs 
			&\ds{\quad \quad\quad\quad\quad+\Inner{\abs{\partial_x u(t)}^{2}z_{\epsilon}(t),r_{\epsilon}(t) }_{L^2(\mathbb{R})} +2\Inner{ (\partial_x z_{\epsilon}(t)\cdot \partial_x u(t))u_{\epsilon}(t),r_{\epsilon}(t) }_{L^2(\mathbb{R})}   }\\
			\vs 
			&\ds{\quad  -\epsilon^{H/2}\Inner{\abs{\partial_{t}u_{\epsilon}(t)}^{2}u_{\epsilon}(t),r_{\epsilon}(t) }_{L^2(\mathbb{R})} -\epsilon^{H/2}\Inner{\partial_{t}^{2}u_{\epsilon}(t),r_{\epsilon}(t)}_{L^2(\mathbb{R})}  }\\
			&\ds{\quad \quad \quad \quad \quad \quad =: -\abs{r_{\epsilon}(t)}_{H^{1}}^{2}+\sum_{i=1}^{7}I_{i}^{\epsilon}(t)  -\epsilon^{H/2}\Inner{\partial_{t}^{2}u_{\epsilon}(t),r_{\epsilon}(t)}_{L^2(\mathbb{R})} }.
		\end{array}
	\end{equation}
	Note that 
	\begin{equation}
		\Inner{\partial_{t}^{2}u_{\epsilon}(t),r_{\epsilon}(t)}_{L^2(\mathbb{R})} = \frac{d}{dt}\Inner{\partial_{t}u_{\epsilon}(t),r_{\epsilon}(t)}_{L^2(\mathbb{R})} - \Inner{\partial_{t}u_{\epsilon}(t),\partial_{t}r_{\epsilon}(t)}_{L^2(\mathbb{R})},
	\end{equation}
	and 
	\begin{equation}
		\Inner{\partial_{t}^{2}u_{\epsilon}(t), \partial_{t}u_{\epsilon}(t)}_{L^2(\mathbb{R})} = \frac{1}{2}\frac{d}{dt}\abs{\partial_{t}u_{\epsilon}(t)}_{L^2(\mathbb{R})}^{2}-\frac{c_0}{2\epsilon}\abs{\partial_{t}u_{\epsilon}(t)}_{L^2(\mathbb{R})}^{2}.
	\end{equation}
Therefore, we get
	\begin{equation}
		\begin{array}{ll}
			&\ds{ \frac 12 \frac{d}{dt}\Big( \gamma_0\abs{r_{\epsilon}(t)}_{L^2(\mathbb{R})}^{2}+2\,\epsilon^{H/2}\Inner{\partial_{t}u_{\epsilon}(t),r_{\epsilon}(t)}_{L^2(\mathbb{R})} +\frac{1}{\gamma_0}\,\epsilon^{H}\,\abs{\partial_{t}u_{\epsilon}(t)}_{L^2(\mathbb{R})}^{2} \Big) +\abs{\partial_xr_{\epsilon}(t)}_{L^2(\mathbb{R})}^{2}  }\\[10pt]
					&\ds{\quad \quad \quad \quad \quad \quad = \sum_{i=1}^{7}I_{i}^{\epsilon}(t) +\frac 1{\gamma_0}\sum_{j=1}^{9}J^\epsilon_{j}(t),}
			\end{array}
	\end{equation}
	where
	\begin{equation}
		\begin{array}{l}		
		\ds{\epsilon^{-H/2}\sum_{j=1}^{9}J^\epsilon_{j}(t):=	 \Inner{\partial_x^{\,2} r_{\epsilon}(t),\partial_{t}u_{\epsilon}(t)}_{L^2(\mathbb{R})} + \Inner{\abs{\partial_x u(t)}^{2}r_{\epsilon}(t),\partial_{t}u_{\epsilon}(t) }_{L^2(\mathbb{R})}   }\\
			\vs 
			\ds{\quad \quad \quad \quad  +\Inner{(\partial_x(u_{\epsilon}(t)-u(t))\cdot \partial_x r_{\epsilon}(t) )u_{\epsilon}(t),\partial_{t}u_{\epsilon}(t) }_{L^2(\mathbb{R})}}\\
			\vs
			\ds{\quad \quad \quad \quad \quad \quad \quad \quad  +\Inner{(\partial_x(u_{\epsilon}(t)-u(t))\cdot \partial_x z_{\epsilon}(t) )u_{\epsilon}(t),\partial_{t}u_{\epsilon}(t) }_{L^2(\mathbb{R})}  }\\
			\vs 
			\ds{\quad \quad \quad +2\Inner{(\partial_x r_{\epsilon}(t)\cdot \partial_x u(t))u_{\epsilon}(t), \partial_{t}u_{\epsilon}(t) }_{L^2(\mathbb{R})}+2\Inner{(\partial_x z_{\epsilon}(t)\cdot \partial_x u(t))u_{\epsilon}(t), \partial_{t}u_{\epsilon}(t) }_{L^2(\mathbb{R})}   }\\
			\vs 
			\ds{ +\Inner{\abs{\partial_x u(t)}^{2}z_{\epsilon}(t),\partial_{t}u_{\epsilon}(t) }_{L^2(\mathbb{R})} - \epsilon^{H/2}\Inner{\abs{\partial_{t}u_{\epsilon}(t)}^{2}u_{\epsilon}(t),\partial_{t}u_{\epsilon}(t) }_{L^2(\mathbb{R})} + \frac{c_0}{2\epsilon^{1-H/2}}\abs{\partial_{t}u_{\epsilon}(t)}_{L^2(\mathbb{R})}^{2}. }
		\end{array}
	\end{equation}
	Now, we will estimate $I^\e_{i}(t)$ and $J^\e_{j}(t)$, for $1\leq i\leq 7$ and $1\leq j \leq 8$. In what follows, $\delta$ is an arbitrary positive constant that we will pick later. Due to \eqref{uniform_bound1}, we have 
	\begin{equation}
		I_{1}^{\epsilon} = \Inner{\abs{\partial_x u}^{2}r_{\epsilon},r_{\epsilon} }_{L^2(\mathbb{R})}\leq \abs{\partial_x u}_{L^{\infty}(\R)}^{2}\abs{r_{\epsilon}}_{L^2(\mathbb{R})}^{2}\lesssim \abs{\partial_x^{2}u}_{L^2(\mathbb{R})}\abs{r_{\epsilon}}_{L^2(\mathbb{R})}^{2}.
	\end{equation}
Moreover	
	\begin{equation}
		\begin{array}{l}
		\ds{	
		I_{2}^{\epsilon}=\Inner{ (\partial_x(u_{\epsilon}-u)\cdot\partial_x r_{\epsilon})u_{\epsilon},r_{\epsilon}  }_{L^2(\mathbb{R})}}\\
		\vs
		\ds{\quad \quad \quad \quad \lesssim \abs{r_{\epsilon}}_{H^{1}}\abs{r_{\epsilon}}_{L^{\infty}}\lesssim \abs{r_{\epsilon}}_{H^{1}}^{3/2}\abs{r_{\epsilon}}_{L^2(\mathbb{R})}^{1/2}\leq \delta\abs{r_{\epsilon}}_{H^{1}}^{2}+c(\delta)\abs{r_{\epsilon}}_{L^2(\mathbb{R})}^{2},}
	\end{array}
\end{equation}
and	
	\begin{equation}
		I_{3}^{\epsilon}=\Inner{ (\partial_x(u_{\epsilon}-u)\cdot\partial_x z_{\epsilon})u_{\epsilon},r_{\epsilon}  }_{L^2(\mathbb{R})}\lesssim\abs{z_{\epsilon}}_{H^{1}}\abs{r_{\epsilon}}_{L^{\infty}(\R)}\leq \delta\abs{r_{\epsilon}}_{H^{1}}^{2}+c(\delta)\abs{z_{\epsilon}}_{H^{1}}^{2}.
	\end{equation}
Next	
	\begin{equation}
		I_{4}^{\epsilon} = 2\Inner{(\partial_x r_{\epsilon}\cdot \partial_x u)u_{\epsilon},r_{\epsilon} }_{L^2(\mathbb{R})}\leq \delta\abs{r_{\epsilon}}_{H^{1}}^{2}+c(\delta)\abs{\partial_x^{2}u}_{L^2(\mathbb{R})}\abs{r_{\epsilon}}_{L^2(\mathbb{R})}^{2},
	\end{equation}
and	
	\begin{equation}
		I_{5}^{\epsilon} = \Inner{\abs{\partial_x u}^{2}z_{\epsilon},r_{\epsilon} }_{L^2(\mathbb{R})}\leq \abs{\partial_x u}_{L^{\infty}}^{2}\abs{z_{\epsilon}}_{L^2(\mathbb{R})}\abs{r_{\epsilon}}_{L^2(\mathbb{R})}\lesssim\abs{z_{\epsilon}}_{L^2(\mathbb{R})}^{2}+\abs{\partial_x^{2}u}_{L^2(\mathbb{R})}^{2}\abs{r_{\epsilon}}_{L^2(\mathbb{R})}^{2}.
	\end{equation}
Finally	
	\begin{equation}
		I_{6}^{\epsilon}= 2\Inner{(\partial_x z_{\epsilon}\cdot \partial_x u)u_{\epsilon},r_{\epsilon} }_{L^2(\mathbb{R})}\lesssim\abs{z_{\epsilon}}_{H^{1}}\abs{\partial_x u}_{L^{\infty}(\R)}\abs{r_{\epsilon}}_{L^2(\mathbb{R})}\lesssim\abs{z_{\epsilon}}_{H^{1}}^{2}+\abs{\partial_x^{2}u}_{L^2(\mathbb{R})}\abs{r_{\epsilon}}_{L^2(\mathbb{R})}^{2},
	\end{equation}
	and
	\begin{equation}
		I_{7}^{\epsilon}=\epsilon^{H/2}\Inner{\abs{\partial_{t}u_{\epsilon}}^{2}u_{\epsilon},r_{\epsilon} }_{L^2(\mathbb{R})}\leq \delta\abs{r_{\epsilon}}_{H^{1}}^{2}+c(\delta)\epsilon^{H-1}\abs{\partial_{t}u_{\epsilon}}_{L^2(\mathbb{R})}^{2}.
	\end{equation}
Concerning the terms $J_j^\epsilon(t)$, we have	
	\begin{equation}
		J_{1}^{\epsilon}=-\epsilon^{H/2}\Inner{\partial_x r_{\epsilon}, \partial_x\partial_{t}u_{\epsilon}}_{L^2(\mathbb{R})}\leq \delta\abs{r_{\epsilon}}_{H^{1}}^{2}+c(\delta)\epsilon^{H}\abs{\partial_{t}u_{\epsilon}}_{H^{1}}^{2},
	\end{equation}
and	
	\begin{equation}
		J_{2}^{\epsilon} = \epsilon^{H/2}\Inner{\abs{\partial_x u}^{2}r_{\epsilon},\partial_{t}u_{\epsilon} }_{L^2(\mathbb{R})}\lesssim\epsilon^{H}\abs{\partial_{t}u_{\epsilon}}_{L^2(\mathbb{R})}^{2}+\abs{\partial_x^{2}u}_{L^2(\mathbb{R})}^{2}\abs{r_{\epsilon}}_{L^2(\mathbb{R})}^{2}.
	\end{equation}
Moreover	
	\begin{equation}
		J_{3}^{\epsilon} =\epsilon^{H/2}\Inner{ (\partial_x(u_{\epsilon}-u)\cdot \partial_x r_{\epsilon})u_{\epsilon},\partial_{t}u_{\epsilon}  }_{L^2(\mathbb{R})}\lesssim\epsilon^{H/2}\abs{r_{\epsilon}}_{H^{1}}\abs{\partial_{t}u_{\epsilon}}_{H^{1}}\leq \delta\abs{r_{\epsilon}}_{H^{1}}^{2}+c(\delta)\epsilon^{H}\abs{\partial_{t}u_{\epsilon}}_{H^{1}}^{2},
	\end{equation}
and	
	\begin{equation}
		J_{4}^{\epsilon} =\epsilon^{H/2}\Inner{ (\partial_x(u_{\epsilon}-u)\cdot \partial_x z_{\epsilon})u_{\epsilon},\partial_{t}u_{\epsilon}  }_{L^2(\mathbb{R})}\lesssim\epsilon^{H/2}\abs{z_{\epsilon}}_{H^{1}}\abs{\partial_{t}u_{\epsilon}}_{H^{1}}\lesssim \abs{z_{\epsilon}}_{H^{1}}^{2}+\epsilon^{H}\abs{\partial_{t}u_{\epsilon}}_{H^{1}}^{2},
	\end{equation}
Next	
	\begin{equation}
		J_{5}^{\epsilon} = 2\,\epsilon^{H/2}\Inner{(\partial_x r_{\epsilon}\cdot \partial_x u)u_{\epsilon},\partial_{t}u_{\epsilon}  }_{L^2(\mathbb{R})}\leq \delta\abs{r_{\epsilon}}_{H^{1}}^{2}+c(\delta)\epsilon^{H}\abs{\partial_{t}u_{\epsilon}}_{H^{1}}^{2},
	\end{equation}
and	
	\begin{equation}
		J_{6}^{\epsilon} = 2\epsilon^{H/2}\Inner{(\partial_x z_{\epsilon}\cdot \partial_x u)u_{\epsilon},\partial_{t}u_{\epsilon}  }_{L^2(\mathbb{R})}\lesssim\abs{z_{\epsilon}}_{H^{1}}^{2}+\epsilon^{H}\abs{\partial_{t}u_{\epsilon}}_{H^{1}}^{2},
	\end{equation}
and	
	\begin{equation}
		J_{7}^{\epsilon} = \epsilon^{H/2}\Inner{ \abs{\partial_x u}^{2}z_{\epsilon},\partial_{t}u_{\epsilon}   }_{L^2(\mathbb{R})}\lesssim\abs{z_{\epsilon}}_{H^{1}}^{2}+\epsilon^{H}\abs{\partial_{t}u_{\epsilon}}_{H^{1}}^{2}.
	\end{equation}
Finally,	
	\begin{equation}
		J_{8}^{\epsilon}=-\epsilon^{H}\Inner{\abs{\partial_{t}u_{\epsilon}}^{2}u_{\epsilon},\partial_{t}u_{\epsilon}}_{L^2(\mathbb{R})}\lesssim\epsilon^{H}\abs{\partial_{t}u_{\epsilon}}_{H^{1}}\abs{\partial_{t}u_{\epsilon}}_{L^2(\mathbb{R})}^{2}\lesssim \epsilon^{H}\big( \abs{\partial_{t}u_{\epsilon}}_{H^{1}}^{2}+\abs{\partial_{t}u_{\epsilon}}_{L^2(\mathbb{R})}^{4}   \big).
	\end{equation}
	Therefore, since  
	\begin{equation}
		\epsilon^{H/2}\Inner{\partial_{t}u_{\epsilon},r_{\epsilon}}_{L^2(\mathbb{R})}\leq \delta\abs{r_{\epsilon}}_{L^2(\mathbb{R})}^{2}+c(\delta)\epsilon^{H}\abs{\partial_{t}u_{\epsilon}}_{L^2(\mathbb{R})}^{2},
	\end{equation}
	thanks to \eqref{uniform_bound1} we can take $\bar{\delta}>0$ sufficiently small so that,  $\P$-almost surely,
	\begin{equation}
		\begin{array}{ll}
			&\ds{\sup_{r\in[0,t]} \abs{r_{\epsilon}(t)}_{L^2(\mathbb{R})}^{2}+\int_{0}^{t}\abs{\partial_x r_{\epsilon}(s)}_{L^2(\mathbb{R})}^{2}ds  }\\
			\vs 
			&\ds{\lesssim \frac{1}{\epsilon^{1-H}}+ \epsilon^{H}\int_{0}^{t}\abs{\partial_{t}u_{\epsilon}(s)}_{H^{1}}^{2}ds+\int_{0}^{t}\abs{z_{\epsilon}(s)}_{H^{1}}^{2}ds+\int_{0}^{t}\left(\abs{\partial_x u(s)}_{L^2(\mathbb{R})}^{2}+1\right)\abs{r_{\epsilon}(s)}_{L^2(\mathbb{R})}^{2}ds }.
		\end{array}
	\end{equation}
	Finally, thanks to \eqref{est_system_H1} and \eqref{est_Z_sum2}, since the mapping
	\[t \in\,(0,T)\mapsto \vert \partial_x ^2u(t)\vert_{L^2(\mathbb{R})}+1 \in\,\mathbb{R},\]
	belongs to $L^1$, 
	a generalized Gronwall's Lemma yields
		\begin{equation}
		\E\sup_{t\in[0,T]}\abs{r_{\epsilon}(t)}_{L^2(\mathbb{R})}^{2}+\E\int_{0}^{T}\abs{r_{\epsilon}(t)}_{H^{1}}^{2}dt\lesssim_{\alpha,T}\epsilon^{-(1/2+\alpha)}+ \epsilon^{H-2}\,\vert u^\e_0-u_0\vert^2_{L^2(\mathbb{R})},
	\end{equation}
	for any $\alpha>0$.
	
\end{proof}

Finally, if we combine \eqref{est_Z_sum1}, \eqref{est_Z_sum2} and \eqref{est_R}, we obtain the following result.
\begin{Lemma}\label{Y}
	Let $H\in[1/2,1)$. Then, for every $T>0$ and $\alpha>0$, we have
	\begin{equation}\label{est_Y}
		\E\sup_{t\in[0,T]}\abs{y_{\epsilon}(t)}_{L^2(\mathbb{R})}^{2}+\E\int_{0}^{T}\abs{y_{\epsilon}(t)}_{H^{1}}^{2}dt\lesssim_{\,\alpha,T}\epsilon^{-(1/2+\alpha)}+ \epsilon^{H-2}\,\vert u^\e_0-u_0\vert^2_{L^2(\mathbb{R})},\ \ \ \ \ \ \  0<\e\ll 1.
	\end{equation}
	
\end{Lemma}

\section{Convergence of $z_\e$}
Our goal is proving that for every $T>0$ and  $H\in(1/2,1)$, 
	\begin{equation}\label{fine1}
		\lim_{\e\to 0}\,\mathbb{E}\,\int_0^T\vert z_{\epsilon}(t)-z(t)\vert^2_{L^2(\mathbb{R})}\,dt=0,
	\end{equation}
where $z_\e$ is the solution of problem \eqref{sm168} and 
	\begin{equation}
		z(t):=\frac 1{\gamma_0} \int_{0}^{t}e^{\frac 1{\gamma_0}(t-s)A}(u(s)\times \partial_{t}u(s))dw^{H}(s),\ \ \ \ t\in[0,T],
	\end{equation}
Here and in what follows $u$ is the solution of the heat flow harmonic map equation \eqref{limiting_equation}.

 As in \eqref{dec}, for any $\epsilon\in (0,1)$ and $t\in[0,T]$, we write $z_{\epsilon} = z_{\epsilon, 1}+z_{\epsilon, 2}$. According to \eqref{initial_condition_rate} and \eqref{est_Z0_H}, we  have  
	\begin{equation}\label{sm100}\lim_{\epsilon\to 0}\,\abs{z_{\epsilon, 1}}_{L^{2}(0,T;L^2(\mathbb{R}))}=0,\end{equation}
	 so, in order to prove \eqref{fine1}, it is sufficient to show that 
	\begin{equation}\label{fine2}
\lim_{\e\to 0}\mathbb{E}\int_0^T	\vert	z_{\epsilon, 2}(t)- z(t)\vert_{L^2(\mathbb{R})}^2\,dt=0. \ \ 
	\end{equation}
Here, we  decompose  $z_{\e,2}$ as
	\begin{equation}
		\begin{array}{ll}
			&\ds{z_{\epsilon, 2}(t)= \frac 1{\gamma_0}\Big(\int_{0}^{t}e^{\frac 1{\gamma_0}(t-s)A}(u_{\epsilon}(s)\times \partial_{t}u_{\epsilon}(s))Q^{\epsilon}dw^{H}(s)-\int_{0}^{t}e^{\frac 1{\gamma_0}(t-s)A}(u_{\epsilon}(s)\times \partial_{t}u_{ \epsilon}(s))dw^{H}(s)\Big) }\\
			 \vs 
			 &\ds {\quad\quad\quad\quad\quad\quad + \frac 1{\gamma_0}\int_{0}^{t}e^{\frac 1{\gamma_0}(t-s)A}\big(u_{\epsilon}(s)\times \partial_{t}u_{\epsilon}(s)\big)dw^{H}(s)=: \zeta_{1}^{\epsilon}(t)+\zeta_{2}^{\epsilon}(t). }
		\end{array}
	\end{equation}
		
\begin{Lemma}
We have 
	\begin{equation}\label{sm101}
		\lim_{\e\to 0}\,\E\int_{0}^{T}\abs{\zeta_{1}^{\epsilon}(t)}_{L^2(\mathbb{R})}^{2}dt=0.
	\end{equation}	
\end{Lemma}
	
\begin{proof}	For every $t\in[0,T]$ we have
	\begin{equation}
		\zeta_{1}^{\epsilon}(t) = \frac 1{\gamma_0}\sum_{k=1}^{\infty}\int_{0}^{t}e^{\frac 1{\gamma_0}(t-s)A}\big[(u(s)\times \partial_{t}u_{\epsilon}(s))(\eta_{\epsilon}\ast \xi_{k}^{H}-\xi_{k}^{H})\big]d\beta_{k}(s),
	\end{equation}
so that, by proceeding as in the proof of Lemma \ref{lemma6.2}, 	we have 
	\begin{equation}
		\begin{array}{l}
			\ds{  \E\abs{\zeta_{1}^{\epsilon}(t)}_{L^2(\mathbb{R})}^{2}    }\\
			\vs 
			\ds{\lesssim \E\int_{0}^{t}\int_{\R^2}\Big\lvert \int_\mathbb{R}e^{-iyz}G_{\frac{t-s}{\gamma_0}}(x-y)(u(s,y)\times \partial_{t}u_{\epsilon}(s,y))dy\Big\rvert^{2}\big\lvert 1-\mathcal{F}\eta_{\epsilon}(z)\big\rvert^{2}\abs{z}^{1-2H}dzdxds }\\
			\vs 
			\ds{ \leq \E\int_{0}^{t}\int_{\R}\int_{\abs{z}>1/\sqrt{\epsilon}}\Big\lvert\int_\mathbb{R} e^{-iyz}G_{\frac{t-s}{\gamma_0}}(x-y)(u(s,y)\times \partial_{t}u_{\epsilon}(s,y))dy\Big\rvert^{2}\mu_H(dz)dxds  }\\
			\vs 
			\ds{+\E\int_{0}^{t}\int_{\R}\int_{\abs{z}\leq 1/\sqrt{\epsilon}}\Big\lvert \int_\mathbb{R} e^{-iyz}G_{\frac{t-s}{\gamma_0}}(x-y)(u(s,y)\times \partial_{t}u_{\epsilon}(s,y))dy\Big\rvert^{2}\big\lvert 1-\mathcal{F}\eta_{\epsilon}(z)\big\rvert^{2}\mu_H(dz)dxds.  }
		\end{array}
	\end{equation}
	Note that, since  $(\mathcal{F}\eta_{\epsilon})(z)= (\mathcal{F}\eta)(\sqrt{\epsilon}z)$, thanks to \eqref{mollifier1}, for any $\abs{z}\leq 1/\sqrt{\epsilon}$ we have 
	\begin{equation}
		(1-\mathcal{F}\eta_{\epsilon}(z))^{2}\abs{z}^{1-2H}\lesssim \abs{\sqrt{\epsilon}z}^{2a}\abs{z}^{1-2H}=\epsilon^{a}\abs{z}^{2a+1-2H}.
	\end{equation}
Hence, we get 
	\begin{equation}
		\begin{array}{ll}
			&\ds{  \E\abs{\zeta_{1}^{\epsilon}(t)}_{L^{2}(\R)}^{2}\lesssim \epsilon^{H-1/2}\E\int_{0}^{t}\int_{\R^2}\Big\lvert \int_\mathbb{R} e^{-iyz}G_{\frac{t-s}{\gamma_0}}(x-y)(u(s,y)\times \partial_{t}u_{\epsilon}(s,y))dy\Big\rvert^{2}dzdxds  }\\
			\vs 
			&\ds{\quad\quad\quad\quad= \epsilon^{H-1/2}\E\int_{0}^{t}\int_{\R^2}\Big\lvert G_{\frac{t-s}{\gamma_0}}(x-y)(u(s,y)\times \partial_{t}u_{\epsilon}(s,y))\Big\rvert^{2}dydxds  }\\
			\vs 
			&\ds{\quad\quad\quad\quad\quad\quad\quad\quad \lesssim \epsilon^{H-1/2} \E\int_{0}^{t}(t-s)^{-1/2}\abs{\partial_{t}u_{\epsilon}(s)}_{L^2(\mathbb{R})}^{2}ds,  }
		\end{array}
	\end{equation}
	and the Young convolution inequality allows to conclude
	\begin{equation}
		\E\int_{0}^{T}\abs{\zeta_{1}^{\epsilon}(t)}_{L^2(\mathbb{R})}^{2}dt \lesssim_{T}\epsilon^{H-1/2}\E\int_{0}^{T}\abs{\partial_{t}u_{\epsilon}(s)}_{L^2(\mathbb{R})}^{2}ds.
	\end{equation}
	Thanks to \eqref{uniform_bound1} and the fact that $H>1/2$, this implies \eqref{sm101}.\end{proof}
	
Therefore, according to  \eqref{sm101}, we obtain \eqref{fine2} once we prove 
	\begin{equation}\label{Z_convergence_step2}
	\lim_{\e\to 0}\mathbb{E}\int_0^T\vert 	\zeta_{2}^{\epsilon}(t)-z(t)\vert^2_{L^2(\mathbb{R})}\,dt=0.
	\end{equation}

\subsection{A few approximation results}

For every $ \e \in\,(0,1)$, we introduce the problem
\begin{equation}
\label{fine13-bis}
\gamma_0\,\partial_t \hat{u}_\epsilon(t)=\partial_x^2 \hat{u}_\epsilon(t)+\vert \partial_x \hat{u}_\epsilon(t)\vert^2\,\hat{u}_\epsilon(t),\ \ \ \ \ \hat{u}_\epsilon(0)=u^\e_0.	
\end{equation}
Thanks to \eqref{initial_condition_rate} and \eqref{initial_condition_rate-bis}, we can apply Lemma \ref{lemB3} and we have
\[\sup_{t\in [0,T]}\abs{\hat{u}_\epsilon(t)-u(t)}_{H^{1}(\R)}^{2}+\int_{0}^{T}\abs{\hat{u}_\epsilon(t)-u(t)}_{H^{2}(\R)}^{2}dt\lesssim \vert u^\e_0-u_0\vert_{H^1(\mathbb{R})}^2.\]
In particular, if $\lambda \in\,(0,2(1-H))$ is the constant introduced in the statement of Theorem \ref{CLT}, 
\begin{equation}\label{fine14-bis}
\begin{array}{l}\ds{\lim_{\e\to 0}\e^{-\lambda}\int_0^T\vert \hat{u}_\epsilon(t)\times \partial_t \hat{u}_\epsilon(t)-u(t)\times \partial_t u(t)\vert_{L^2(\mathbb{R})}^2\,dt}\\[14pt]\ds{\quad \quad \quad =\frac 1{\gamma_0^2}\,\lim_{\e\to 0}\e^{-\lambda}\int_0^T\vert \hat{u}_\epsilon(t)\times \partial_x^2 \hat{u}_\epsilon(t)-u(t)\times \partial_x^2 u(t)\vert_{L^2(\mathbb{R})}^2\,dt=0.}\end{array}	
\end{equation}

Next, for every $\e \in\,(0,1)$ we  denote 
by $\eta_{\epsilon}$ the solution of the problem
\begin{equation}\label{fine11}\epsilon\,\partial_t\eta_{\epsilon}(t)+\gamma_0\,\eta_{\epsilon}(t)=\e^{\lambda} \partial_x^2\, \eta_{\epsilon}(t)+\hat{u}_\epsilon(t)\times \partial_x^2 \hat{u}_\epsilon(t),\ \ \ \ \ \eta_{\epsilon}(0)=u^\e_0\times v^\e_0,\end{equation}		
where $\hat{u}_\epsilon$ is the solution of problem \eqref{fine13-bis} and $\lambda $ is the positive constant above. Moreover, we   denote 
\[\vartheta_\epsilon(t):=u_\epsilon(t)\times \partial_t u_\epsilon(t).\]
It is immediate to check that
\[\epsilon\,\partial_t\vartheta_\e(t)+\gamma_0\,\vartheta_\e(t)=u_\epsilon(t)\times \partial_x^2 u_\epsilon(t)-\e^{1-H/2} \partial_t u_\e(t) Q_\e\partial_t w^H(t),\ \ \ \ \ \vartheta_\epsilon(0)=u^\e_0\times v^\e_0.\]
In the lemma below, we will show that $\vartheta_\epsilon$ and $\eta_\epsilon$  get closer and closer, for $0<\epsilon\ll 1$.

\begin{Lemma}
We have
	\begin{equation}
	\label{fine3}
	\lim_{\e\to 0}\mathbb{E}\int_0^T\vert\vartheta_\e(t)-\eta_\epsilon(t)\vert_{H^{-\sigma}(\mathbb{R})}^2\,dt=0,	
	\end{equation}
	where $\sigma:=H-\frac 12.$

\end{Lemma}
\begin{proof}
	If we define $h_\epsilon(t):=\eta_\epsilon(t)-\vartheta_\epsilon(t)$, we have that $h_\epsilon$ solves the equation
	\[\begin{array}{l}
\ds{\epsilon\,\partial_t h_\epsilon(t)+\gamma_0\,h_\e(t)=\epsilon^\lambda\, \partial_x^2 h_\e(t)+ \epsilon^\lambda\, \partial_x^2 \vartheta_\e(t) +\left( \hat{u}_\epsilon(t)\times \partial_x^2 \hat{u}_\epsilon(t)-u_\epsilon(t)\times \partial_x^2 u_\epsilon(t)\right)}\\[10pt]
\ds{ \quad \quad \quad \quad \quad \quad \quad \quad \quad \quad \quad +\e^{1-H/2} \partial_t u_\e(t) Q_\e\partial_t w^H(t),\ \ \ \ \ h_\epsilon(0)=0.}	
\end{array}\]
Hence, if we define
\begin{equation}\label{fine10-bis}B_\e:=-\frac {\gamma_0}\e\,I+\frac{1}{\e^{1-\lambda}}\, \partial_x^2,\end{equation}
we get
\[\begin{array}{l}
\ds{h_\e(t)= \frac 1{\epsilon^{1-\lambda}}\int_0^t e^{B_\epsilon(t-s)}\, \partial_x^2 \vartheta_\e(s)\,ds+\frac 1\epsilon\int_0^t e^{B_\epsilon(t-s)}\,\left( \hat{u}_\epsilon(s)\times \partial_x^2 \hat{u}_\epsilon(s)-u_\epsilon(s)\times \partial_x^2 u_\epsilon(s)\right)\,ds}\\[16pt]
\ds{\quad \quad \quad \quad \quad \quad \quad +\frac{1}{\epsilon^{H/2}}\int_0^t e^{B_\epsilon(t-s)}\,\partial_t u_\e(s) Q_\e\partial_t w^H(s)=:I_\e^1(t)+I_\e^2(t)+I_\e^3(t).}
\end{array}\]

\smallskip

{\em Study of} $I^1_\epsilon$. For every $\a \in\,(0,1)$ we have
\[\begin{array}{l}
\ds{	\mathbb{E}\int_0^T\vert I^1_\epsilon(t)\vert_{H^{-\sigma}(\mathbb{R})}^2\,dt=\frac 1{\epsilon^{2(1-\lambda)}}\mathbb{E}\int_0^T\Big\vert\int_0^t e^{B_\epsilon(t-s)}\, \partial_x^2 \vartheta_\e(s)\,ds\Big\vert_{H^{-\sigma}(\mathbb{R})}^2\,dt}\\[16pt]
\ds{\quad \quad \leq \frac 1{\epsilon^{2(1-\lambda)}}\,\mathbb{E}\int_0^T\Big(\int_0^t e^{-\frac{\gamma_0}\e (t-s)} ((t-s)/\e^{1-\lambda} )^{-\frac{2-\alpha-\sigma}2}\vert \vartheta_\epsilon(s)\vert_{H^{\alpha}(\mathbb{R})}\,ds\Big)^2\,dt}\\[16pt]
\ds{\quad \quad\quad \quad\leq \frac 1{\epsilon^{2(1-\lambda)}}\,\Big(\int_0^T e^{-\frac{\gamma_0}\e t} (t/\e^{1-\lambda} )^{-\frac{2-\alpha-\sigma}2}\,dt\Big)^2\mathbb{E}\int_0^T\vert \vartheta_\epsilon(t)\vert_{H^{\alpha}(\mathbb{R})}^2\,dt}\\[16pt]
\ds{\quad \quad\quad \quad\quad \quad\quad \quad \lesssim_{\,T,\alpha} e^{\lambda(\alpha+\sigma)}\int_0^T\mathbb{E}\,\vert \vartheta_\epsilon(t)\vert_{H^{\alpha}(\mathbb{R})}^2\,dt.}\end{array}\]
Now, for every $\alpha \in\,(0,1)$ it is immediate to check that
\[\vert \vartheta_\epsilon(t)\vert_{H^{\alpha}(\mathbb{R})}=\vert u_\epsilon(t)\times \partial_t u_\epsilon(t)\vert_{H^{\alpha}(\mathbb{R})}\lesssim \vert \partial_t u_\epsilon(t)\vert_{L^2(\mathbb{R})}^{1-\alpha}\,\vert \partial_t u_\epsilon(t)\vert_{H^1(\mathbb{R})}^{\alpha}\left(1+\vert  u_\epsilon(t)\vert_{\dot{H}^1(\mathbb{R})}^{\alpha}\right).\]
Thus, in view of \eqref{uniform_bound1} and \eqref{uniform_bound2}, we obtain
\[\int_0^T\mathbb{E}\,\vert \vartheta_\epsilon(t)\vert_{H^{\alpha}(\mathbb{R})}^2\,dt\lesssim \e^{-\alpha},\]
so that
\begin{equation}
\label{fine4}
\mathbb{E}\int_0^T\vert I^1_\epsilon(t)\vert_{H^{-\sigma}(\mathbb{R})}^2\,dt\lesssim_{\,T,\alpha}	\e^{-\alpha+\lambda(\alpha+\sigma)}=\e^{-\alpha+\lambda(\alpha+H-1/2)}.
\end{equation}
Since $H>1/2$ and $\lambda>0$, this implies that  we can find $\bar{\alpha}>0$ such that $-\bar{\alpha}+\lambda(\bar{\alpha}+H-1/2)>0$ and
\begin{equation}
\label{fine5}
\lim_{\e\to 0}	\mathbb{E}\int_0^T\vert I^1_\epsilon(t)\vert_{H^{-\sigma}(\mathbb{R})}^2\,dt=0.
\end{equation}

\smallskip

{\em Study of} $I^2_\epsilon$. Notice that
\[ \hat{u}_\epsilon(t)\times \partial_x^2 \hat{u}_\epsilon(t)-u_\epsilon(t)\times \partial_x^2 u_\epsilon(t)= \partial_x\left(\hat{u}_\epsilon(t)\times \partial_x \hat{u}_\epsilon(t)-u_\epsilon(t)\times \partial_x u_\epsilon(t)\right).\]
Then, we have
\[\begin{array}{l}
\ds{	\mathbb{E}\int_0^T\vert I^2_\epsilon(t)\vert_{L^2(\mathbb{R})}^2\,dt=\frac 1{\epsilon^{2}}\mathbb{E}\int_0^T\Big\vert\int_0^t e^{B_\epsilon(t-s)}\, \partial_x \left(\hat{u}_\epsilon(s)\times \partial_x \hat{u}_\epsilon(s)-u_\epsilon(s)\times \partial_x u_\epsilon(s)\right)\,ds\Big\vert_{L^2(\mathbb{R})}^2\,dt}\\[16pt]
\ds{\quad  \leq \frac 1{\epsilon^{2}}\,\mathbb{E}\int_0^T\Big(\int_0^t e^{-\frac{\gamma_0}\e (t-s)} ((t-s)/\e^{1-\lambda} )^{-\frac 1{2}}\vert \hat{u}_\epsilon(s)\times \partial_x \hat{u}_\epsilon(s)-u_\epsilon(s)\times \partial_x u_\epsilon(s)\vert_{L^2(\mathbb{R})}\,ds\Big)^2\,dt}\\[16pt]
\ds{\quad \quad \quad \leq \frac 1{\epsilon^{2}}\,\Big(\int_0^T e^{-\frac{\gamma_0}\e t} (t/\e^{1-\lambda} )^{-\frac 1{2}}\,dt\Big)^2\mathbb{E}\int_0^T\vert \hat{u}_\epsilon(t)\times \partial_x \hat{u}_\epsilon(t)-u_\epsilon(t)\times \partial_x u_\epsilon(t)\vert_{L^2(\mathbb{R})}^2\,dt}\\[16pt]
\ds{\quad \quad\quad \quad \quad \quad\lesssim\,\e^{-\lambda}\,\mathbb{E}\int_0^T\vert \hat{u}_\epsilon(t)\times \partial_x \hat{u}_\epsilon(t)-u_\epsilon(t)\times \partial_x u_\epsilon(t)\vert_{L^2(\mathbb{R})}^2\,dt.}
\end{array}\]
Since 
\[\vert u(t)\times \partial_x u(t)-\hat{u}_\epsilon(t)\times \partial_x \hat{u}_\epsilon(t)\vert_{L^2(\mathbb{R})}\lesssim \vert \partial_x u(t)\vert_{L^2(\mathbb{R})}\,\vert u(t)-\hat{u}_\epsilon(t)\vert_{H^1(\mathbb{R})}+\vert u(t)-\hat{u}_\epsilon(t)\vert_{H^1(\mathbb{R})},\]
in view of Theorem \ref{regularity} and   \eqref{continuity2} we have
\[\begin{array}{l}\ds{\int_0^T\vert u(t)\times \partial_x u(t)-\hat{u}_\epsilon(t)\times \partial_x \hat{u}_\epsilon(t)\vert_{L^2(\mathbb{R})}^2\,dt\leq c_T(\vert u_0\vert_{\dot{H}^1(\mathbb{R})})\int_0^T \vert u(t)-\hat{u}_\epsilon(t)\vert^2_{H^1(\mathbb{R})}\,dt}\\[14pt]
\ds{\quad \quad \quad \quad \quad \quad \leq c_T(\vert u_0\vert_{\dot{H}^1(\mathbb{R})}, \vert u_0^\epsilon\vert_{\dot{H}^1(\mathbb{R})})\vert u^\e_0-u_0\vert_{L^2(\mathbb{R})}^2.}
\end{array}
\]
Hence, thanks to \eqref{initial_condition_rate} and the fact that $\lambda<2(1-H)<2-H$, we get
\begin{equation}\label{fine20-bis}\lim_{\e\to 0}\e^{-\lambda}\,\int_0^T\vert u(t)\times \partial_x u(t)-\hat{u}_\epsilon(t)\times \partial_x \hat{u}_\epsilon(t)\vert_{L^2(\mathbb{R})}^2\,dt=0.\end{equation}
In the same way
thanks to \eqref{est_Y}, for every $\alpha \in\,(0,1)$ we have
\[\e^{-\lambda}\,\mathbb{E}\int_0^T\vert u(t)\times \partial_x u(t)-u_\epsilon(t)\times \partial_x u_\epsilon(t)\vert_{L^2(\mathbb{R})}^2\,dt\lesssim_{\,T,\alpha} \epsilon^{\frac 32-\alpha-H-\lambda}+\e^{-\lambda}\vert u^\e_0-u_0\vert_{L^2(\mathbb{R})}^2.\]
As we are assuming $H\leq 1$ and $\lambda< 2(1-H)<3/2-H$,  we can fix $\bar{\alpha} \in\,(0,1)$ such that 
\[\frac 32-\bar{\alpha}-H-\lambda>0,\ \ \ \ \ \e^{-\lambda}\vert u^\e_0-u_0\vert_{L^2(\mathbb{R})}^2=o( \e^{-\lambda+2-H}),\]
and we  conclude 
\[\lim_{\e\to 0}\e^{-\lambda}\,\mathbb{E}\int_0^T\vert u(t)\times \partial_x u(t)-u_\epsilon(t)\times \partial_x u_\epsilon(t)\vert_{L^2(\mathbb{R})}^2\,dt=0.\]
This, together with \eqref{fine20-bis}, implies
\begin{equation}
\label{fine6}
\lim_{\e\to 0}	\mathbb{E}\int_0^T\vert I^2_\epsilon(t)\vert_{L^2(\mathbb{R})}^2\,dt=0.
\end{equation}

\smallskip

{\em Study of} $I^3_\e$. We have
\[\begin{array}{l}
\ds{\mathbb{E}\int_0^T \vert I^3_\e(t)\vert_{L^2(\mathbb{R})}^2\,dt=\frac{1}{\epsilon^{H}}\mathbb{E}\int_0^T\Big\vert\int_0^t e^{B_\epsilon(t-s)}\,\partial_t u_\e(s) Q_\e\partial_t w^H(s)\Big\vert^2_{L^2(\mathbb{R})}\,dt}\\[16pt]
\ds{\quad \quad \leq \frac{1}{\epsilon^{H}}\int_0^T\int_0^t e^{-\frac{2\gamma_0}{\epsilon}(t-s)}\left(1+\left(|t-s|/\e^{1-\lambda}\right)^{-1/2}\right)\mathbb{E}\,\vert 	\partial_t u_\e(s)\vert^2_{L^2(\mathbb{R})}\,ds\,dt}\\[16pt]
\ds{\quad \quad \quad \quad \leq  \frac{1}{\epsilon^{H}}\int_0^T e^{-\frac{2\gamma_0}{\epsilon}t}\,dt\  \mathbb{E}\int_0^T\vert 	\partial_t u_\e(t)\vert^2_{L^2(\mathbb{R})}\,dt\lesssim \epsilon^{1-H-\frac \lambda 2}.}
\end{array}\]
Therefore, since we assume that $H<1$ and $\lambda<2(1-H)$, we have
\begin{equation}
\label{fine7}
\lim_{\e\to 0}	\mathbb{E}\int_0^T\vert I^3_\epsilon(t)\vert_{L^2(\mathbb{R})}^2\,dt=0.
\end{equation}

\smallskip

{\em Conclusion}. If we collect together \eqref{fine5}, \eqref{fine6} and \eqref{fine7} we get \eqref{fine3}.

\end{proof}

Next we prove the following estimate for the stochastic integral.

\begin{Lemma}
For every $f \in\,L^2(0,T;H^{-\sigma}(\mathbb{R}))$, with $\sigma=H-1/2$,  we have
\begin{equation}
	\label{fine8}
	\mathbb{E}\int_0^T \Big\vert\int_0^t e^{\frac 1{\gamma_0}(t-s)A}f(s)\partial_t w^H(s)\Big \vert_{L^2(\mathbb{R})}^2\,dt\lesssim_{\,T}\mathbb{E}\int_0^T\vert f(t)\vert_{H^{-\sigma}(\mathbb{R})}^2\,dt.
\end{equation}	
\end{Lemma}

\begin{proof}
We have
\[\begin{array}{l}
\ds{\mathbb{E}\,\Big\vert\int_0^t e^{\frac 1{\gamma_0}(t-s)A}f(s)\partial_t w^H(s)\Big \vert_{L^2(\mathbb{R})}^2}\\[16pt]
\ds{\quad \quad =\mathbb{E}\int_0^t \int_{\mathbb{R}}\ \sum_{k=1}^\infty \Big\vert \int_{\mathbb{R}^2} e^{-iyz}G_{\frac{t-s}{\gamma_0}}(x-y)	f(s,y)e_k^H(z)|z|^{1-2H}\,dz\,dy\Big\vert^2\,dx\,ds}\\[16pt]
\ds{\quad \quad \quad \quad =\mathbb{E}\int_0^t \int_{\mathbb{R}}\ \sum_{k=1}^\infty \Big\vert \int_{\mathbb{R}}e^H_k(z)|z|^{1-2H}\int_{\mathbb{R}}e^{-iyz}G_{\frac{t-s}{\gamma_0}}(x-y)	f(s,y)\,dy\,dz\Big\vert^2\,dx\,ds}\\[16pt]
\ds{\quad \quad \quad \quad \quad \quad =\mathbb{E}\int_0^t \int_{\mathbb{R}^2}\  \Big\vert \int_{\mathbb{R}}e^{-iyz}G_{\frac{t-s}{\gamma_0}}(x-y)	f(s,y)\,dy\Big\vert^2|z|^{1-2H}\,dz\,dx\,ds.}
\end{array}\]
Now, we have
\[\begin{array}{l}
\ds{\int_{\mathbb{R}}\Big\vert \int_{\mathbb{R}}e^{-iyz}G_{\frac{t-s}{\gamma_0}}(x-y)	f(s,y)\,dy\Big\vert^2\,dx=\int_{\mathbb{R}}\Big\vert [G_{\frac{t-s}{\gamma_0}}\star e^{-iz\cdot}f(s,\cdot)](x)\Big\vert^2\,dx}	\\[16pt]
\ds{\quad \quad \quad=\int_{\mathbb{R}}\vert \mathcal{F} G_{\frac{t-s}{\gamma_0}}(x)\vert^2 \vert \mathcal{F}(e^{-iz\cdot}f(s,\cdot))(x)\vert^2\,dx=\int_{\mathbb{R}} e^{-\frac{2}{\gamma_0}(t-s)x^2}\vert \mathcal{F}f(s,\cdot)(x+z)\vert^2\,dx. }
\end{array}\]
Hence, we get
\[\begin{array}{l}
\ds{\mathbb{E}\,\Big\vert\int_0^t e^{\frac{1}{\gamma_0}(t-s)A}f(s)\partial_t w^H(s)\Big \vert_{L^2(\mathbb{R})}^2=\mathbb{E}\int_0^t \int_{\mathbb{R}^2}e^{-\frac{2}{\gamma_0}(t-s)x^2}\vert \mathcal{F}f(s,\cdot)(x+z)\vert^2|z|^{1-2H}\,dz\,dx\,ds}\\[16pt]
\ds{\quad \quad \quad \quad =\mathbb{E}\int_0^t \int_{\mathbb{R}^2}e^{-\frac{2}{\gamma_0}(t-s)x^2}\vert \mathcal{F}f(s,\cdot)(\eta)\vert^2\vert\eta-x\vert^2\,d\eta\,dx\,ds.}
\end{array}\]
It is possible to check that
\[\int_{\mathbb{R}^2}e^{-\frac{2}{\gamma_0}(t-s)x^2}\vert\eta-x\vert^2\,dx\lesssim \left(1+(t-s)^{-H}\right)(1+\vert\eta\vert^2)^{1/2-H}.\]
Thus,
\[\begin{array}{l}
\ds{\mathbb{E}\int_0^T\Big\vert\int_0^t e^{\frac 1{\gamma_0}(t-s)A}f(s)\partial_t w^H(s)\Big \vert_{L^2(\mathbb{R})}^2\,dt}\\[16pt]
\ds{\quad \quad \quad \lesssim \mathbb{E}\int_0^T\int_0^t \int_{\mathbb{R}}\left(1+(t-s)^{-H}\right)(1+\vert\eta\vert^2)^{1/2-H}\vert\mathcal{F} f(s,\cdot)(\eta)\vert^2\,d\eta\,ds\,dt }\\[16pt]
\ds{\quad \quad \quad \quad \quad \lesssim_{\,T}\mathbb{E}\int_0^T\int_0^t \left(1+(t-s)^{-H}\right)\vert f(s)\vert_{H^{-\sigma}(\mathbb{R})}^2\,ds\,dt \lesssim_{\,T} \mathbb{E}\int_0^T \vert f(t)\vert_{H^{-\sigma}(\mathbb{R})}^2\,dt.}
\end{array}\]

\end{proof}

\subsection{Conclusion of the proof of \eqref{fine1}}

As we have seen above, \eqref{fine1} is a consequence of \eqref{fine2}, which is a consequence of \eqref{Z_convergence_step2}. Thus, in what follows we are going to prove \eqref{Z_convergence_step2}. In view  of \eqref{fine8}, we have
\[\begin{array}{l}
\ds{\mathbb{E}\int_0^T\vert \zeta^\epsilon_2(t)-z(t)\vert_{L^2(\mathbb{R})}^2\,dt}\\[16pt]
\ds{\quad \quad =\frac 1{\gamma_0^2}\,\mathbb{E}\int_0^T\Big \vert\int_{0}^{t}e^{\frac 1{\gamma_0}(t-s)A}\big(u_{\epsilon}(s)\times \partial_{t}u_{\epsilon}(s)-u(s)\times \partial_{t}u(s)\big)dw^{H}(s)\Big\vert_{L^2(\mathbb{R})}^2\,dt}\\[16pt]
\ds{\quad \quad \quad \quad \quad \quad \lesssim_{\,T}\mathbb{E}\int_0^T\vert u_{\epsilon}(t)\times \partial_{t}u_{\epsilon}(t)-u(t)\times \partial_{t}u(t)\vert^2_{H^{-\sigma}(\mathbb{R})}\,dt.}
	\end{array}\]
	Thanks to \eqref{fine3} and \eqref{fine14-bis},
	 this implies  that
	\[\begin{array}{l}\ds{\limsup_{\e\to 0}\mathbb{E}\int_0^T\vert \zeta^\epsilon_2(t)-z(t)\vert_{L^2(\mathbb{R})}^2\,dt\lesssim \lim_{\e\to 0}\mathbb{E}\int_0^T\vert u_{\epsilon}(t)\times \partial_{t}u_{\epsilon}(t)-\eta_\epsilon(t)\vert^2_{H^{-\sigma}(\mathbb{R})}\,dt}\\[16pt]
	\ds{ +\limsup_{\e\to 0}\int_0^T\vert \eta_\epsilon(t)- \hat{u}_\epsilon(t)\times \partial_{t}\hat{u}_\epsilon(t)\vert^2_{H^{-\sigma}(\mathbb{R})}\,dt +\lim_{\e\to 0}\int_0^T\vert \hat{u}_\epsilon(t)\times \partial_{t}\hat{u}_\epsilon(t)-u(t)\times \partial_{t}u(t)\vert^2_{L^2(\mathbb{R})}\,dt }\\[16pt]
	\ds{\quad \quad \quad \quad =\limsup_{\e\to 0} \int_0^T\vert \eta_\epsilon(t)- \hat{u}_\epsilon(t)\times \partial_{t}\hat{u}_\epsilon(t)\vert^2_{H^{-\sigma}(\mathbb{R})}\,dt.}
	\end{array}\]
	Thus, we obtain \eqref{Z_convergence_step2} and conclude the proof of \eqref{fine1}  once we prove the following result.
	\begin{Lemma}
	It holds
\begin{equation}
\label{fine9}
\lim_{\e\to 0} \int_0^T\vert \eta_\epsilon(t)- \hat{u}_\epsilon(t)\times \partial_{t}\hat{u}_\epsilon(t)\vert^2_{L^2(\mathbb{R})}\,dt=0.	
\end{equation}
		\end{Lemma}
\begin{proof}
Since $\eta_\epsilon$ solves equation \eqref{fine11} and 
\[\hat{u}_\epsilon(t)\times \partial_{x}^2\hat{u}_\epsilon(t)=\gamma_0\,\hat{u}_\epsilon(t)\times \partial_{t}\hat{u}_\epsilon(t),\]
we have
\[\eta_\epsilon(t)=e^{B_\e t}(u^\e_0\times v^\e_0)+\frac {\gamma_0}\epsilon\int_0^t e^{B_\epsilon(t-s)}\left(\hat{u}_\epsilon(s)\times \partial_{t}\hat{u}_\epsilon(s)\right)\,ds=:\eta_{\epsilon,1}(t)+\eta_{\epsilon,2}(t),\] 
where $B_\epsilon$ is the operator defined in \eqref{fine10-bis}.
It is immediate to check that 
\begin{equation} \label{fine12-bis}
	\lim_{\e\to 0}\vert \eta_{\epsilon,1}\vert_{L^2(0,T;L^2(\mathbb{R}))}= 0,
\end{equation}
so that proving \eqref{fine9} reduces to proving
\begin{equation}
\label{fine9-bis}
\lim_{\e\to 0} \int_0^T\vert \eta_{\epsilon,2}(t)-f_\epsilon(t)\vert^2_{L^2(\mathbb{R})}\,dt=0,	
\end{equation}
where we have set
\[f_\epsilon(t,x):=\hat{u}_\epsilon(t,x)\times \partial_t \hat{u}_\epsilon(t,x)=\frac 1{\gamma_0}\,\hat{u}_\epsilon(t,x)\times \partial_x^2\, \hat{u}_\epsilon(t,x).\]

Now, let us consider the Fourier transform of $\eta_{\epsilon,2}(t)$, which is given by 
\begin{equation}
	\mathcal{F}\eta_{\epsilon,2}(t)(x) = \frac{\gamma_0}{\epsilon} \int_{0}^{t}e^{-\big(\frac{\gamma_0}{\epsilon}+\frac{x^2}{\epsilon^{1-\lambda}}\big)(t-s)}\mathcal{F}f_\epsilon(s,x)ds,\ \ \ \ x \in\,\mathbb{R}.
\end{equation}
For every $s \in\,[0,t]$ we write
\begin{equation}
	\mathcal{F}f_{\epsilon}(s)(x) = \mathcal{F}f_{\epsilon}(t)(x)- \int_s^t \partial_{t}(\mathcal{F}f_{\epsilon})(r,x)\,dr,
\end{equation}
so that
\[\begin{array}{l}
\ds{\vert\mathcal{F}\eta_{\epsilon,2}(t)(x)-	\mathcal{F}f_{\epsilon}(t)(x)\vert\leq \left(\Big\vert\frac{\gamma_0}{\epsilon} \big(\frac{\gamma_0}{\epsilon}+\frac{x^2}{\epsilon^{1-\lambda}}\big)^{-1}-1\Big\vert+e^{-\frac{\gamma_0 t}{\e}}\right)\,\vert\mathcal{F}f_{\epsilon}(t)(x)\vert}\\[14pt]
\ds{\quad \quad \quad +\frac{\gamma_0}{\epsilon} \int_{0}^{t}e^{-\big(\frac{\gamma_0}{\epsilon}+\frac{x^2}{\epsilon^{1-\lambda}}\big)(t-s)}\int_s^t \vert \partial_{t}(\mathcal{F}f_{\epsilon})(r,x)\vert\,dr\,ds}\\[14pt]
\ds{\quad \quad \quad \quad \quad \lesssim e^{-\frac{\gamma_0 t}{\e}}\,\vert\mathcal{F}f_{\epsilon}(t)(x)\vert+\epsilon^\lambda x^2\,\vert\mathcal{F}f_{\epsilon}(t)(x)\vert+\e^{1/2}\left(\int_0^t \vert \partial_{t}(\mathcal{F}f_{\epsilon})(r,x)\vert^2\,dr \right)^{1/2}.}
\end{array}\]
This implies
\[\begin{array}{l}
\ds{\int_0^T \vert \eta_{\epsilon,2}(t)-f_{\epsilon}(t)\vert_{L^2(\mathbb{R})}^2\,dt=\int_0^T \vert	\mathcal{F}\eta_{\epsilon,2}(t)-	\mathcal{F}f_{\epsilon}(t)\vert_{L^2(\mathbb{R})}^2\,dt}\\[16pt]
\ds{\quad \quad \quad \quad \lesssim_{\,T} \int_0^T e^{-\frac{2\gamma_0 t}{\e}}\,\vert\mathcal{F}f_{\epsilon}(t)(x)\vert^2\,dt+\epsilon^{2\lambda} \int_0^T \vert f_{\epsilon}(t)\vert^2_{H^2(\mathbb{R})}\,dt+\e\int_0^T \vert \partial_t f_{\epsilon}(t)\vert_{L^2(\mathbb{R})}^2\,dt.}
\end{array}\]
Now, since 
\[\vert f_{\epsilon}(t)\vert_{L^2(\mathbb{R})}\leq \vert \hat{u}_\epsilon(t)\vert_{\dot{H}^2(\mathbb{R})},\ \ \ \ \ \ \vert f_{\epsilon}(t)\vert_{H^2(\mathbb{R})}\lesssim \vert \hat{u}_\epsilon(t)\vert_{\dot{H}^4(\mathbb{R})}\]
and
\[\vert \partial_t f_{\epsilon}(t)\vert_{L^2(\mathbb{R})}=\vert \hat{u}_\e\times \partial_x^2 \partial_t\hat{u}_{\epsilon}(t)\vert_{L^2(\mathbb{R})}\leq \vert \partial_t\hat{u}_{\epsilon}(t)\vert_{H^2(\mathbb{R})},\]
according to Theorem \ref{regularity}, as we are assuming $\lambda <2(1-H)<1/2$, we have
\[\int_0^T \vert \eta_{\epsilon,2}(t)-f_{\epsilon}(t)\vert_{L^2(\mathbb{R})}^2\,dt\lesssim_{\,T} \epsilon\,\vert u^\epsilon_0\vert_{\dot{H}^2(\mathbb{R})}^2+ \e^{2\lambda}\,\vert u^\e_0\vert_{\dot{H}^3(\mathbb{R})}^2.\]
Thanks to \eqref{fine25-bis}, this allows to obtain \eqref{fine9-bis} and our proof is completed.

\end{proof}

\section{Fluctuations of $u_\e$ - Proof of Theorem \ref{CLT}}
\label{secCLT}

In this section, we  study the normal fluctuations for $\{u_{\epsilon}\}_{\epsilon\in (0,1)}$ around its deterministic limit $u$, in the special case the noise ${w}(t)$ in system \eqref{SPDE} is the spacial convolution of a  fractional noise $w^{H}(t)$ of  Hurst index $H\in (1/2,1)$ with a smooth function.
As in Section \ref{sec7}, we assume Hypotheses \ref{H1}, \ref{H1-bis} and \ref{H2}, and take $(u_{0}^{\epsilon},v_{0}^{\epsilon})\in \big(\dot{H}^{2}(\R)\times H^{1}(\R)\big)$ fulfilling condition \eqref{initial_condition_rate}, for some  $u_{0}\in \dot{H}^{1}(\R)\cap M$.

\smallskip

By applying Theorem \ref{small_mass_limit}, we have 
	$u_{\epsilon}$ converges in probability to $u$ in $C([0,T];H^{\delta_{1}}_{\text{loc}}(\R))\cap L^{2}(0,T;H^{\delta_{2}}_{\text{loc}}(\R))$.
for any $\delta_{1}<1$ and $\delta_{2}<2$, where $u\in L^{\infty}(0,T; \dot{H}^{1}(\R))\cap L^{2}(0,T; \dot{H}^{2}(\R))$ is the unique solution of equation \eqref{limiting_equation}.

\subsection{A linear and continuous mapping in $L^{2}(\Omega)$}

We fix  $T>0$ and $\xi\in L^{2}(0,T;L^2(\mathbb{R}))$, and  for every  $v\in L^{2}(0,T;L^2(\mathbb{R}))$ and $t \in\,[0,T]$ we define  
\begin{equation}
	\Theta_{\xi}(v)(t):=\frac 1{\gamma_0}\int_{0}^{t}e^{\frac 1{\gamma_0}(t-s)A}\big(\abs{\partial_x u(s)}^{2}v(s)\big)ds+\frac 2{\gamma_0}\int_{0}^{t}e^{\frac 1{\gamma_0}(t-s)A}\big((\partial_x u(s)\cdot \partial_x v(s))u(s)\big)ds+\xi(t),
\end{equation}
where $u$ is the solution of equation \eqref{limiting_equation}.

\begin{Lemma}\label{linear_continuous_operator}
	The mapping $\Theta_{\xi}:L^{2}(0,T;L^2(\mathbb{R}))\to L^{2}(0,T;L^2(\mathbb{R}))$ is well-defined, and for every $\xi\in L^{2}(0,T;L^2(\mathbb{R}))$ there is a unique $\Lambda(\xi)\in L^{2}(0,T;L^2(\mathbb{R}))$ such that
	\begin{equation}\label{well_posed}
		\Theta_{\xi}(\Lambda(\xi))=\Lambda(\xi).
	\end{equation}
	Moreover, the mapping $\Lambda:L^{2}(0,T;L^2(\mathbb{R}))\to L^{2}(0,T;L^2(\mathbb{R}))$ is linear and continuous.  Namely	\begin{equation}\label{continuous_L2}
		\abs{\Lambda(\xi)}_{L^{2}(0,T;L^2(\mathbb{R}))}\lesssim_{\,T}\abs{\xi}_{L^{2}(0,T;L^2(\mathbb{R}))},\ \ \ \ \ \ \ \ \xi\in L^{2}(0,T;L^2(\mathbb{R})).
	\end{equation}
\end{Lemma}

\begin{proof}
We fix $\xi\in L^{2}(0,T;L^2(\mathbb{R}))$ and  write
\[\Theta_\xi(v)=Lv+\xi,\ \ \ \ \ \ v \in\,L^{2}(0,T;L^2(\mathbb{R})),\]
where 
\[Lv:=\frac 1{\gamma_0}\int_{0}^{t}e^{\frac 1{\gamma_0}(t-s)A}\big(\abs{\partial_x u(s)}^{2}v(s)\big)ds+\frac 2{\gamma_0}\int_{0}^{t}e^{\frac 1{\gamma_0}(t-s)A}\big((\partial_x u(s)\cdot \partial_x v(s))u(s)\big)ds.\]
We have
	\begin{equation}
		\begin{array}{l}
			\ds{ \lvert Lv(t)\big\rvert_{L^2(\mathbb{R})}\lesssim \int_{0}^{t}\Big\lvert e^{\frac 1{\gamma_0}(t-s)A}\big(\abs{\partial_x u(s)}^{2}v(s)\big) \Big\rvert_{L^2(\mathbb{R})}ds}\\
			\vs 
			\ds{ +\int_{0}^{t}\Big\lvert e^{\frac 1{\gamma_0}(t-s)A}\partial_x\big((\partial_x u(s)\cdot v(s))u(s)\big)\Big\rvert_{L^2(\mathbb{R})}ds +\int_{0}^{t}\Big\lvert e^{\frac 1{\gamma_0}(t-s)A}\big((\partial_x u(s)\cdot v(s))\partial_x u(s)\big)\Big\rvert_{L^2(\mathbb{R})}ds}\\ 
			\vs 
			\ds{  \quad\quad\quad\quad \quad \quad\quad\quad +\int_{0}^{t}\Big\lvert e^{\frac 1{\gamma_0}(t-s)A}\big((\partial_x^{2} u(s)\cdot v(s))u(s)\big)\Big\rvert_{L^2(\mathbb{R})}ds =:\sum_{i=1}^{4}I_{i}(t)    }
		\end{array}
	\end{equation}
	Recalling  that
	\begin{equation}
		\sup_{t\in[0,T]}\abs{\partial_x u(t)}_{L^2(\mathbb{R})}^{2}+\int_{0}^{T}\abs{\partial_x^{2}u(s)}_{L^2(\mathbb{R})}^{2}ds<\infty,
	\end{equation}
we take some $\lambda>0$, to be determined later, and in view of the Gagliardo-Nirenberg inequality we get
	\begin{equation}
		\begin{array}{ll}
			&\ds{ \int_0^T e^{-2\lambda t}|I_{1}(t)|^2\,dt\lesssim_{\,T}\int_{0}^{T}\Big(\int_{0}^{t}e^{-\lambda(t-s)}(t-s)^{-1/4}\abs{\partial_x u(s)}_{L^{4}(\R)}^{2}e^{-\lambda s}\abs{v(s)}_{L^2(\mathbb{R})}ds\Big)^{2}dt  }\\
			\vs 
			&\ds{\quad\quad  \lesssim_{\,T}\int_{0}^{T}\Big(\int_{0}^{t}e^{-\lambda(t-s)}(t-s)^{-1/4}\abs{\partial_x^{2} u(s)}_{L^2(\mathbb{R})}^{1/2}e^{-\lambda s}\abs{v(s)}_{L^2(\mathbb{R})}ds\Big)^{2}dt  }\\
			\vs 
			&\ds{\quad\quad \lesssim_{\,T} \Big(\int_{0}^{T}e^{-4\lambda s/3}s^{-1/3}ds\Big)^{3/2}\Big(\int_{0}^{T}\abs{\partial_x^{2}u(s)}_{L^2(\mathbb{R})}^{2}ds\Big)^{1/2}\int_{0}^{T}e^{-2\lambda t}\abs{v(t)}_{L^2(\mathbb{R})}^{2}dt. }
		\end{array}
	\end{equation}
In the same way	
	\begin{equation}
		\begin{array}{ll}
			&\ds{\int_0^T e^{-2\lambda t}|I_{2}(t)|^2\,dt\lesssim_{\,T} \int_{0}^{T}\Big(\int_{0}^{t}e^{-\lambda(t-s)}(t-s)^{-1/2}\abs{\partial_x^{2} u(s)}_{L^2(\mathbb{R})}^{1/2}e^{-\lambda s}\abs{v}_{L^2(\mathbb{R})}ds\Big)^{2}dt  }\\
			\vs 
			&\ds{\quad\quad \lesssim_{\,T}\Big(\int_{0}^{T}e^{-4\lambda s/3}s^{-2/3}ds\Big)^{3/2}\Big(\int_{0}^{T}\abs{\partial_x^{2}u(s)}_{L^2(\mathbb{R})}^{2}ds\Big)^{1/2}\int_{0}^{T}e^{-2\lambda t}\abs{v(t)}_{L^2(\mathbb{R})}^{2}dt,   }
		\end{array}
	\end{equation}
and	
	\begin{equation}
	\begin{array}{l}
	\ds{\int_0^T e^{-2\lambda t}|I_{3}(t)|^2\,dt}\\
	\vs
	\ds{\quad \quad \lesssim_{\,T} \Big(\int_{0}^{T}e^{-4\lambda s/3}s^{-1/3}ds\Big)^{3/2}\Big(\int_{0}^{T}\abs{\partial_x^{2}u(s)}_{L^2(\mathbb{R})}^{2}ds\Big)^{1/2}\int_{0}^{T}e^{-2\lambda t}\abs{v(t)}_{L^2(\mathbb{R})}^{2}dt,	}
	\end{array}
	\end{equation}
Finally	
	\begin{equation}
		\begin{array}{ll}
			&\ds{ \int_0^T e^{-2\lambda t}|I_{4}(t)|^2\,dt\lesssim_{\,T}\int_{0}^{T}\Big(\int_{0}^{t}e^{-\lambda(t-s)}(t-s)^{-1/4}\abs{\partial_x^{2}u(s)}_{L^2(\mathbb{R})}e^{-\lambda s}\abs{v(s)}_{L^2(\mathbb{R})}ds\Big)^{2}dt  }\\
			\vs 
			&\ds{\quad\quad \lesssim_{\,T}\int_{0}^{T}e^{-2\lambda s}s^{-1/2}ds\cdot \int_{0}^{T}\abs{\partial_x^{2}u(s)}_{L^2(\mathbb{R})}^{2}ds\cdot \int_{0}^{T}e^{-2\lambda t}\abs{v(t)}_{L^2(\mathbb{R})}^{2}dt    }.
		\end{array}
	\end{equation}
	Now,  for any $\lambda>0$ and $a\in (0,1)$, we have 
	\begin{equation}
		\int_{0}^{T}e^{-\lambda s }s^{a-1}ds = \int_{0}^{\delta}e^{-\lambda s}s^{a-1}ds+\int_{\delta}^{T}e^{-\lambda s}s^{a-1}ds \lesssim \delta^{a} +\frac{\delta^{a-1}}{\lambda}\exp(-\lambda \delta),\ \ \ \ \ \delta \in\,[0,T],
	\end{equation}
	and, due to the arbitrariness of $\delta$, this implies that for any $a\in (0,1)$,
	\begin{equation}
		\lim_{\lambda\to+\infty}\int_{0}^{T}e^{-\lambda s}s^{a-1}ds =0.
	\end{equation}
	Hence, there exists a constant $\bar{\lambda}>0=\bar{\lambda}(T)$ such that 
	\begin{equation}
		\int_{0}^{T}e^{-2\bar{\lambda} t}\big\lvert Lv(t)\big\rvert_{L^2(\mathbb{R})}^{2}dt \leq \frac{1}{2}\int_{0}^{T}e^{-2\bar{\lambda}t}\abs{v(t)}_{L^2(\mathbb{R})}^{2}dt,
	\end{equation}
so that the linear operator $L$ is a contraction in $L^2(0,T;L^2(\mathbb{R}))$, endowed with the norm
\begin{equation}\label{fine35}\vert \varphi\vert^2_{L^2(0,T;L^2(\mathbb{R}))}=\int_{0}^{T}e^{-2\bar{\lambda} t}\vert \varphi(t)\vert^2_{L^2(\mathbb{R})}\,dt.\end{equation}	
In particular, this implies that $\Theta_\xi$ maps $L^2(0,T;L^2(\mathbb{R}))$ into itself, for every $\xi \in\,L^2(0,T;L^2(\mathbb{R}))$, and, due to the linearity of $L$ is a contraction with respect to the norm \eqref{fine35}. This means that $\Theta_\xi$ 
admits a unique fixed point $\Lambda(\xi)$. Finally, by using  similar arguments, we can show that \eqref{continuous_L2} holds. Since $L$ is linear, we have that $\Lambda$ is linear and \eqref{continuous_L2} holds.
	
\end{proof}

\subsection{Proof of  Theorem \ref{CLT} }

With the same notations introduced in Section \ref{sec-main} and Lemma \ref{linear_continuous_operator}, if $\varrho_\e$ is the solution of the equation 
\begin{equation}\label{fine50}
	\le\{\begin{array}{l}
		\ds{\gamma_0\,\partial_{t}\varrho_{\epsilon}(t,x) = \partial_x^{\,2} \varrho_{\epsilon}(t,x)+\abs{\partial_x u(t,x)}^{2}\varrho_{\epsilon}(t,x)+2(\partial_x \varrho_{\epsilon}(t,x)\cdot \partial_x u(t,x))u_{\epsilon}(t,x) }\\[10pt]
		
		\ds{\quad \quad\quad\quad\quad\quad\quad\quad \quad\quad\quad\quad\quad\quad+\big(u_{\epsilon}(t)\times\partial_{t}u_{\epsilon}(t)\big)Q^{\epsilon}\partial_{t}w^{H}(t,x), }\\
		[10pt]
		\ds{\varrho_{\epsilon}(0,x)=\epsilon^{H/2-1}(u_0^{\epsilon}(x)-u_{0}(x)), }
	\end{array}\r.
\end{equation}
 we have  $\varrho_{\epsilon}=\Lambda(z_{\epsilon})$. Thus, as a consequence of \eqref{fine1} and Lemma \ref{linear_continuous_operator}, we obtain 
\begin{equation}
	\varrho_{\epsilon} \to  \Lambda(z)=:\varrho\ \ \text{in}\ \ L^{2}(\Omega;L^{2}(0,T;L^2(\mathbb{R}))),
\end{equation}
where $\varrho$ is the unique solution of the equation
	\begin{equation}
	\le\{\begin{array}{l}
		\ds{\gamma_0\partial_{t}\varrho(t) = \partial_x^{\,2} \varrho(t)+\abs{\partial_x u(t)}^{2}\varrho(t)+2(\partial_x u(t)\cdot \partial_x \varrho(t))u(t) +\big(u(t)\times\partial_{t}u(t)\big)\partial_{t}w^{H}(t), }\\
		[10pt]
		\ds{\varrho(0)=0 }.
	\end{array}\r.
\end{equation}
In particular,  Theorem \ref{CLT} is proved once we show the following result.

\begin{Lemma}
We have 	\begin{equation}\label{sm82}
		\lim_{\epsilon\to0}\,\E\,\abs{y_\e-\varrho_{\epsilon}}_{L^{2}(0,T;L^2(\mathbb{R}))}=0.
	\end{equation}
\end{Lemma}

\begin{proof}
	 Without loss of generality, we take $\gamma_0=1$. First of all, we notice that if we define
	\begin{equation}
		 \lambda_{\epsilon}(t):=y_\e(t)-\varrho_\e(t)+\e^{H/2}\partial_t u_\e(t),\ \ \ \ \ \ t \in\,[0,T],
	\end{equation}
	then  $\lambda_{\epsilon} = \Lambda(\xi_{\epsilon})$, with 
	\begin{equation}
		\begin{array}{l}
			\ds{ \epsilon^{H/2}\xi_{\epsilon}(t) = - \int_{0}^{t}e^{(t-s)A}\partial_x^{\,2} \partial_{t}u_{\epsilon}(s)ds - \int_{0}^{t}e^{(t-s)A}\big(\abs{\partial_x u(s)}^{2}\partial_{t}u_{\epsilon}(s)\big)ds  }\\
			\vs 
			\ds{ -2\int_{0}^{t}e^{(t-s)A}\big((\partial_x \partial_{t}u_{\epsilon}(s)\cdot \partial_x u(s))u(s)\big)ds + \epsilon^{1-H}\int_{0}^{t}e^{(t-s)A}\big((\partial_x u(s)\cdot \partial_x y_{\epsilon}(s))y_{\epsilon}(s)\big)ds }\\
			\vs 
			\ds{+\epsilon^{1-H}\int_{0}^{t}e^{(t-s)A}\big(\abs{\partial_x y_{\epsilon}(s)}^{2}u_{\epsilon}(s)\big)ds -\int_{0}^{t}e^{(t-s)A}\big(\abs{\partial_{t}u_{\epsilon}(s)}^{2}u_{\epsilon}(s)\big)ds = : \e^{H/2}\sum_{i=1}^{6}\xi_{\epsilon, i}(t).   }
		\end{array}
	\end{equation} 
	Therefore, since
	\begin{equation}
		\lim_{\epsilon\to0}\epsilon^{H}\,\E\abs{\partial_{t}u_{\epsilon}}_{L^{2}(0,T;L^2(\mathbb{R}))}^{2}=0,
	\end{equation}
due to \eqref{continuous_L2} in order to prove \eqref{sm82}  it is enough  to show that 
\begin{equation}
	\lim_{\epsilon\to0}\E\abs{\xi_{\epsilon, i}}_{L^{2}(0,T;L^2(\mathbb{R}))}=0.
\end{equation}
for every $i=1,\ldots,6$. If we fix an arbitrary $\alpha\in (0,1)$, we have 
\begin{equation}
	\begin{array}{l}
		\ds{  \E\int_{0}^{T}\abs{\xi_{\epsilon, 1 }(t)}_{L^2(\mathbb{R})}^{2}dt \lesssim \epsilon^{H}\,\E\int_{0}^{T}\Big(\int_{0}^{t}\Big\lvert e^{(t-s)A}\partial_x^{\,2} \partial_{t}u_{\epsilon}(s)\Big\rvert_{L^2(\mathbb{R})}ds\Big)^{2}dt   }\\
		\vs 
		\ds{\quad\quad\quad  \lesssim\epsilon^{H}\,\E\int_{0}^{T}\Big(\int_{0}^{t}(t-s)^{-1+\alpha/2}\abs{\partial_{t}u_{\epsilon}(s)}_{H^{\alpha}}ds\Big)^{2}dt    }\\
		\vs 
		\ds{\quad\quad\quad\quad \lesssim_{\,\a}\epsilon^{H} \Big(\int_{0}^{T}s^{-1+\alpha/2}ds\Big)^{2}\E\int_{0}^{T} \abs{\partial_{t}u_{\epsilon}(s)}_{H^{1}}^{2\alpha}\abs{\partial_{t}u_{\epsilon}(s)}_{L^2(\mathbb{R})}^{2(1-\alpha)}ds   }\\
		\vs 
		\ds{\quad\quad \lesssim_{\,\a, T}\epsilon^{H}\E\,\left(\Big(\int_{0}^{T}\abs{\partial_{t}u_{\epsilon}(s)}_{H^{1}}^{2}ds\Big)^{\alpha}\Big(\int_{0}^{T}\abs{\partial_{t}u_{\epsilon}(s)}_{L^2(\mathbb{R})}^{2}ds\Big)^{2(1-\alpha)} \right) \lesssim_{\,\alpha,T}\epsilon^{H-\alpha}.  }
	\end{array}
\end{equation}
Thus, if we pick $\alpha\in(0,H)$, we get
\begin{equation} \label{sm90}
	\lim_{\epsilon\to0} \E\int_{0}^{T}\abs{\xi_{\epsilon, 1}(t)}_{L^2(\mathbb{R})}^{2}dt =0.
\end{equation}
Next,
\begin{equation}
	\begin{array}{ll}
		&\ds{ \E\int_{0}^{T}\abs{\xi_{\epsilon, 2}(t)}_{L^2(\mathbb{R})}^{2}dt \lesssim \epsilon^{H}\E\int_{0}^{T}\Big(\int_{0}^{t}(t-s)^{-1/4}\abs{\partial_x^{2}u(s)}_{L^2(\mathbb{R})}^{1/2}\abs{\partial_{t}u_{\epsilon}(s)}_{L^2(\mathbb{R})}ds \Big)^{2}dt }\\
		\vs 
		&\ds{\quad\quad\quad\quad  \lesssim_{\,T}\epsilon^{H}\E\Big(\int_{0}^{T}\abs{\partial_x^{2}u(s)}_{L^2(\mathbb{R})}^{1/2}\abs{\partial_{t}u_{\epsilon}(s)}_{L^2(\mathbb{R})}ds\Big)^{2} \lesssim_{\,T}\epsilon^{H}\E\int_{0}^{T}\abs{\partial_{t}u_{\epsilon}(s)}_{L^2(\mathbb{R})}^{2}ds  },
	\end{array}
\end{equation}
which implies 
\begin{equation}\label{sm91}
	\lim_{\epsilon\to0} \E\int_{0}^{T}\abs{\xi_{\epsilon, 2}(t)}_{L^2(\mathbb{R})}^{2}dt =0.
\end{equation}
Moreover, since 
\begin{equation}
	(\partial_x \partial_{t}u_{\epsilon}\cdot \partial_x u)u=\partial_x\big((\partial_{t}u_{\epsilon}\cdot \partial_x u)u\big)-(\partial_{t}u_{\epsilon}\cdot \partial_x u)\partial_x u-(\partial_{t}u_{\epsilon}\cdot \partial_x^{2}u)u,
\end{equation}
 we have 
\begin{equation}
	\begin{array}{l}
		\ds{ \E\int_{0}^{T}\abs{\xi_{\epsilon, 3}(t)}_{L^2(\mathbb{R})}^{2}dt \lesssim\epsilon^{H} \E\int_{0}^{T}\Big(\int_{0}^{t}(t-s)^{-1/2}\abs{\partial_{t}u_{\epsilon}(s)}_{L^2(\mathbb{R})}\abs{\partial_x^{2}u(s)}_{L^2(\mathbb{R})}^{1/2}ds\Big)^{2}dt   }\\
		\vs 
		\ds{\quad\quad + \epsilon^{H}\E\int_{0}^{T}\Big(\int_{0}^{t}(t-s)^{-1/4}\abs{\partial_{t}u_{\epsilon}(s)}_{L^2(\mathbb{R})}\abs{\partial_x^{2}u(s)}_{L^2(\mathbb{R})}^{1/2}ds\Big)^{2}dt    }\\
		\vs 
		\ds{\quad\quad\quad\quad  +\epsilon^{H}\E\int_{0}^{T}\Big(\int_{0}^{t}(t-s)^{-1/4}\abs{\partial_{t}u_{\epsilon}(s)}_{L^2(\mathbb{R})}\abs{\partial_x^{2}u(s)}_{L^2(\mathbb{R})}ds\Big)^{2}dt  }\\
		\vs 
		\ds{\quad \lesssim\epsilon^{H}\Big(\int_{0}^{T}\big(s^{-2/3}+s^{-1/3}\big)ds\Big)^{3/2} \Big(\int_{0}^{T}\abs{\partial_x^{2}u(s)}_{L^2(\mathbb{R})}^{2}ds\Big)^{1/2}\int_{0}^{T}\abs{\partial_{t}u_{\epsilon}(s)}_{L^2(\mathbb{R})}^{2}ds   }\\
		\vs 
		\ds{\quad\quad\quad\quad  +\epsilon^{H}\Big(\int_{0}^{T}s^{-1/2}ds\Big)\Big(\int_{0}^{T}\abs{\partial_x^{2}u(s)}_{L^2(\mathbb{R})}^{2}ds\Big) \int_{0}^{T}\abs{\partial_{t}u_{\epsilon}(s)}_{L^2(\mathbb{R})}^{2}ds },
	\end{array}
\end{equation}
which implies  
\begin{equation}\label{sm92}
	\lim_{\epsilon\to0} \E\int_{0}^{T}\abs{\xi_{\epsilon,3}(t)}_{L^2(\mathbb{R})}^{2}dt =0.
\end{equation}
Next, thanks to \eqref{est_Y}, we have
\begin{equation}
	\begin{array}{ll}
		&\ds{ \E\Big(\int_{0}^{T}\abs{\xi_{\epsilon, 4}(t)}_{L^2(\mathbb{R})}^{2}dt\Big)^{1/2}  \lesssim\epsilon^{1-H/2}\,\E\Bigg(\int_{0}^{T}\Big(\int_{0}^{t}(t-s)^{-1/4}\abs{y_{\epsilon}(s)}_{H^{1}}^{2}ds\Big)^{2}dt\Bigg)^{1/2}    }\\
		\vs 
		&\ds{\quad\quad\quad\quad \lesssim\epsilon^{1-H/2}\,\E\Bigg(\Big(\int_{0}^{T}s^{-1/2}ds\Big)\Big(\int_{0}^{T}\abs{y_{\epsilon}(s)}_{H^{1}}^{2}ds\Big)^{2}\Bigg)^{1/2}  }\\
		\vs 
		&\ds{\quad\quad \lesssim_{\,T}\epsilon^{1-H/2}}\,\E\int_{0}^{T}\abs{y_{\epsilon}(s)}_{H^{1}}^{2}ds \lesssim_{\,\alpha,T}\epsilon^{(1-H)/2-\alpha}+\e^{H/2-1}\vert u^\e_0-u_0\vert_{L^2(\mathbb{R})},
	\end{array}
\end{equation}
for any $\alpha>0$. Hence thanks to \eqref{initial_condition_rate} we conclude that 
\begin{equation}\label{sm93}
	\lim_{\epsilon\to0}\E\Big(\int_{0}^{T}\abs{\xi_{\epsilon, 4}(t)}_{L^2(\mathbb{R})}^{2}dt\Big)^{1/2} =0.
\end{equation}
Similarly, we obtain 
\begin{equation}\label{sm94}
	\lim_{\epsilon\to0}\E\Big(\int_{0}^{T}\abs{\xi_{\epsilon, 5}(t)}_{L^2(\mathbb{R})}^{2}dt\Big)^{1/2} =0.
\end{equation}
Finally, since 
\begin{equation}
	\begin{array}{ll}
		&\ds{ \E\Big(\int_{0}^{T}\abs{\xi_{\epsilon, 6}(t)}_{L^2(\mathbb{R})}^{2}dt\Big)^{1/2} \lesssim\epsilon^{H/2}\, \E\Bigg(\int_{0}^{T}\Big(\int_{0}^{t}(t-s)^{-1/4}\abs{\partial_{t}u_{\epsilon}(s)}_{L^2(\mathbb{R})}^{2}ds\Big)^{2}dt\Bigg)^{1/2} }\\
		\vs 
		&\ds{\quad\quad\quad\quad\quad\quad \lesssim_{\,T}\epsilon^{H/2}\,\E\int_{0}^{T}\abs{\partial_{t}u_{\epsilon}(s)}_{L^2(\mathbb{R})}^{2}ds },
	\end{array}
\end{equation}
we get
\begin{equation}\label{sm95}
	\lim_{\epsilon\to0}\E\Big(\int_{0}^{T}\abs{\xi_{\epsilon, 6}(t)}_{L^2(\mathbb{R})}^{2}dt\Big)^{1/2} =0.
\end{equation}
This concludes the proof of 
\[\lim_{\epsilon\to0}\E\abs{\xi_{\epsilon, i}}_{L^{2}(0,T;L^2(\mathbb{R}))}=0,\ \ \ \ \ \ \ i=1,\ldots,6,\]
and \eqref{sm82} follows.

\end{proof}

\begin{Remark}\label{rem6.1}
{\em   
   	
 We would like to emphasize that the only passage in the whole proof of Theorem \ref{CLT} where we need the Hurst parameter $H$ to be strictly larger than $1/2$ is in the proof of limit \eqref{sm101}. Moreover, the only passage where we need $H$ to be strictly less that $1$ is in the proof of limit \eqref{fine7}.

}	
\end{Remark}

\section{Back to the  convergence of $u_\e$ and $\partial_t u_\e$}

 Estimate \eqref{est_Y} has a very important consequence. Actually, it implies that if $H\in[1/2,1)$, then, for every $T>0$ and $\alpha>0$,  
	\begin{equation}\label{conv-rate}
		\E\sup_{t\in[0,T]}\abs{u_{\epsilon}(t)-u(t)}_{L^2(\mathbb{R})}^{2}+\E\int_{0}^{T}\abs{u_{\epsilon}(t)-u(t)}_{H^{1}}^{2}dt\lesssim_{\,\alpha,T}\epsilon^{3/2-H-\alpha}+ \,\vert u^\e_0-u_0\vert^2_{L^2(\mathbb{R})}.
	\end{equation}

	In particular, if we fix $u_0$ and a sequence $\{u^\e_0\}_{\e \in\,(0,1)}$ fulfilling 
\eqref{initial_condition_rate} , and take some $\a \in\,(0,(1-H)/2)$,  
	we conclude 
	\begin{equation}\label{limit-strong}
		\lim_{\e\to 0}\,\E\sup_{t\in[0,T]}\abs{u_{\epsilon}(t)-u(t)}_{L^2(\mathbb{R})}^{2}+\E\int_{0}^{T}\abs{u_{\epsilon}(t)-u(t)}_{H^{1}}^{2}dt=0.
	\end{equation}
	This means that if the noise has the special form $\eta\ast w_H$, for some kernel $\eta$ satisfying Hypothesis \ref{H2} and a Gaussian noise $w^H(t)$ having a spatially homogeneous correlation of fractional type, with Hurst coefficient $H \in\,[1/2,1)$,   and if the initial data  satisfy condition \eqref{initial_condition_rate}, then the convergence in probability of $u_\e$ to $u$ in $C([0,T];L^2_{\text{loc}}(\R))\cap L^2(0,T;H^1_{\text{loc}}(\R))$ proved in  Theorem  \ref{small_mass_limit} can be improved to mean-square convergence in $C([0,T];L^2(\R))\cap L^2(0,T;H^1(\R))$. Moreover,  a bound on the rate of convergence is given, depending on the initial conditions.	

\smallskip

We conclude this section with the following estimate which involves $\partial_t u_\e$ as well.

\begin{Lemma}\label{lem8.5}
	Let $H\in [1/2,1)$ and assume that $u_0\in \dot{H}^{3}(\R)\cap M$, with 
	\begin{equation}\label{initial_condition_rate_additional}
		\big\lvert (u_{0}^{\epsilon}-u_0,\sqrt{\epsilon}v_{0}^{\epsilon} ) \big\rvert_{H^{1}(\R)\times L^2(\mathbb{R})} =O(\epsilon^{\beta}), \ \ \ \ 0<\epsilon\ll1,
	\end{equation}
	for some $\beta>0$. Then, for every $T>0$  we have
	\begin{equation}
		\E\sup_{t\in[0,T]}\Big(\abs{u_{\epsilon}(t)-u(t)}_{H^{1}}^{2}+\epsilon\abs{\partial_{t}u_{\epsilon}(t)-\partial_{t}u(t)}_{L^2(\mathbb{R})}^{2}\Big) +  \E\int_{0}^{T}\abs{\partial_{t}u_{\epsilon}(t)-\partial_{t}u(t)}_{L^2(\mathbb{R})}^{2}dt \lesssim_{\,T} c_0 +\epsilon^{1\wedge 2\beta},
	\end{equation}
	for $0<\epsilon\ll1$.
\end{Lemma}

\begin{proof}
	
If for every  $\epsilon\in (0,1)$ and $t\in[0,T]$ we denote $u_{R}^{\epsilon}(t):= u_{\epsilon}(t)-u(t)$,  we have 
	\begin{equation}
		\le\{\begin{array}{l}
			\ds{\epsilon \partial_{t}^{2}u_{R}^{\epsilon} = \partial_x^{\,2} u_{R}^{\epsilon} +\abs{\partial_x u}^{2}u_{R}^{\epsilon}-\gamma_0\partial_{t}u_{R}^{\epsilon}+\abs{\partial_x u_{R}^{\epsilon}}^{2}u_{\epsilon}+2(\partial_x u\cdot \partial_x u_{R}^{\epsilon})u_{\epsilon}  }\\
			[10pt]
			\ds{\quad\quad\quad\quad\quad\quad-\epsilon\partial_{t}^{2}u-\epsilon\abs{\partial_{t}u_{\epsilon}}^{2}u_{\epsilon}+\sqrt{\epsilon}(u_{\epsilon}\times \partial_{t}u_{\epsilon})\partial_{t}w^{\epsilon}, } \\
			[10pt]
			\ds{u_{R}^{\epsilon}(0)=u_{0}^{\epsilon}-u_0,\ \ \ \ \partial_{t} u_{R}^{\epsilon}(0)=v_0^{\epsilon}-\tilde{v}_0 },
		\end{array}\r.
	\end{equation}
	where $\tilde{v}_0:=\frac 1{\gamma_0}\big(D^{\,2} u_0+\abs{Du_0}^{2}u_0\big)$. If we denote
	\begin{equation}
		v=\partial_{t}u,\ \ \ \ \ v_{\epsilon}= \partial_{t}u_{\epsilon},\ \ \ \ \ v_{R}^{\epsilon}=\partial_{t}u_{R}^{\epsilon}=v_\epsilon-v,
	\end{equation}
recalling that $(u_\epsilon,v_\epsilon)\in \mathcal{M}$, 	from the It\^{o} formula we get
	\begin{equation}
		\begin{array}{ll}
			&\ds{  \frac{1}{2}d\Big(\abs{\partial_x u_{R}^{\epsilon}(t)}_{L^2(\mathbb{R})}^{2}+\epsilon\abs{v_{R}^{\epsilon}(t)}_{L^2(\mathbb{R})}^{2}\Big) +\gamma_0\abs{v_{R}^{\epsilon}(t)}_{L^2(\mathbb{R})}^{2}dt  = \Inner{\abs{\partial_x u(t)}^{2}u_{R}^{\epsilon}(t),v_{R}^{\epsilon}(t)  }_{L^2(\mathbb{R})}dt   }\\
			\vs 
			&\ds{ \quad \quad  -\Inner{\abs{\partial_x u_{R}^{\epsilon}(t)}^{2}u_{\epsilon},v(t) }_{L^2(\mathbb{R})}dt -2\Inner{ (\partial_x u(t)\cdot \partial_x u_{R}^{\epsilon}(t))u_{R}^{\epsilon}(t),v(t)  }_{L^2(\mathbb{R})}dt  }\\
			\vs 
			&\ds{\quad \quad \quad -\epsilon\Inner{\partial_{t}^{2}u(t),v_{R}^{\epsilon}(t)}_{L^2(\mathbb{R})}dt+\epsilon\Inner{\abs{v_{\epsilon}(t)}^{2}u_{\epsilon}(t),v(t) }_{L^2(\mathbb{R})}dt +\frac{c_0}{2}\abs{v_{\epsilon}(t)}_{L^2(\mathbb{R})}^{2}dt }\\
			\vs
			&\ds{\quad \quad \quad \quad \quad \quad +\sqrt{\epsilon}\Inner{v_{R}^{\epsilon}(t), (u_{\epsilon}(t)\times v_{\epsilon}(t))dw^{\epsilon}(t)}_{L^2(\mathbb{R})}  }.
		\end{array}
	\end{equation}
	For any $\delta>0$,
	\begin{equation}
	\begin{array}{l}
	\ds{
		\big\lvert\Inner{\abs{\partial_x u}^{2}u_{R}^{\epsilon},v_{R}^{\epsilon}}_{L^2(\mathbb{R})} \big\rvert \leq \delta\abs{v_{R}^{\epsilon}}_{L^2(\mathbb{R})}^{2}+c(\delta)\abs{\partial_x u}_{L^{4}}^{4}\abs{u_{R}^{\epsilon}}_{H^{1}}^{2}}\\
		\vs 
		\ds{\quad \quad \quad \quad \quad \quad \leq \delta\abs{v_{R}^{\epsilon}}_{L^2(\mathbb{R})}^{2}+c(\delta)\abs{\partial_x^2 u}_{L^2(\mathbb{R})}\abs{\partial_x u}_{L^2(\mathbb{R})}^{3}\abs{u_{R}^{\epsilon}}_{H^{1}}^{2},}	
	\end{array}
\end{equation}
	and
	\begin{equation}
		\epsilon\big\lvert \Inner{\partial_{t}^{2}u,v_{R}^{\epsilon}}_{L^2(\mathbb{R})} \big\rvert \leq \delta\abs{v_{R}^{\epsilon}}_{L^2(\mathbb{R})}^{2}+c(\delta)\epsilon^{2}\abs{\partial^{2}_{t}u}_{L^2(\mathbb{R})}^{2}.
	\end{equation}
	Moreover,
	\begin{equation}
		\big\lvert \Inner{\abs{\partial_x u_{R}^{\epsilon}}^{2}u_{\epsilon},v}_{L^2(\mathbb{R})} \big\rvert+2\big\lvert\Inner{(\partial_x u\cdot \partial_x u_{R}^{\epsilon})u_{R}^{\epsilon},v}_{L^2(\mathbb{R})} \big\rvert \lesssim\big(1+\abs{\partial_x  u}_{L^2(\mathbb{R})}\big)\abs{v}_{H^{1}}\abs{ u_{R}^{\epsilon}}_{H^{1}}^{2},
	\end{equation}
	and
	\begin{equation}
		\epsilon \big\lvert \Inner{\abs{v_{\epsilon}}^{2}u_{\epsilon},v}_{L^2(\mathbb{R})}\big\rvert \leq \epsilon\abs{v}_{H^{1}}\abs{v_{\epsilon}}_{L^2(\mathbb{R})}^{2}\lesssim\epsilon\abs{v}_{H^{1}}\abs{v}_{L^2(\mathbb{R})}^{2}+\epsilon\abs{v}_{H^{1}}\abs{v_{R}^{\epsilon}}_{L^2(\mathbb{R})}^{2}.
	\end{equation}
	By the Burkholder-Davis-Gundy inequality, for any $\delta>0$ we have
	\begin{equation}
		\begin{array}{ll}
			&\ds{  \sqrt{\epsilon}\ \E\sup_{r\in[0,t]}\Big\lvert \int_{0}^{r}\Inner{v_{R}^{\epsilon}(r),(u_{\epsilon}(r)\times v_{\epsilon}(r))dw^{\epsilon}(r)}_{L^2(\mathbb{R})}\Big\rvert    }\\
			\vs 
			&\ds{ \quad \quad \quad  \lesssim \sqrt{c_0\epsilon}\ \E\Big( \int_{0}^{t}\abs{v_{R}^{\epsilon}(r)}_{L^2(\mathbb{R})}^{2} \abs{v_{\epsilon}(r)}_{L^2(\mathbb{R})}^{2}dr\Big)^{1/2}}\\
			\vs
			&\ds{ \quad \quad \quad \quad \quad \quad \leq \delta \epsilon\, \E\sup_{r\in[0,t]}\abs{v_{R}^{\epsilon}(r)}_{L^2(\mathbb{R})}^{2} +c_0\,c(\delta)\,\E\int_{0}^{t}\abs{v_{\epsilon}(s)}_{L^2(\mathbb{R})}^{2}dt. }
		\end{array}
	\end{equation}
	Hence, if we pick some $\delta>0$ sufficiently small, by first integrating in time, then taking the supremum and finally taking the expectations, we get 
	\begin{equation}
		\begin{array}{l}
			\ds{ \E\sup_{r\in[0,t]}\Big(\abs{u_{R}^{\epsilon}(r)}_{H^{1}}^{2}+\epsilon\abs{v_{R}^{\epsilon}(r)}_{L^2(\mathbb{R})}^{2}\Big)+\E\int_{0}^{t}\abs{v_{R}^{\epsilon}(s)}_{L^2(\mathbb{R})}^{2}ds \lesssim c_0\,\E\int_{0}^{t}\abs{v_{\epsilon}(s)}_{L^2(\mathbb{R})}^{2}ds   }\\
			\vs 
			\ds{ \quad\quad\quad +\abs{\partial_x u^{\epsilon}_{R}(0)}_{L^2(\mathbb{R})}^{2}+\epsilon\abs{v_{R}^{\epsilon}(0)}_{L^2(\mathbb{R})}^{2}+\epsilon\Big(\int_{0}^{t}\abs{v(s)}_{H^{1}}\abs{v(s)}_{L^2(\mathbb{R})}^{2}ds+\epsilon\int_{0}^{t}\abs{\partial_{t}^{2}u(s)}_{L^2(\mathbb{R})}^{2}ds\Big) }\\
			\vs 
			\ds{ +\int_{0}^{t}\Big(\abs{\partial_x ^{2}u(s)}_{L^2(\mathbb{R})}\abs{\partial_x u(s)}_{L^2(\mathbb{R})}^{3}+(1+\abs{\partial_x u(s)}_{L^2(\mathbb{R})})\abs{v(s)}_{H^{1}}\Big)\E\big(\abs{u_{R}^{\epsilon}(s)}_{H^{1}}^{2}+\epsilon\abs{v_{R}^{\epsilon}(s)}_{L^2(\mathbb{R})}^{2}\big)ds   }.
		\end{array}
	\end{equation}
	Therefore, from \eqref{initial_condition_rate_additional}, together with \eqref{fine30}, for $k=3$, by applying Gronwall's inequality, we complete the proof.

\end{proof}

In particular, in the deterministic case, that is when when $c_0=0$, we get 	\begin{equation}
		\abs{u_{\epsilon}-u}_{L^{\infty}(0,T;H^{1})}+\abs{\partial_{t}u_{\epsilon}-\partial_{t}u}_{L^{2}(0,T;L^2(\mathbb{R}))} \lesssim_{T}\epsilon^{(1/2)\wedge \beta},\ \ \ \ 0<\epsilon\ll 1.
	\end{equation}
Thus,	unlike in the stochastic case, we have the strong convergence of $\partial_{t}u_{\epsilon}$  to $\partial_{t}u$ in $L^{2}(0,T;L^2(\mathbb{R}))$. Notice that this provides an alternative proof to the result proved in \cite{zarnescu}, where the quite more delicate case of space dimension $d>1$ is considered, but where a  higher regularity of initial conditions is assumed, and, of course, the convergence can only be local in time.

\appendix

\section{Uniqueness and further regularity properties for the heat flow  harmonic map equation}
\label{AppB}

Here, we prove a few results  concerning   the   equation
\begin{equation}\label{heat_flow}
	\le\{\begin{array}{l}
		\ds{ \gamma_0\,\partial_{t}u(t,x) = \partial_{x}^{2} u(t,x)+\abs{\partial_{x} u(t,x)}^{2}u(t,x)  , \ \ \ \ (t,x)\in \mathbb{R}^{+}\times \mathbb{R},}\\[10pt]
		\ds{u(0,x)=u_{0}(x)\in \mathbb{S}^{2},\ \ \ \ x\in \mathbb{R} }.
	\end{array}\r.
\end{equation}

\subsection{Uniqueness of the solution of equation \eqref{heat_flow}}\label{subA1}

We start giving a proof of Theorem  \ref{uniqueness}. If $u_{1}$ and $u_{2}$ are any two solutions of \eqref{heat_flow}, and  we denote $\rho:=u_{1}-u_{2}$,  since $\rho(0,x)=0$ we have   
	$\rho \in\,L^2(0,T;H^2(\mathbb{R}))$. In particular, this implies that $\rho(t)\cdot \partial_{x}\rho(t) \in L^{1}(\R)$ and $\partial_{x}\big(\rho(t)\cdot \partial_{x}\rho(t)\big)\in L^{1}(\R)$, so that 
	\begin{equation*}
		\Inner{\rho(t),\partial_{x}^{2}\rho(t) }_{L^{2}(\R)} = \int_{\R}\partial_{x}\big(\partial_{x}\rho(t)\cdot \rho(t)\big)dx - \abs{\partial_{x}\rho(t)}_{L^{2}(\R)}^{2}= - \abs{\partial_{x}\rho(t)}_{L^{2}(\R)}^{2},
	\end{equation*}
and
	\begin{equation*}
		\begin{array}{ll}
			&\ds{  \frac{\gamma_0}{2}\frac{d}{dt}\abs{\rho(t)}_{L^{2}(\R)}^{2} = -\abs{\partial_{x}\rho(t)}_{L^{2}(\R)}^{2} + \Inner{\big(\abs{\partial_{x}u_{1}(t)}^{2}-\abs{\partial_{x}u_{2}(t)}^{2}\big)u_{1}(t),\rho(t)}_{L^{2}(\R)} }\\
			\vs
			&\ds{\quad\quad\quad\quad\quad\quad\quad\quad\quad\quad\quad\quad  +\Inner{\abs{\partial_{x}u_{2}(t)}^{2}\rho(t),\rho(t)}_{L^{2}(\R)} }.
		\end{array}
	\end{equation*}
	Thanks to  the continuous embedding $H^{1}(\R)\subset L^{\infty}(\R)$, we have 
	\begin{equation}
		\begin{array}{ll}
		\ds{\Big\lvert \Inner{\big(\abs{\partial_{x}u_{1}(t)}^{2}-\abs{\partial_{x}u_{2}(t)}^{2}\big)u_{1}(t),\rho(t)}_{L^{2}(\R)} \Big\rvert}\\
		\vs
		\ds{\quad\quad\quad\quad \leq \frac{1}{2}\abs{\partial_{x}\rho(t)}_{L^{2}(\R)}^{2} + c\Big(\abs{\partial_{x}^{2}u_{1}(t)}_{L^{2}(\R)}^{2}+\abs{\partial_{x}^{2}u_{2}(t)}_{L^{2}(\R)}^{2}\Big)\abs{\rho(t)}_{L^{2}(\R)}^{2} },
		\end{array}
	\end{equation}
	and 
	\begin{equation}
		\Inner{\abs{\partial_{x}u_{2}(t)}^{2}\rho(t),\rho(t)}_{L^{2}(\R)}\lesssim \abs{\partial_{x}^{2}u_{2}(t)}_{L^{2}(\R)}^{2}\abs{\rho(t)}_{L^{2}(\R)}^{2}.
	\end{equation}
This implies that 	\begin{equation}
		\gamma_0\,\frac{d}{dt}\abs{\rho(t)}_{L^{2}(\R)}^{2} +\abs{\partial_{x}\rho(t)}_{L^{2}(\R)}^{2} \lesssim \Big(\abs{\partial_{x}^{2}u_{1}(t)}_{L^{2}(\R)}^{2}+\abs{\partial_{x}^{2}u_{2}(t)}_{L^{2}(\R)}^{2}\Big)\abs{\rho(t)}_{L^{2}(\R)}^{2}, 
	\end{equation}
so that, since $u_1, u_2 \in\,L^2(0,T,\dot{H}^2(\mathbb{R}))$, we can apply Gronwall's Lemma and  get $\rho(t)\equiv0$.

\subsection{An approximation result}

Next, we show how we can approximate functions in $M$ by smooth functions still in $M$.
\begin{Lemma}\label{lemB1}
	For every  $u\in \dot{H}^{1}(\R)\cap M$ there exists a sequence $\{u_n\}_{n\geq 1}\subset \bigcap_{k\in\mathbb{N}} \dot{H}^{k}(\R)\cap M$ such that $u-u_n \in\,L^2(\mathbb{R})$, for every $n \in\,\mathbb{N}$, with 
\begin{equation}
\label{sm195}
\vert 	u-u_n\vert_{L^2(\mathbb{R})}\lesssim \frac 1n\,\vert Du\vert_{L^2(\mathbb{R})},\ \ \ \ \ n \in\,\mathbb{N},	
\end{equation}
and
	\[\lim_{n \to \infty}
	\vert Du_n-Du\vert_{L^2(\mathbb{R})}=0.\]
	\end{Lemma}

\begin{proof}
Let $\eta\in C^{\infty}_{c}(\R)$ be the standard mollifier, and for  every $n\in\mathbb{N}$ and $x\in \R$ let $\eta_{n}(x):=n\eta(nx)$. For every  $u\in \dot{H}^{1}(\R)\cap M$ we define $\tilde{u}_{n}:=\eta_{n}\ast u$. As known, we  have $\tilde{u}_{n}\in \bigcap_{k\in\mathbb{N}}\dot{H}^{k}(\R)$ and  $\tilde{u}_n$ converges to $u$ in $\dot{H}^{1}(\R)$, as $n\to\infty$. Moreover, since   \[\abs{u(x)-u(y)}\leq c\sqrt{\abs{x-y}},\ \ \ \ \ \ \ x, y \in\,\mathbb{R},\] we get
	\[\vert u(x)-\tilde{u}_n(x)\vert \leq \int_{\mathbb{R}}\eta_n(y)\vert u(x)-u(x-y)\vert\,dy\lesssim \int_{\mathbb{R}}\eta_n(y)\sqrt{|y|}\,dy\lesssim \frac 1{\sqrt{n}},\ \ \ \ \ \ x \in\,\mathbb{R},\]
	so that $\tilde{u}_n$ converges uniformly to $u$, as $n\to \infty$.
In particular, there exists some $n_0 \in\,\mathbb{N}$ such that 
	\begin{equation}\label{sm191}\inf_{x \in\,\mathbb{R}}\,\abs{\tilde{u}_{n}(x)}\geq \frac 12,\ \ \ \  \ \ \ n\geq n_0.\end{equation}
	
Now, we show that $u-\tilde{u}_n \in\,L^2(\mathbb{R})$, for all $n \in\,\mathbb{N}$, and
\begin{equation}\label{sm190}
\vert 	u-\tilde{u}_n\vert_{L^2(\mathbb{R})}\lesssim \frac 1n\,\vert Du\vert_{L^2(\mathbb{R})},\ \ \ \ \ n \in\,\mathbb{N}.
\end{equation}
Actually, for every $n \in\,\mathbb{N}$ we have
\[u(x)-\tilde{u}_n(x)=\int_{-1/n}^{1/n}\eta_n(y)\int_{x-y}^x Du(s)\,ds\,dy,\ \ \ \ \ \ \ x \in\,\R,\]
and hence
\[\begin{array}{l}
\ds{\vert u(x)-\tilde{u}_n(x)\vert^2\leq \int_{\mathbb{R}}|\eta_n(y)|^2\,dy\,\int_{-1/n}^{1/n}\left(\int_{x-y}^x|Du(s)|\,ds\right)^2\,dy\leq n\int_0^{1/n} y \int_{x-y}^{x+y}|Du(s)|^2\,ds\,dy.}
\end{array}\]
This implies
\[\begin{array}{l}
\ds{\vert u-\tilde{u}_n\vert_{L^2(\mathbb{R})}^2\leq n\int_0^{1/n} y \int_{\mathbb{R}}\int_{x-y}^{x+y}|Du(s)|^2\,ds\,dx\,dy=n\int_0^{1/n} y \int_{\mathbb{R}}|Du(s)|^2\int_{s-y}^{s+y}dx\,ds\,dy}\\
\vs
\ds{\quad \quad \quad \quad \quad \quad \quad \quad =2n\int_0^{1/n} y^2\,dy\int_{\mathbb{R}}|Du(s)|^2\,ds=\frac{2}{3n^2}\,\vert Du\vert_{L^2(\mathbb{R})}^2,}	
\end{array}\]
and \eqref{sm190} follows.

Next, we define	 
	\begin{equation}
		u_{n}(x):=\frac{\tilde{u}_{n}(x)}{\abs{\tilde{u}_{n}(x)}},\ \ \ \ x\in\R,\ \ \ \ \ \ \ n\geq n_0.
	\end{equation}
Due to \eqref{sm191}, $u_n$ is well-defined and belongs to  $\bigcap_{k\in\mathbb{N}} \dot{H}^{k}(\R)\,\cap\,M$. 
Moreover, 
\[\vert u(x)-u_n(x)\vert\leq \frac{\vert u(x)-\tilde{u}_n(x)\vert}{\vert \tilde{u}_n(x)\vert}+\frac{\vert u(x)\vert\,\vert \vert u(x)\vert-\vert \tilde{u}_n(x)\vert \vert}{\vert \tilde{u}_n(x)\vert}\leq 4\,\vert u(x)-\tilde{u}_n(x)\vert,\]
and this implies that $u-u_n \in\,L^2(\mathbb{R})$ and thanks to \eqref{sm190}, \eqref{sm195} follows.

As for the convergence in $\dot{H}^1(\mathbb{R})$, we have  
	\begin{equation}
		Du_{n}(x) = \frac{D\tilde{u}_{n}(x)}{\abs{\tilde{u}_{n}(x)}} - \frac{(D\tilde{u}_{n}(x)\cdot \tilde{u}_{n}(x)  )\tilde{u}_{n}(x)  }{\abs{\tilde{u}_{n}(x)}^{3}},\ \ \ \ x\in\R,\ \ \ n \geq n_0.
	\end{equation}
	Thus, we have
	\begin{equation}
		\begin{array}{ll}
			&\ds{ \int_{\R}\abs{ Du_{n}(x) - D\tilde{u}_{n}(x) }^{2}dx \lesssim \int_{\R}\frac{\big\lvert \abs{\tilde{u}_{n}(x)} -1 \big\rvert^{2}}{\abs{\tilde{u}_{n}(x)}^{2}}\abs{D\tilde{u}_{n}(x)}^{2}dx + \int_{\R}\frac{ \big\lvert (D\tilde{u}_{n}(x)\cdot \tilde{u}_{n}(x) ) \big\rvert^{2} }{\abs{\tilde{u}_{n}(x)}^{4}}dx   }\\
			\vs 
			&\ds{\quad\quad  \lesssim \int_{\R}\frac{\big\lvert \abs{\tilde{u}_{n}(x)} -1 \big\rvert^{2}}{\abs{\tilde{u}_{n}(x)}^{2}} \abs{D\tilde{u}_{n}(x)-Du(x)}^{2}dx + \int_{\R}\frac{\big\lvert \abs{\tilde{u}_{n}(x)} -1 \big\rvert^{2}}{\abs{\tilde{u}_{n}(x)}^{2}}\abs{Du(x)}^{2}dx   }\\
			\vs 
			&\ds{\quad\quad+ \int_{\R}\frac{ \big\lvert ( (D\tilde{u}_{n}(x)-Du(x)) \cdot \tilde{u}_{n}(x) ) \big\rvert^{2} }{\abs{\tilde{u}_{n}(x)}^{4}}dx +\int_{\R}\frac{ \big\lvert ( Du(x) \cdot \tilde{u}_{n}(x) ) \big\rvert^{2} }{\abs{\tilde{u}_{n}(x)}^{4}}dx =:\sum_{i=1}^{4}I_{i}(n) }.
		\end{array}
	\end{equation}
For $I_1(n)$ and $I_3(n)$, we have		\begin{equation}
		I_{1}(n)+I_{3}(n) \lesssim \int_{\R}\abs{D\tilde{u}_{n}(x)-Du(x)}^{2}dx \to 0,\ \ \ \ \text{as}\ n\to\infty.
	\end{equation}
For $I_{2}(n)$, since  $\tilde{u}_n$ converges uniformly to $u$, for every $\e>0$ we can find $n_\e \in\,\mathbb{R}$ such that 
\begin{equation}
\label{sm192}
\sup_{x \in\,\mathbb{R}}\vert \vert\tilde{u}_n(x)\vert -1\vert=\sup_{x \in\,\mathbb{R}}\vert \vert\tilde{u}_n(x)\vert -\vert u(x)\vert\vert	<\epsilon,\ \ \ \ \ \ n\geq n_\e.
\end{equation}
 This means that 	\begin{equation}
I_2(n)\leq 		4\epsilon\int_{\mathbb{R}} \abs{Du(x)}^{2}dx,\ \ \ \ \ n\geq n_\e,
	\end{equation}
and due to the arbitrariness of $\epsilon$ we conclude that $I_{2}(n)\to 0$, as $n\to\infty$. Finally, for $I_4(n)$ thanks to \eqref{sm192} and to the fact that $Du(x)\cdot u(x)=0$, for every $n\geq n_\e$ we have 
	\begin{equation}
		\begin{array}{ll}
			&\ds{ I_{4}(n)= \int_{\mathbb{R}} \frac{\big\lvert (Du(x)\cdot (\tilde{u}_{n}(x)-u(x)) ) \big\rvert^{2}}{\vert \tilde{u}_n(x)\vert^4}\,dx \lesssim  \int_{\mathbb{R}} \vert Du(x)\vert^2 \vert\tilde{u}_{n}(x)-u(x)\vert^2\,dx \leq \e\,\vert Du\vert_{L^2(\mathbb{R})}^2,}
		\end{array}
	\end{equation} 
	so  that $I_{4}(n)\to 0$, as $n\to\infty$. All this allows us to conclude that
	\begin{equation}
		\lim_{n\to\infty} \int_{\R}\abs{ Du_{n}(x) - D\tilde{u}_{n}(x) }^{2}dx=0,
	\end{equation}
and recalling that $\tilde{u}_{n}$ converges to $u$ in $\dot{H}^{1}(\R)$, we obtain that $Du_n$ converges to $Du$ in $L^2(\R)$.
\end{proof}

\subsection{Regularity of solutions}

We start by proving some uniform bounds for smooth solutions of equation \eqref{heat_flow}. Then, by an approximation argument, we will  extend these bounds to solutions with lower regularity. Throughout this section, whenever we write $c_{k,T}(r)$ we denote a function depending on $k$, $T$ and $r$, which is increasing with respect to $r\geq 0$.

\begin{Lemma}\label{lemmaA2}
	If $u_{0}\in \bigcap_{k\geq1}\dot{H}^{k}(\R)\,\cap\, M$, then the solution $u$ of equation \eqref{heat_flow} belongs to the space $ L^{\infty}(0,T;\bigcap_{k\geq 1}\dot{H}^{k}(\R))$,  for every $T>0$.
\end{Lemma}

\begin{proof}
	If $u_{0}\in \bigcap_{k\geq1}\dot{H}^{k}(\R)\,\cap\, M$, we have that
	$u_{0}\in C_{b}^{\infty}(\R)$. Then, from the combination of Theorem \ref{uniqueness} and  Theorem \ref{small_mass_limit} and from the results in \cite[Chapter 5]{lin2008} (see also \cite{topping1996}), we have that equation \eqref{heat_flow} admits a unique solution 
	\begin{equation}
		u\in C^{\infty}((0,T)\times \R)\cap L^{\infty}(0,T;\dot{H}^{1}(\R))\cap L^{2}(0,T;\dot{H}^{2}(\R)),
	\end{equation}
for every $T>0$, 	with 
	\begin{equation}\label{heat_flow_H1}
		\sup_{t\in[0,T]}\abs{\partial_{x}u(t)}_{L^{2}(\R)}^{2} +\int_{0}^{T}\abs{\partial_{x}^{2}u(t)}_{L^{2}(\R)}^{2}dt\leq c_{T}\big(\abs{u_0}_{\dot{H}^{1}(\R)}\big).
	\end{equation}
	In what follows, our aim is proving that $u\in L^{\infty}(0,T; \dot{H}^{k}(\R) )$, for all $k\in\mathbb{N}$ and $T>0$.
	
\smallskip
	
	{\em Step 1. } We show that $\partial_{x}^{k}u\in L^{\infty}((0,T)\times \R)$, for any $T>0$ and $k\in\mathbb{N}$, with
	\begin{equation}\label{heat_flow_L_infty}
		\sup_{t \in\,[0,T]}\,\abs{\partial_{x}^{k}u(t)}_{L^{\infty}(\R)}\leq c_{k,T}\big(\abs{u_0}_{\dot{H}^{k}(\R)}\big)\big(1+\abs{\partial_{x}^{k+1}u_0 }_{L^{2}(\R)} \big).
	\end{equation}
	
	We will proceed with an induction argument.
	It can be easily verified that 
	\begin{equation}
		\frac{1}{2}\big(\partial_t\abs{\partial_{x}u}^{2}-\partial_x^2\abs{\partial_{x}u}^{2} \big) = \abs{\partial_{x}u}^{4}-\abs{\partial_{x}^{2}u}^{2}=-\vert\partial_{x}^{2}u+\abs{\partial_{x}u}^{2}u \vert^{2}\leq 0,
	\end{equation}
and, as a consequence of the maximum principle,  we get
	\begin{equation}
		\abs{\partial_{x}u(t)}_{L^{\infty}(\R)}^{2}\leq \abs{Du_{0}}_{L^{\infty}(\R)}^{2}\leq \abs{Du_{0}}_{L^{2}(\R)}\,\abs{D^{2}u_{0}}_{L^{2}(\R)},\ \ \ \ t\geq0.
	\end{equation}
	
	In general,  we have
\[\frac{1}{2}\big(\partial_t\abs{\partial_{x}^{k}u}^{2}-\partial_x^2\abs{\partial_{x}^{k}u}^{2} \big)=-\vert \partial_x^{k+1} u\vert^2+\partial_k^k u \cdot \left(\partial_t\partial_x^k u-\partial_x^2\partial_x^{k}\right)=-\vert \partial_x^{k+1} u\vert^2+\partial_k^k u \cdot \partial_x^k(\vert \partial_x u\vert^2u). \]
Now, for every $k \in\,\mathbb{N}$ it holds
	\begin{equation}\label{general_formula}
\partial_x^k(\vert\partial_x u\vert^2u)=		\sum_{i=0}^{k}\partial_{x}^{i}u\sum_{j=1}^{k+1-i}\alpha^i_{j,k}\,(\partial_{x}^{j}u\cdot \partial_{x}^{k+2-i-j}u),
	\end{equation}
for some constants $\alpha^i_{j,k} \in\,\mathbb{N}$. Moreover, as a consequence of $\vert u(x)\vert\equiv 1$, we have
\begin{equation}
\label{A8-bis}	
u\cdot \partial_{x}^{k}u = \sum_{i=1}^{k-1}\beta_{i,k}(\partial_{x}^{i}u\cdot \partial_{x}^{k-i}u),
\end{equation}
for some constants $\beta_{j,k} \in\,\mathbb{N}$.
Therefore,
 we have 
	\begin{equation}
		\begin{array}{ll}
			&\ds{ \frac{1}{2}\big(\partial_t\abs{\partial_{x}^{k}u}^{2}-\partial_x^2\abs{\partial_{x}^{k}u}^{2} \big) = -\abs{\partial_{x}^{k+1}u}^{2}+ (\alpha^0_{1,k}+\alpha^0_{k+1,k})(\partial_{x}u\cdot \partial_{x}^{k+1}u)(u\cdot \partial_{x}^{k}u) }\\
			\vs 
			&\ds{\quad\quad\quad + \sum_{j=2}^{k}\alpha^0_{j,k}(\partial_{x}^{j}u\cdot \partial_{x}^{k+2-j}u )(u\cdot \partial_{x}^{k}u )  + \sum_{i=1}^{k}(\partial_{x}^{i}u\cdot \partial_x^k u)\sum_{j=1}^{k+1-i}\alpha^i_{j,k}\left(\partial_{x}^{j}u\cdot \partial_{x}^{k+2-i-j}u\right)      }\\
			\vs 
			&\ds{\leq -\frac{1}{2}\abs{\partial_{x}^{k+1}u}^{2}+c \abs{\partial_{x}u}^{2}\abs{\partial_{x}^{k}u}^{2} +\sum_{j=2}^{k}\alpha^0_{j,k}\vert\partial_{x}^{j}u\cdot \partial_{x}^{k+2-j}u\vert\sum_{i=1}^{k-1}\beta_{i,k}\vert\partial_{x}^{i}u\cdot \partial_{x}^{k-i}u\vert  }\\
			\vs 
			&\ds{\quad\quad\quad \quad\quad\quad \quad\quad\quad + \sum_{i=1}^{k}\left(\partial_{x}^{i}u\cdot \partial_x^k u\right)\sum_{j=1}^{k+1-i}\alpha^i_{j,k}\left(\partial_{x}^{j}u\cdot \partial_{x}^{k+2-i-j}u\right) }.
		\end{array}
	\end{equation}
	Hence, if we assume that 
\eqref{heat_flow_L_infty} is true for all $j\leq k-1$, we can find  a function $c_{k,T}\big(\abs{u_0}_{\dot{H}^{k}(\R)}\big)>0$ such that
	\begin{equation}
		\partial_t\abs{\partial_{x}^{k}u(t)}^{2}-\partial_x^2 \abs{\partial_{x}^{k}u(t)}^{2} \leq c_{1,k}\big(\abs{u_0}_{\dot{H}^{k}(\R)}\big)\big(1+\abs{\partial_{x}^{k}u(t)}^{2}\big).
	\end{equation}
	and from the generalized maximum principle \eqref{heat_flow_L_infty} follows for $k$. This implies that  \eqref{heat_flow_L_infty} holds for every $k\geq 1$.
\smallskip
	
	{\em Step 2. } We show that $\partial_{x}^{k}u\in L^{\infty}(0,T;L^2(\R))$, for any $T>0$ and $k\in\mathbb{N}$, with
	\begin{equation}\label{heat_flow_L_2}
		\sup_{t \in\,[0,T]}\,\abs{\partial_{x}^{k}u(t)}_{L^{2}(\R)}\leq c_{k,T}\big(\abs{u_0}_{\dot{H}^{k+1}(\R)}\big).
	\end{equation}

As in Step 1, we proceed with an induction argument.	In view of  \eqref{heat_flow_H1}, our statement is true for $k=1$. Now, we fix an arbitrary $k\geq 2$ and assume that \eqref{heat_flow_L_2} is true for all $j\leq k-1$.
	Since
	\begin{equation}
	u(t) = e^{tA}u_{0}+ \int_{0}^{t}e^{(t-s)A}\big(\abs{\partial_{x}u(s)}^{2}u(s)\big)ds,\ \ \ \  t\in [0,T],
	\end{equation}
we have
\begin{equation}\label{fine10}\begin{array}{l}
\ds{\vert \partial_x^k u(t)\vert_{L^2(\R)}\lesssim 	\abs{D^{k}u_{0}}_{L^{2}(\R)}+\int_0^t (t-s)^{-1/2}\vert \partial_x^{k-1}\big(\abs{\partial_{x}u(s)}^{2}u(s)\big)\vert_{L^{2}(\R)}\,ds.}
\end{array}
\end{equation}
Now, since we are assuming that \eqref{heat_flow_L_2} is true for all $j\leq k-1$, thanks  to \eqref{heat_flow_L_infty} and  \eqref{general_formula}, we have
	\begin{equation}\label{fine20}
		\sup_{t \in\,[0,T]}\abs{ \partial_{x}^{k-1}\big(\abs{\partial_{x}u(t)}^{2}u(t)\big) }_{L^{2}(\R)}^{2}\leq c_{k,T}(\abs{u_0}_{\dot{H}^{k+1}(\R)}\big),
	\end{equation}
	and this, together with \eqref{fine10}, implies that  \eqref{heat_flow_L_2}
holds for $k$, and hence for all $k\geq 1$. 	
	
	\end{proof}

\begin{Remark}
	{\em Notice that as a consequence of the fact that, under the same assumptions of Lemma \ref{lemmaA2}, the solution of \eqref{heat_flow} belongs to $L^\infty(0,T;\bigcap_{k\geq 1}\dot{H}^{k}(\R))$ we have that under the same assumptions $u \in\,C^\infty([0,T];\bigcap_{k\geq 1}\dot{H}^{k}(\R))$.}
\end{Remark}

\begin{Lemma}\label{lemA4}
	Let $T>0$ and $u_0\in \bigcap_{k\geq1}\dot{H}^{k}(\R)\,\cap\, M$. 
 If $u\in \bigcap_{k\geq 1} C^{\infty}([0,T];\dot{H}^{k}(\R))$ is a solution of \eqref{heat_flow}, with initial condition $u_0$, then for every $k\in \mathbb{N}$ we have 
		\begin{equation}\label{boundedness}
			\sup_{t\in [0,T]}\abs{u(t)}_{\dot{H}^{k}(\R)}^{2} + \int_{0}^{T}\abs{u(t)}_{\dot{H}^{k+1}(\R)}^{2}dt+\int_{0}^{T}\abs{\partial_{t}u(t)}_{H^{k-1}(\R)}^{2}dt \leq c_{k,T}\big(\abs{u_0}_{\dot{H}^{k}(\R)}\big).
		\end{equation}
				\end{Lemma}

\begin{proof}
We will provide the proof of \eqref{boundedness} for the cases when $k=1$. Then we will assume that \eqref{boundedness} is true for every $j\leq k$ and we will show that this implies it is true for $k$. By induction this gives the validity of \eqref{boundedness} for every $k\geq 1$.
	
	{\em Step 1.} We show here that \eqref{boundedness} holds for $k=1$.
	Since $u\in \bigcap_{k\geq 1} C^{\infty}([0,T];\dot{H}^{k}(\R))$, we have
\[\int_{\R} \partial_x\left(\partial_xu(t)\cdot\partial_t u(t)\right)\,dx=0,\ \ \ \ \ \ t \in\,[0,T].\]	
Hence,	integrating by parts, we obtain 
	\begin{equation}
		\gamma_0\abs{\partial_t u(t)}_{L^2(\R)}^{2} = \Inner{\partial_{x}^{2}u(t), \partial_t u(t)}_{L^2(\R)}+\Inner{\abs{\partial_{x}u(t)}^{2}u(t),\partial_t u(t)}_{L^2(\R)} = -\frac{1}{2}\frac{d}{dt}\abs{\partial_{x}u(t)}_{L^2(\R)}^{2},
	\end{equation}
	and this implies that for every $t\in[0,T]$
	\begin{equation}\label{case1_boundedness1}
		\abs{\partial_{x}u(t)}_{L^2(\R)}^{2}+2\gamma_0\int_{0}^{t}\abs{\partial_t u(s)}_{L^2(\R)}^{2}ds = \abs{D u_{0          }}_{L^2(\R)}^{2}.
	\end{equation}
	Moreover, since
	\begin{equation}
		\int_{0}^{T}\abs{\partial_{x}^{2}u(t)}_{L^2(\R)}^{2}dt \leq c\int_{0}^{T}\abs{\partial_t u(t)}_{L^2(\R)}^{2}+c\int_{0}^{T}\abs{\partial_{x}u(t)}_{L^{4}(\R)}^{4}dt,
	\end{equation}
	 thanks to the interpolation inequality $\abs{h}_{L^{4}(\R)}\leq c\,\abs{\partial_x h}_{L^2(\R)}^{1/4}\abs{h}_{L^2(\R)}^{3/4}$, for $h\in H^{1}(\R)$,  we have 
	\begin{equation}
		\int_{0}^{T}\abs{\partial_{x}^{2}u(t)}_{L^2(\R)}^{2}dt \leq c\int_{0}^{T}\abs{\partial_x u(t)}_{L^2(\R)}^{2}+\frac{1}{2}\int_{0}^{T}\abs{\partial_{x}^{2}u(t)}_{L^2(\R)}^{2}dt+c\int_{0}^{T}\abs{\partial_{x}u(t)}_{L^2(\R)}^{6}dt.
	\end{equation}
	Hence, by \eqref{case1_boundedness1} we conclude that
	\begin{equation}\label{case1_boundedness2}
		\int_{0}^{T}\abs{\partial_{x}^{2}u(t)}_{L^2(\R)}^{2}dt \lesssim_{\,T} \big(\abs{D u_{0}}_{L^2(\R)}^{2}+\abs{D u_{0}}_{L^2(\R)}^{6}\big).
	\end{equation}
	This, together with \eqref{case1_boundedness1}, implies \eqref{boundedness} for $k=1$.
\smallskip
	 
{\em Step 2.} We assume that \eqref{boundedness} holds for every $j\leq k-1$, and we show that this implies that it is true for $k$. 

If we define $\varphi_k=\partial_x^k u$, we have
\[\partial_t \varphi_k(t)=\partial_x^2 \varphi_k(t)+\partial_x^k\left(\vert \partial_x u(t)\vert^2u(t)\right),\ \ \ \ \ \varphi_k(0)=D^k u_0.\]
Hence
\begin{equation}\label{fine15}\frac 12\frac{d}{dt}\vert \varphi_k(t)\vert^2_{L^2(\R)}=\langle 
\partial_x^2 \varphi_k(t),\varphi_k(t)\rangle_{L^2(\R)}+\langle \partial_x^k\left(\vert \partial_x u(t)\vert^2u(t)\right),\varphi_k(t)\rangle_{L^2(\R)}.\end{equation}
As we have seen above for the case $k=1$, due to the fact that $u(t) \in\,\bigcap_{k\geq 1} \dot{H}^{k}(\R)$, we have
\begin{equation}\label{fine16}\langle 
\partial_x^2 \varphi_k(t),\varphi_k(t)\rangle_{L^2(\R)}=-\frac 12\vert \partial_x \varphi_k(t)\vert^2_{L^2(\R)}.\end{equation} 
Moreover, 
if we define
\[A(t):=\partial_x^k\left(\vert \partial_x u(t)\vert^2u(t)\right)\cdot\varphi_k(t),\]
in view of \eqref{general_formula} we have
\[\begin{array}{l}
\ds{A(t)=\vert \partial_xu(t)\vert^2\vert \partial_x^ku(t)\vert^2+(\alpha^0_{1,k}+\alpha^0_{k+1,k})(\partial_{x}u(t)\cdot \partial_{x}^{k+1}u(t))(u(t)\cdot \partial_{x}^{k}u(t))}\\[10pt]
\ds{\quad \quad \quad \quad \quad +\left(u(t)\cdot \partial_x^ku(t) \right)\sum_{j=2}^{k}\alpha^0_{j,k}\,\left(\partial_{x}^{j}u(t)\cdot \partial_{x}^{k+1-j}u(t)\right)}\\[10pt]
\ds{\quad \quad \quad \quad \quad \quad \quad \quad \quad \quad +\sum_{i=1}^{k-1}\left(\partial_{x}^{i}u(t)\cdot\partial_x^ku(t) \right) \sum_{j=1}^{k+1-i}\alpha^i_{j,k}\,\left(\partial_{x}^{j}u(t)\cdot \partial_{x}^{k+2-i-j}u(t)\right).
}	
\end{array}\]
This implies
\[\begin{array}{l}
\ds{\int_{\R} A(t)\,dx\lesssim \vert \partial_xu(t)\vert^2_{L^\infty(\R)}\vert \partial_x^ku(t)\vert^2_{L^2(\R)} +\vert \partial_xu(t)\vert_{L^\infty(\R)}\vert \partial_x^ku(t)\vert_{L^2(\R)}\vert \partial_x^{k+1}u(t)\vert_{L^2(\R)}  }	\\[10pt]
\ds{\quad \quad \quad \quad +\vert \partial_x^ku(t)\vert_{L^2(\R)} \sum_{j=2}^{k-1}\vert\partial_{x}^{j}u(t)\vert_{L^2(\R)}\vert \partial_{x}^{k+1-j}u(t)\vert_{L^\infty(\R)}}\\[10pt]
\ds{\quad \quad \quad \quad \quad \quad \quad \quad +\vert \partial_x^ku(t)\vert_{L^2(\R)}\sum_{i=2}^{k-1}\vert\partial_{x}^{i}u(t)\vert_{L^2(\R)}\sum_{j=1}^{k+1-i}\vert\partial_{x}^{j}u(t)\vert_{L^\infty(\R)}\vert\partial_{x}^{k+2-i-j}u(t)\vert_{L^\infty(\R)}.}
\end{array}\]
Therefore, in view of  \eqref{heat_flow_L_infty} and of the inductive hypothesis, we conclude
\[\begin{array}{l}
\ds{\int_0^T \int_{\R} A(t)\,dx\,dt\leq \frac 12\int_0^T \vert \partial_x^{k+1}u(t)\vert_{L^2(\R)}^2\,dt+c_T\big ( \vert u_0\vert_{\dot{H}^{2}(\R)}\big)\int_0^T \vert \partial_x^{k}u(t)\vert_{L^2(\R)}^2\,dt}\\[14pt]
\ds{\quad \quad \quad +c\,\int_0^T \vert \partial_x^{k}u(t)\vert_{L^2(\R)}^2\,dt+c_{k-1,T}\big ( \vert u_0\vert_{\dot{H}^{k}(\R)}\big)\sum_{i=2}^{k-1}\int_0^T \vert \partial_x^{i}u(t)\vert_{L^2(\R)}^2\,dt}\\[14pt]
\ds{\quad \quad   \leq \frac 12\int_0^T \vert \partial_x^{k+1}u(t)\vert_{L^2(\R)}^2\,dt+c_T\big ( \vert u_0\vert_{\dot{H}^{2}(\R)}\big)\int_0^T \vert \partial_x^{k}u(t)\vert_{L^2(\R)}^2\,dt+c_{k-1,T}\big ( \vert u_0\vert_{\dot{H}^{k}(\R)}\big).}	
\end{array}\]
 This, together with \eqref{fine15} and \eqref{fine16}, implies
\[\vert \partial_x^ku(t)\vert^2_{L^2(\R)}+\int_0^t \vert \partial_x^{k+1}u(s)\vert_{L^2(\R)}^2\,ds\leq  
c_T\big ( \vert u_0\vert_{\dot{H}^{2}(\R)}\big)\int_0^t \vert \varphi_k(s)\vert_{L^2(\R)}^2\,ds+c_{k,T}\big ( \vert u_0\vert_{\dot{H}^{k}(\R)}\big),\]
and the Gronwall's lemma gives 
\begin{equation}\label{fine21}\vert \partial_x^ku(t)\vert^2_{L^2(\R)}+\int_0^t \vert \partial_x^{k+1}u(s)\vert_{L^2(\R)}^2\,ds\leq c_{k,T}\big  ( \vert u_0\vert_{\dot{H}^{k}(\R)}\big).\end{equation}
Moreover, as
\[\partial_x^{k-1}\partial_t u(t)=\partial_x^{k+1}u(t)+\partial_x^{k-1}\left(\vert \partial_x u(t)\vert^2 u(t)\right),\]
due to \eqref{fine20} and \eqref{fine21}, we conclude
\[\int_0^T \vert \partial_x^{k-1}\partial_t u(t)\vert^2_{L^2(\R)}\lesssim_{\,T,k}c\big(u_0\vert_{\dot{H}^{k}(\R)}\big),\]
and \eqref{boundedness} follows for $k$.

\end{proof}

Next, we show that the solution of \eqref{heat_flow} depends continuously on its initial condition.

\begin{Lemma} \label{lemB3}
Let $u_{1},u_{2}$ be  solutions of \eqref{heat_flow}, with initial conditions $u_{1,0},u_{2,0}$, respectively, such that $u_{1,0}, u_{2,0}\in \bigcap_{k\geq1}\dot{H}^{k}(\R)\,\cap\, M$ 
  and $u_{1,0}-u_{2,0}\in L^2(\mathbb{R})$. Then  
		for any $k\geq 1$ we have
		\begin{equation}\label{continuity2}
\begin{array}{l}
\ds{\sup_{t\in [0,T]}\abs{u_{1}(t)-u_{2}(t)}_{H^{k}(\R)}^{2}+\int_{0}^{T}\abs{u_{1}(t)-u_{2}(t)}_{H^{k+1}(\R)}^{2}dt }\\
\vs
\ds{\quad \quad \quad \quad \quad \quad \leq c_{k,T}\big(\abs{u_{1,0}}_{\dot{H}^{k\vee 1}(\R)},\abs{u_{2,0}}_{\dot{H}^{k\vee 1}(\R)}\big)\abs{u_{1,0}-u_{2,0}}_{H^{k}(\R)}^{2}.
}	
\end{array}
					\end{equation}
		
	\end{Lemma}

\begin{proof} In what follows, we will show \eqref{continuity2} for the case when $k=1$ only. The cases when $k\geq 2$ can be treated by using a similar arguments and  induction.

Let $u_{1},u_{2}$ be solutions of \eqref{heat_flow}, with initial conditions $u_{1,0}, u_{2,0}$, respectively. Then, if we define $\rho:=u_{1}-u_{2}$, by proceeding as in Subsection \ref{subA1},  we have
	\begin{equation}
		\abs{\rho(t)}_{L^2(\R)}^{2}+\int_{0}^{t}\abs{\partial_{x}\rho(s)}_{L^2(\R)}^{2}dt \leq \abs{\rho(0)}_{L^2(\R)}^{2}+c\int_{0}^{t}\Big(\abs{\partial_{x}^{2}u_{1}(s)}_{L^2(\R)}^{2}+\abs{\partial_{x}^{2}u_{2}(s)}_{L^2(\R)}^{2}\Big)\abs{\rho(s)}_{L^2(\R)}^{2}ds.
	\end{equation}
	Therefore, the Gronwall's Lemma gives	\begin{equation}
		\sup_{t\in[0,T]}\abs{\rho(t)}_{L^2(\R)}^{2} \lesssim \abs{\rho(0)}_{L^2(\R)}^{2}\exp\Big(c\int_{0}^{T}\big(\abs{\partial_{x}^{2}u_{1}(t)}_{L^2(\R)}^{2}+\abs{\partial_{x}^{2}u_{2}(t)}_{L^2(\R)}^{2}\big)dt\Big).
	\end{equation}
	This, together with \eqref{case1_boundedness2}, implies 
	\begin{equation}\label{continuity1}
		\begin{array}{l}
		\ds{\sup_{t\in [0,T]}\abs{u_{1}(t)-u_{2}(t)}_{L^2(\mathbb{R})}^{2}+\int_{0}^{T}\abs{u_{1}(t)-u_{2}(t)}_{H^{1}(\R)}^{2}dt}\\
		\vs \ds{\quad \quad \quad \quad \quad \quad \quad \quad \leq c_T\big(\abs{u_{1,0}}_{\dot{H}^{1}(\R)},\abs{u_{2,0}}_{\dot{H}^{1}(\R)}\big)\abs{u_{1,0}-u_{2,0}}_{L^2(\mathbb{R})}^{2}.}	
		\end{array}
	\end{equation}
		
Next, we have
	\begin{equation}
		\begin{array}{ll}
			&\ds{ \partial_{t}\partial_{x}\rho(t) = \partial_{x}^{3}\rho(t) + 2(\partial_{x}\rho(t)\cdot \partial_{x}^{2}u_{1}(t))u_{1}(t)+2(\partial_{x}u_{2}(t)\cdot \partial_{x}^{2}\rho(t))u_{1}(t)   }\\
			\vs 
			&\ds{\quad\quad\quad\quad +\big(\abs{\partial_{x}u_{1}(t)}^{2}-\abs{\partial_{x}u_{2}(t)}^{2}\big)\partial_{x}u_{1}(t) +2(\partial_{x}u_{2}(t)\cdot \partial_{x}^{2}u_{2}(t))\rho(t)+\abs{\partial_{x}u_{2}(t)}^{2}\partial_{x}\rho(t) },
		\end{array}
	\end{equation}
	and then,  integrating by parts,
	\begin{equation}
		\begin{array}{ll}
			&\ds{ \frac{1}{2}\frac{d}{dt}\abs{\partial_{x}\rho(t)}_{L^2(\R)}^{2} = -\abs{\partial_{x}^{2}\rho(t)}_{L^2(\R)}^{2}+2\Inner{\partial_{x}\rho(t), (\partial_{x}\rho(t)\cdot \partial_{x}^{2}u_{1}(t))u_{1}(t)}_{L^2(\R)} }\\
			&\vs
			&\ds{\quad   +2\Inner{\partial_{x}\rho(t),(\partial_{x}u_{2}(t)\cdot \partial_{x}^{2}\rho(t))u_{1}(t) }_{L^2(\R)}    +\Inner{\partial_{x}\rho(t), \big(\abs{\partial_{x}u_{1}(t)}^{2}-\abs{\partial_{x}u_{2}(t)}^{2}\big)\partial_{x}u_{1}(t) }_{L^2(\R)} }\\
			\vs 
			&\ds{\quad\quad \quad\quad+2\Inner{\partial_{x}\rho(t), (\partial_{x}u_{2}(t)\cdot \partial_{x}^{2}u_{2}(t))\rho(t)}_{L^2(\R)}   +\Inner{\partial_{x}\rho(t), \abs{\partial_{x}u_{2}(t)}^{2}\partial_{x}\rho(t) }_{L^2(\R)} }.
		\end{array}
	\end{equation}
	For any $\delta>0$, 
	\begin{equation}\begin{array}{l}
		\ds{\big\lvert \Inner{\partial_{x}\rho, (\partial_{x}\rho\cdot \partial_{x}^{2}u_{1})u_{1}}_{L^2(\R)}\big\rvert \leq \abs{\partial_{x}^{2}u_{1}}_{L^2(\R)}\abs{\partial_{x}\rho}_{L^{\infty}}\abs{\partial_{x}\rho}_{L^2(\R)}}\\
		\vs 
		\ds{\quad \quad \quad \quad \quad\quad \quad \quad \quad \quad\quad \quad \quad \quad \quad \leq \delta\abs{\partial_{x}^{2}\rho}_{L^2(\R)}^{2}+c(\delta)\abs{\partial_{x}^{2}u_{1}}_{L^2(\R)}^{4/3}\abs{\partial_{x}\rho}_{L^2(\R)}^{2},}
		\end{array}
	\end{equation}
	and
	\begin{equation}\begin{array}{l}
		\ds{
		\big\lvert \Inner{\partial_{x}\rho, (\partial_{x}u_{2}\cdot \partial_{x}^{2}\rho)u_{1} }_{L^2(\R)}\big\rvert \leq \abs{\partial_{x}u_{2}}_{L^{\infty}}\abs{\partial_{x}\rho}_{L^2(\R)}\abs{\partial_{x}^{2}\rho}_{L^2(\R)}}\\
		\vs 
		\ds{\quad \quad \quad \quad \quad\quad \quad \quad \quad \quad\leq \delta\abs{\partial_{x}^{2}\rho}_{L^2(\R)}^{2}+c(\delta)\abs{\partial_{x}^{2}u_{2}}_{L^2(\R)}\abs{\partial_{x}u_{2}}_{L^2(\R)}\abs{\partial_{x}\rho}_{L^2(\R)}^{2}.}
		\end{array}
	\end{equation}
	Moreover, we have
	\begin{equation}
		\begin{array}{ll}
			&\ds{ \big\lvert \Inner{\partial_{x}\rho, (\abs{\partial_{x}u_{1}}^{2}- \abs{\partial_{x}u_{2}}^{2} )\partial_{x}u_{1} }_{L^2(\R)}  \big\rvert \lesssim\abs{\partial_{x}\rho}_{L^2(\R)}^{2}\abs{\partial_{x}u_{1}}_{L^{\infty}}\big(\abs{\partial_{x}u_{1}}_{L^{\infty}}+\abs{\partial_{x}u_{2}}_{L^{\infty}}\big)  }\\
			\vs 
			&\ds{\quad\quad\quad\quad \lesssim\big(\abs{\partial_{x}u_{1}}_{L^{\infty}}^{2}+\abs{\partial_{x}u_{2}}_{L^{\infty}}^{2}\big)\abs{\partial_{x}\rho}_{L^2(\R)}^{2}}\\
			\vs 
			&\ds{\quad\quad\quad\quad \quad\quad\quad\quad \lesssim\big(\abs{\partial_{x}^{2}u_{1}}_{L^2(\R)}\abs{\partial_{x}u_{1}}_{L^2(\R)}+\abs{\partial_{x}^{2}u_{2}}_{L^2(\R)}\abs{\partial_{x}u_{2}}_{L^2(\R)}\big)\abs{\partial_{x}\rho}_{L^2(\R)}^{2}  },
		\end{array}
	\end{equation}
and
	\begin{equation}
		\begin{array}{ll}
		&\ds{   \big\lvert \Inner{\partial_{x}\rho, (\partial_{x}u_{2}\cdot \partial_{x}^{2}u_{2})\rho}_{L^2(\R)} \big\rvert \leq \abs{\partial_{x}\rho}_{L^2(\R)}\abs{\partial_{x}^{2}u_{2}}_{L^2(\R)}\abs{\partial_{x}u_{2}}_{L^{\infty}}\abs{\rho}_{L^{\infty}}}\\
		\vs 
		&\ds{\quad \quad \lesssim \abs{\rho}_{L^2(\R)}^{1/2}\abs{\partial_{x}\rho}_{L^2(\R)}^{3/2}\abs{\partial_{x}^{2}u_{2}}_{L^2(\R)}^{3/2}\abs{\partial_{x}u_{2}}_{L^2(\R)}^{1/2} \lesssim\abs{\rho}_{L^2(\R)}^{2} +\abs{\partial_{x}^{2}u_{2}}_{L^2(\R)}^{2}\abs{\partial_{x}u_{2}}_{L^2(\R)}^{2/3}\abs{\partial_{x}\rho}_{L^2(\R)}^{2} },
		\end{array}
	\end{equation}
	and
	\begin{equation}
		\big\lvert \Inner{\partial_{x}\rho,\abs{\partial_{x}u_{2}}^{2}\partial_{x}\rho}_{L^2(\R)}\big\rvert \leq \abs{\partial_{x}u_{2}}_{L^{\infty}}^{2}\abs{\partial_{x}\rho}_{L^2(\R)}^{2} \lesssim \abs{\partial_{x}^{2}u_{2}}_{L^2(\R)}\abs{\partial_{x}u_{2}}_{L^2(\R)}\abs{\partial_{x}\rho}_{L^2(\R)}^{2}.
	\end{equation}
	Therefore, if we put all this together and choose $\delta>0$ sufficiently small, 
	we obtain 
	\begin{equation}
		\begin{array}{ll}
		&\ds{  \abs{\partial_{x}\rho(t)}_{L^2(\R)}^{2}+\int_{0}^{t}\abs{\partial_{x}^{2}\rho(s)}_{L^2(\R)}^{2}ds }\\
		\vs 
		&\ds{\lesssim \abs{\partial_x \rho(0)}_{L^2(\R)}^{2}+\int_{0}^{t}\abs{\rho(s)}_{L^2(\R)}^{2}ds + \int_{0}^{t}\big(1+\abs{\partial_{x}u_{1}(s)}_{L^2(\R)}+\abs{\partial_{x}u_{2}(s)}_{L^2(\R)}\big) }\\
		\vs 
		&\ds{  + \int_{0}^{t}\big(1+\abs{\partial_{x}u_{1}(s)}_{L^2(\R)}+\abs{\partial_{x}u_{2}(s)}_{L^2(\R)}\big)\big(1+\abs{\partial_{x}^{2}u_{1}(s)}_{L^2(\R)}^{2}+\abs{\partial_{x}^{2}u_{2}(s)}_{L^2(\R)}^{2}\big)\abs{\partial_{x}\rho(s)}_{L^2(\R)}^{2}ds  }.
		\end{array}
	\end{equation}
	Finally,  by applying Gronwall's Lemma and taking into account of \eqref{continuity1}, we complete the proof of \eqref{continuity2}.
\end{proof}

Now, if we fix $u_0\in \dot{H}^{k}(\R)\cap M$, for some $k\geq 1$, due to Lemma \ref{lemB1} we have that there exists a sequence $\{u_{0,n}\}_{n\geq 1}\subset \bigcap_{k\in\mathbb{N}} \dot{H}^{k}(\R)\cap M$ such that $u_{0,n}$ converges to $u_0 \in\,\dot{H}^k(\mathbb{R})$ and $u_{0,n}-u \in\,L^2(\R)$ converges to $0$ in $L^2(\mathbb{R})$. In particular, for every $n,m \in\,\mathbb{N}$, we have that $u_{0,n}-u_{0,m} \in\,H^k(\mathbb{R})$ and 
\begin{equation}\label{sm196}
\lim_{n, m\to\infty }\vert u_{0,n}-u_{0,m}\vert_{H^k(\mathbb{R})}=0.	
\end{equation}
According to Lemma \ref{lemB3}, if we denote by $u_n\in \bigcap_{k\geq 1} C^{\infty}((0,T);\dot{H}^{k}(\R))$  the unique solution of \eqref{heat_flow}, with initial condition $u_{0,n}$, we have that the sequence 
\[\{u_{n}\}_{n \in\,\mathbb{N}}\subset L^\infty(0,T;C_b(\mathbb{R}))\,\cap\,L^{\infty}(0,T;\dot{H}^{k}(\R))\,\cap\, L^{2}(0,T;\dot{H}^{k+1}(\R)).\]
is Cauchy. In particular, it converges in  the space above and its limit $u$ is a solution of \eqref{heat_flow}. Moreover, due to Lemma \ref{lemA4}, the following result holds.

\begin{Theorem}
	For every  $T>0$, $k\in\mathbb{N}$ and $u_0\in \dot{H}^{k}(\R)\cap M$, equation \eqref{heat_flow} admits  a unique solution $u \in\, L^{\infty}(0,T;\dot{H}^{k}(\R))\cap L^{2}(0,T;\dot{H}^{k+1}(\R))$, with $\partial_t u \in\,L^2(0,T;H^{k-1})$ and 
	\begin{equation} \label{fine30}
		\sup_{t\in [0,T]}\abs{u(t)}_{\dot{H}^{k}(\R)}^{2} + \int_{0}^{T}\abs{u(t)}_{\dot{H}^{k+1}(\R)}^{2}dt+\int_{0}^{T}\abs{\partial_{t}u(t)}_{H^{k-1}(\R)}^{2}dt \leq c_{k,T}(\abs{u_0}_{\dot{H}^{k}(\R)} \big).
	\end{equation}
	Moreover, in case $k=1$,  we have the following identity
	\begin{equation}\label{case1_boundedness1-bis}
		\abs{\partial_{x}u(t)}_{L^2(\R)}^{2}+2\gamma_0\int_{0}^{t}\abs{\partial_t u(s)}_{L^2(\R)}^{2}ds = \abs{D u_{0          }}_{L^2(\R)}^{2}.
	\end{equation}
\end{Theorem}


\medskip\

{\bf Acknowledgements:} The first-named author would like to thank Manoussos Grillakis, Yu Gu, and Andrew Lawrie for several helpful conversations regarding various aspects of this paper. In addition, both authors are grateful to Zdzislaw Brze\'zniak for the stimulating discussions during his visit to the Department of Mathematics at the University of Maryland, which provided valuable insights relevant to this work.



\end{document}